\documentclass[11pt]{amsart}
\usepackage[left=10mm,right=10mm]{geometry}
\usepackage[english]{babel}
\usepackage[T1]{fontenc}
\usepackage[utf8]{inputenc}
\usepackage{amsthm}
\usepackage{amsmath}
\title{Fluctuations of the local times of the self-repelling random walk with directed edges}
\author{Laure Mar{\^e}ch\'e}
\email{laure.mareche@math.unistra.fr}
\address{Institut de Recherche Mathématique Avancée, 
UMR 7501 Université de Strasbourg et CNRS, 
7 rue René-Descartes, 67000 Strasbourg, France}
\theoremstyle{plain}
\newtheorem{theorem}{Theorem}
\newtheorem*{theoremstar}{Theorem}

\newtheorem{lemma}[theorem]{Lemma}
\newtheorem{proposition}[theorem]{Proposition}
\newtheorem*{propositionstar}{Proposition}

\newtheorem{claim}[theorem]{Claim}
\theoremstyle{definition}

\theoremstyle{remark}
\newtheorem{remark}[theorem]{Remark}
\usepackage{dsfont}
\usepackage{hyperref}
\usepackage{xcolor}
\usepackage{tikz}
\usetikzlibrary{decorations}
\usetikzlibrary{decorations.pathreplacing}

\begin{document}

\begin{abstract}
In 2008, Tóth and Vető defined the self-repelling random walk with directed edges as a \emph{non-Markovian} random walk on $\mathds{Z}$: in this model, the probability that the walk moves from a point of $\mathds{Z}$ to a given neighbor depends on the number of previous crossings of the directed edge from the initial point to the target, called the local time of the edge. They found this model had a very peculiar behavior, as the process formed by the local times of all the edges, evaluated at a stopping time of a certain type and suitably renormalized, converges to a deterministic process, instead of a random one as in similar models. In this work, we study the fluctuations of the local times process around its deterministic limit, about which nothing was previously known. We prove that these fluctuations converge in the Skorohod $M_1$ topology, as well as in the uniform topology away from the discontinuities of the limit, but \emph{not} in the most classical Skorohod topology. We also prove the convergence of the fluctuations of the aforementioned stopping times.
\end{abstract}

\maketitle


\section{Introduction and results}

\subsection{Self-interacting random walks} 

The study of self-interacting random walks began in 1983 in an article of Amit et al. \cite{Amit_et_al1983}. Before \cite{Amit_et_al1983}, the expression ``self-avoiding random walk'' referred to paths on graphs that do not intersect themselves. However, these are not easy to construct step by step, hence one would consider the set of all possible paths of a given length. Since one does not follow a single path as it grows with time, it is not really a random walk model. In order to work with an actual random walk model with a self-avoiding behavior, the authors of \cite{Amit_et_al1983} introduced the \emph{``true'' self-avoiding random walk}. It is a random walk on $\mathds{Z}^d$ for which, at each step, the position of the process at the next step is chosen randomly among the neighbors of the current position depending on the number of the previous visits to said neighbors, with lower probabilities for those that have been visited the most. This process is a random walk in the sense that it is constructed step by step, but contrary to most random walks in the literature, it is \emph{non-Markovian}: at each step, the law of the next step depends on the whole past of the process.

It turns out that the ``true'' self-avoiding random walk is hard to study. This led to the introduction by Tóth \cite{Toth1994,Toth1995,Toth1996} of non-Markovian random walks \emph{with bond repulsion}, for which the probability to go from one site to another, instead of depending of the number of previous visits to the target, depends on the number of previous crossings of the undirected edge between the two sites, which is called the \emph{local time} of the edge, with lower probabilities for the edges that were crossed the most in the past. These walks are much easier to study, at least on $\mathds{Z}$, because one can apply the \emph{Ray-Knight approach} to them. This approach was introduced by Ray and Knight in \cite{Ray1963,Knight_1963}, and used for the first time for non-Markovian random walks by Tóth in \cite{Toth1994,Toth1995,Toth1996}. Since then, it was applied to many other non-Markovian random walks, such as a continuous-time version of the ``true'' self-avoiding random walk in \cite{Toth_et_al2011}, \emph{edge-reinforced random walks} (see the corresponding part of the review \cite{Pemantle2007} and references therein) and \emph{excited random walks} (see \cite{Kosygina_et_al2022} and references therein). The Ray-Knight approach works as follows: though the random walk itself is not Markovian, if we stop it when the local time at a given edge has reached a certain threshold, then the local times on the edges will form a Markov chain, which allows their analysis. Thanks to this approach, Tóth was able to prove scaling limits for the local times process for many different random walks with bond repulsion in his works \cite{Toth1994,Toth1995,Toth1996}. The law of the limit depends on the random walk model, but it is always a random process\footnote{The model studied by Tóth in \cite{Toth1997} has a deterministic limit, but it is not a random walk with bond repulsion, as it is \emph{self-attracting}: the more an edge was crossed in the past, the more likely it is to be crossed in the future.}.

\subsection{The self-repelling random walk with directed edges}

In 2008, Tóth and Vető \cite{Toth_et_al2008} introduced a process seemingly very similar to the aforementioned random walks with bond repulsion, in which the probability to go from one site to another depends on the number of crossings of the \emph{directed} edge between them instead of the crossings of the undirected edge. This process, called \emph{self-repelling random walk with directed edges}, is a nearest-neighbor random walk on $\mathds{Z}$ defined as follows. For any set $A$, we denote by $|A|$ the cardinal of $A$. Let $w : \mathds{Z} \mapsto (0,+\infty)$ be a non-decreasing and non-constant function. We will denote the walk by $(X_n)_{n\in\mathds{N}}$. We set $X_0=0$, and for any $n\in\mathds{N}$, $i\in\mathds{Z}$, we denote $\ell^\pm(n,i)=|\{0 \leq m \leq n-1 \,|\, (X_m,X_{m+1})=(i,i\pm1)\}|$ the number of crossings of the directed edge $(i,i\pm1)$ before time $n$, that is the \emph{local time} of the directed edge at time $n$. Then 
\[
\mathds{P}(X_{n+1}=X_n\pm1)=\frac{w(\pm(\ell^-(n,X_n)-\ell^+(n,X_n)))}{w(\ell^+(n,X_n)-\ell^-(n,X_n))+w(\ell^-(n,X_n)-\ell^+(n,X_n))}. 
\]

Using the local time of directed edges instead of that of undirected edges may seem like a very small change in the definition of the process, but the behavior of the self-repelling random walk with directed edges is actually very different from that of classical random walks with bond repulsion. Indeed, Tóth and Vető \cite{Toth_et_al2008} were able to prove that the local times process has a \emph{deterministic} scaling limit, which is in sharp contrast with the random limit processes obtained for the random walks with bond repulsion on undirected edges \cite{Toth1994,Toth1995,Toth1996} and even for the simple random walk~\cite{Knight_1963}. 

The result of \cite{Toth_et_al2008} is as follows. For any $a\in\mathds{R}$, we denote $a_+=\max(a,0)$. If for any $n\in\mathds{N}$, $i\in\mathds{Z}$, we denote by $T_{n,i}^\pm$ the stopping time defined by $T_{n,i}^\pm=\min\{m\in\mathds{N}\,|\,\ell^\pm(m,i) = n\}$, then $T_{n,i}^\pm$ is almost-surely finite by Proposition 1 of \cite{Toth_et_al2008} and we have the following.

\begin{theoremstar}[Theorem 1 of \cite{Toth_et_al2008}]
 For any $\theta>0$, $x\in \mathds{R}$, then $\sup_{y \in \mathds{R}}|\frac{1}{N}\ell^+(T_{\lfloor N \theta\rfloor,\lfloor N x\rfloor}^\pm,\lfloor N y\rfloor)-(\frac{|x|-|y|}{2}+\theta)_+|$ converges in probability to 0 when $N$ tends to $+\infty$.
\end{theoremstar}

Thus the local times process of the self-repelling random walk with directed edges admits the deterministic scaling limit $: y \mapsto (\frac{|x|-|y|}{2}+\theta)_+$, which has the shape of a triangle. This also implies the following convergence result to a deterministic limit for the $T_{\lfloor N \theta\rfloor,\lfloor N x\rfloor}^\pm$.

\begin{propositionstar}[Corollary 1 of \cite{Toth_et_al2008}]
 For any $\theta>0$, $x\in \mathds{R}$, then $\frac{1}{N^2}T_{\lfloor N \theta\rfloor,\lfloor N x\rfloor}^\pm$ converges in probability to $(|x|+2\theta)^2$ when $N$ tends to $+\infty$.
\end{propositionstar}

The deterministic character of these limits makes the behavior of the self-repelling random walk with directed edges very unusual, hence worthy of study. In particular, it is natural to consider the possible fluctuations of the local times process and of the $T_{\lfloor N \theta\rfloor,\lfloor N x\rfloor}^\pm$ around their deterministic limits. However, prior to this paper, nothing was known about these fluctuations. In this work, we prove convergence in distribution of the fluctuations of the local times process and of the $T_{\lfloor N \theta\rfloor,\lfloor N x\rfloor}^\pm$. It happens that the limit of the fluctuations of the local times process is discontinuous, hence before stating the results, we have to be careful of the topology in which it may converge.

\subsection{Topologies for the convergence of the local times process}

For any interval $I \subset R$, let $D I$ be the space of \emph{càdlàg} functions on $I$, that is the set of functions $ : I \mapsto \mathds{R}$ that are right-continuous and have left limits everywhere in $I$. For any function $Z : I \mapsto \mathds{R}$, we denote by $\|Z\|_\infty=\sup_{y\in I}|Z(y)|$ the uniform norm of $Z$ on $I$. The uniform norm on $I$ gives a topology on $DI$, but it is often too strong to deal with discontinuous functions.

For discontinuous càdlàg functions, the most widely used topology is the \emph{Skorohod $J_1$ topology}, introduced by Skorohod in \cite{Skorohod1956} (see chapter VI of \cite{PollardJ1} for a course), which is often called ``the'' Skorohod topology. Intuitively, two functions are close in this topology if they are close for the uniform norm after allowing some small perturbation of time. Rigorously, for $a<b$ in $\mathds{R}$ the Skorohod $J_1$ topology on $D[a,b]$ is defined as follows. We call $\Lambda_{a,b}$ the set of functions $\lambda : [a,b] \mapsto [a,b]$ that are bijective, strictly increasing and continuous (they correspond to the possible perturbations of time), and we denote by $\mathrm{Id}_{a,b} : [a,b] \mapsto [a,b]$ the identity map, defined by $\mathrm{Id}_{a,b}(y)=y$ for all $y\in[a,b]$. The Skorohod $J_1$ topology on $D[a,b]$ is defined through the following metric: for any $Z_1,Z_2 \in D[a,b]$, we set $d_{J_1,a,b}(Z_1,Z_2)=\inf_{\lambda\in\Lambda_{a,b}}\max(\|Z_1\circ\lambda-Z_2\|_\infty,\|\lambda-\mathrm{Id}_{a,b}\|_\infty)$. It can be proven rather easily that this is indeed a metric. We can then define the Skorohod $J_1$ topology in $D(-\infty,\infty)$ with the following metric: if for any sets $A_1 \subset A_2$ and $A_3$ and any function $Z:A_2 \mapsto A_3$, we denote $Z|_{A_1}$ the restriction of $Z$ to $A_1$, then for $Z_1,Z_2 \in D(-\infty,\infty)$, we set $d_{J_1}(Z_1,Z_2)=\int_0^{+\infty}e^{-a}(d_{J_1,-a,a}(Z_1|_{[-a,a]},Z_2|_{[-a,a]}) \wedge 1 )\mathrm{d}a$. The Skorohod $J_1$ topology is widely used to study the convergence of càdlàg functions. However, when the limit function has a jump, which will be the case here, convergence in the Skorohod $J_1$ topology requires the converging functions to have a single big jump approximating the jump of the limit process. To account for other cases, like having the jump of the limit functions approximated by several smaller jumps in quick succession or by a very steep continuous slope, one has to use a less restrictive topology, like the \emph{Skorohod $M_1$ topology}. 

The Skorohod $M_1$ topology was also introduced by Skorohod in \cite{Skorohod1956} (see Section 3.3 of \cite{Whitt_Skorohod_topologies} for an overview). For any $a<b$ in $\mathds{R}$, the Skorohod $M_1$ distance on $D[a,b]$ is defined as follows: the distance between two functions will be roughly ``the distance between the completed graphs of the functions''. More rigorously, if $Z\in D[a,b]$, we denote $Z(a^-)=Z(a)$ and for any $y \in (a,b]$, we denote $Z(y^-)=\lim_{y' \to y, y'<y}Z(y')$. Then the \emph{completed graph of $Z$} is $\Gamma_Z=\{(y,z)\,|\,y\in[a,b],\exists\,\varepsilon\in[0,1]$ so that $z=\varepsilon Z(y^-)+(1-\varepsilon)Z(y)\}$. To express the ``distance between two such completed graphs'', we need to define the \emph{parametric representations} of $\Gamma_Z$ (by abuse of notation, we will often write ``the parametric representations of $Z$''). We define an order on $\Gamma_Z$ as follows: for $(y_1,z_1),(y_2,z_2) \in \Gamma_Z$, we have $(y_1,z_1) \leq (y_2,z_2)$ when $y_1 < y_2$ or when $y_1=y_2$ and $|Z(y_1^-)-z_1|\leq|Z(y_1^-)-z_2|$. A \emph{parametric representation} of $\Gamma_Z$ is a continuous, surjective function $(u,r):[0,1] \mapsto \Gamma_Z$ that is non-decreasing with respect to this order, thus intuitively, when $t$ goes from 0 to 1, $(u(t),r(t))$ ``travels through the completed graph of $Z$ from its beginning to its end''. A parametric representation of $Z$ always exists (see Remark 12.3.3 in \cite{Whitt_Skorohod_topologies}). For $Z_1,Z_2 \in D[a,b]$, the Skorohod $M_1$ distance between $Z_1$ and $Z_2$, denoted by $d_{M_1,a,b}(Z_1,Z_2)$, is $\inf\{\max(\|u_1-u_2\|_\infty,\|r_1-r_2\|_\infty)\}$ where the infimum is on the parametric representations $(u_1,r_1)$ of $Z_1$ and $(u_2,r_2)$ of $Z_2$. It can be proven that this indeed gives a metric (see Theorem 12.3.1 of \cite{Whitt_Skorohod_topologies}), and this metric defines the Skorohod $M_1$ topology on $D[a,b]$. For any $a>0$, we will denote $d_{M_1,-a,a}$ by $d_{M_1,a}$ for short. We can now define the Skorohod $M_1$ topology in $D(-\infty,\infty)$ through the following metric: for $Z_1,Z_2 \in D(-\infty,\infty)$, we set $d_{M_1}(Z_1,Z_2)=\int_0^{+\infty}e^{-a}(d_{M_1,a}(Z_1|_{[-a,a]},Z_2|_{[-a,a]}) \wedge 1 )\mathrm{d}a$. It can be seen that the Skorohod $M_1$ topology is weaker than the Skorohod $J_1$ topology (see Theorem 12.3.2 of \cite{Whitt_Skorohod_topologies}), thus less restrictive. Indeed, since the distance between two functions is roughly ``the distance between the completed graphs of the functions'', the Skorohod $M_1$ topology will allow a function with a jump to be the limit of functions with steep slopes or with several smaller jumps. For this reason, the Skorohod $M_1$ topology is often more adapted when considering convergence to a discontinuous function.

\subsection{Results}

We are now ready to state our results on the convergence of the fluctuations of the local times process. For any $\theta>0$, $x\in \mathds{R}$, $\iota\in\{+,-\}$, for any $N \in \mathds{N}^*$, we define functions $Y_N^-,Y_N^+$ as follows: for any $y\in\mathds{R}$, we set 
\[
Y_N^\pm(y)=\frac{1}{\sqrt{N}}\left(\ell^\pm(T_{\lfloor N \theta\rfloor,\lfloor N x\rfloor}^\iota,\lfloor N y\rfloor)-N\left(\frac{|x|-|y|}{2}+\theta\right)_+\right).
\]
$Y_N^\pm$ actually depends on $\iota$, but we do not write this dependency in the notation to make it lighter. Moreover, $(B_y^x)_{y\in\mathds{R}}$ will denote a two-sided Brownian motion with $B_x^x=0$ and variance $\mathrm{Var}(\rho_-)$, where $\rho_-$ is the distribution on $\mathds{Z}$ defined later in \eqref{eq_rho}. We proved the following convergence for the fluctuations of the local times process of the self-repelling random walk with directed edges.

\begin{theorem}\label{thm_main}
 For any $\theta>0$, $x\in \mathds{R}$, $\iota\in\{+,-\}$, the process $Y_N^\pm$ converges in distribution to $(B_y^x \mathds{1}_{\{y\in[-|x|-2\theta,|x|+2\theta)\}})_{y\in\mathds{R}}$ in the Skorohod $M_1$ topology on $D(-\infty,+\infty)$ when $N$ tends to $+\infty$. 
\end{theorem}

Therefore the fluctuations of the local times process have a diffusive limit behavior. However, it is necessary to use the Skorohod $M_1$ topology here, as the following result states the convergence does not occur in the stronger Skorohod $J_1$ topology. 

\begin{proposition}\label{prop_no_J1_conv}
 For any $\theta>0$, $x\in \mathds{R}$, $\iota\in\{+,-\}$, the process $Y_N^\pm$ does not converge in distribution in the Skorohod $J_1$ topology on $D(-\infty,+\infty)$ when $N$ tends to $+\infty$. 
\end{proposition}

We stress the fact that the use of the Skorohod $M_1$ topology is only required to deal with the discontinuities of the limit process at $-|x|-2\theta$ and $|x|+2\theta$. Indeed, if we consider the convergence of the process on an interval that does not include $-|x|-2\theta$ or $|x|+2\theta$, it converges in the much stronger topology given by the uniform norm, which is the following result.

\begin{proposition}\label{prop_uniform_conv}
 For any $\theta>0$, $x\in \mathds{R}$, $\iota\in\{+,-\}$, for any closed interval $I\in\mathds{R}$ that does not contain $-|x|-2\theta$ or $|x|+2\theta$, the process $(Y_N^\pm(y))_{y\in I}$ converges in distribution to $(B_y^x \mathds{1}_{\{y\in[-|x|-2\theta,|x|+2\theta)\}})_{y\in I}$ in the topology on $D I$ given by the uniform norm when $N$ tends to $+\infty$.
\end{proposition}

Finally, we also proved the convergence of the fluctuations of $T_{\lfloor N \theta\rfloor,\lfloor N x\rfloor}^\pm$. For any $\sigma^2>0$, we denote by $\mathcal{N}(0,\sigma^2)$ the Gaussian distribution with mean 0 and variance $\sigma^2$, and we recall that $\rho_-$ will be defined in \eqref{eq_rho}. We then have the following.

\begin{proposition}\label{prop_fluctuations_T}
 For any $\theta>0$, $x\in \mathds{R}$, $\iota\in\{+,-\}$, we have that $\frac{1}{N^{3/2}}( T_{\lfloor N \theta\rfloor,\lfloor N x\rfloor}^\iota-N^2(|x|+2\theta)^2)$ converges in distribution to $\mathcal{N}(0,\frac{32}{3}\mathrm{Var}(\rho_-)((|x|+\theta)^3+\theta^3))$ when $N$ tends to $+\infty$. 
\end{proposition}

\begin{remark}
 Instead of studying the fluctuations of $\ell^\pm(T_{\lfloor N \theta\rfloor,\lfloor N x\rfloor}^\iota,.)$, it would seem more natural to consider those of $\ell^\pm(N^2,.)$. However, the Ray-Knight arguments that allow to study $\ell^\pm(T_{\lfloor N \theta\rfloor,\lfloor N x\rfloor}^\iota,.)$ completely break down for $\ell^\pm(N^2,.)$, and it is not even clear whether these two processes should have the same behavior.
\end{remark}

\begin{remark}
 Besides the article of Tóth and Vető \cite{Toth_et_al2008} that introduced the self-repelling random walk with directed edges, there have been few other works on this model. These works were motivated by another important question, that of the existence of a scaling limit for $(X_n)_{n \in \mathds{N}}$, which means the convergence in distribution of the process $(\frac{1}{N^\alpha}X_{\lfloor Nt\rfloor})_{t \geq 0}$ for some $\alpha$. Obtaining such a scaling limit for the trajectory of the random walk is harder that obtaining scaling limits for the local times. Indeed, for the random walks with bond repulsion with undirected edges introduced by Tóth in \cite{Toth1994,Toth1995,Toth1996}, the scaling limits for the local times are known since the introduction of the models, but the scaling limits for the trajectories are not. Some results were proven by Kosygina, Mountford and Peterson in \cite{Kosygina_et_al2022bis}, but they do not cover all models. For the self-repelling random walk with directed edges, the behavior of the scaling limit of the trajectory turns out to be surprising. Indeed, Mountford, Pimentel, and Valle proved in \cite{Mountford_et_al2014} that $\frac{1}{\sqrt{N}}X_N$ converges in distribution, but Mountford and the author showed in \cite{Mareche_Mountford_2023} that $(\frac{1}{\sqrt{N}}X_{\lfloor Nt \rfloor})_{t \geq 0}$ does \emph{not} converge in distribution, and that the trajectories of the walk satisfy a more complex limit theorem, of a new kind. 
\end{remark}

\subsection{Proof ideas}

We begin by explaining why the limit of the local times process $Y_N^\pm$ is $(B_y^x \mathds{1}_{\{y\in[-|x|-2\theta,|x|+2\theta)\}})_{y\in\mathds{R}}$ and the ideas behind the proofs of Theorem \ref{thm_main} and Proposition \ref{prop_uniform_conv}. To show the convergence of the local times process, we use a Ray-Knight argument, that is we notice that $(\ell^-(T_{\lfloor N \theta\rfloor,\lfloor N x\rfloor}^\iota,i))_i$ is a Markov chain. Moreover, as long as $\ell^-(T_{\lfloor N \theta\rfloor,\lfloor N x\rfloor}^\iota,i)$ is not too low, the $\ell^-(T_{\lfloor N \theta\rfloor,\lfloor N x\rfloor}^\iota,i+1)-\ell^-(T_{\lfloor N \theta\rfloor,\lfloor N x\rfloor}^\iota,i)$ will roughly be i.i.d. random variables in the sense that they can be coupled with i.i.d. random variables with a high probability to be equal to them. This coupling was already used in \cite{Toth_et_al2008} to prove the convergence of $\frac{1}{N}\ell^+(T_{\lfloor N \theta\rfloor,\lfloor N x\rfloor}^\pm,\lfloor N y\rfloor)$ to its deterministic limit (for a given $y$, the coupling makes this convergence a law of large numbers). However, when $\ell^-(T_{\lfloor N \theta\rfloor,\lfloor N x\rfloor}^\iota,\lfloor N y\rfloor)$ is too low, the coupling fails and the $\ell^-(T_{\lfloor N \theta\rfloor,\lfloor N x\rfloor}^\iota,\lfloor N y\rfloor+1)-\ell^-(T_{\lfloor N \theta\rfloor,\lfloor N x\rfloor}^\iota,\lfloor N y\rfloor)$ are no longer i.i.d. We have to prove that this occurs only around $|x|+2\theta$ and $-|x|-2\theta$, and most of our work is dealing with what happens there. To show it occurs only around $|x|+2\theta$ and $-|x|-2\theta$, we control the amplitude of the fluctuations to prove the local times are close to their deterministic limit. This limit is large inside $(-|x|-2\theta,|x|+2\theta)$, so we can use the coupling inside this interval, thus the $\ell^-(T_{\lfloor N \theta\rfloor,\lfloor N x\rfloor}^\iota,\lfloor N y\rfloor+1)-\ell^-(T_{\lfloor N \theta\rfloor,\lfloor N x\rfloor}^\iota,\lfloor N y\rfloor)$ are roughly i.i.d. there, hence the fluctuations will converge to a Brownian motion by Donsker’s Invariance Principle. When we are close to $|x|+2\theta$ (the same reasoning works for $-|x|-2\theta$) the deterministic limit will be small hence the local times too, and tools of \cite{Toth_et_al2008} allow to prove that they reach 0 quickly. Once they reach 0, we notice that for $y \geq |x|+2\theta$, if $\ell^-(T_{\lfloor N \theta\rfloor,\lfloor N x\rfloor}^\iota,\lfloor N y\rfloor)=0$, the walk $X$ did not go from $\lfloor N y\rfloor$ to $\lfloor N y\rfloor+1$ before time $T_{\lfloor N \theta\rfloor,\lfloor N x\rfloor}^\iota$, so it did not go to $\lfloor N y\rfloor+1$ before time $T_{\lfloor N \theta\rfloor,\lfloor N x\rfloor}^\iota$, hence $\ell^-(T_{\lfloor N \theta\rfloor,\lfloor N x\rfloor}^\iota,j)=0$ for any $j \geq \lfloor N y\rfloor$. Therefore, once the local times process reaches 0, it stays there. Consequently, we expect $\ell^-(T_{\lfloor N \theta\rfloor,\lfloor N x\rfloor}^\iota,\lfloor N y\rfloor)$ to be 0 when $y > |x|+2\theta$, and thus to have no fluctuations when $y > |x|+2\theta$, and similarly when $y < -|x|-2\theta$. This is why our limit is $(B_y^x \mathds{1}_{\{y\in[-|x|-2\theta,|x|+2\theta)\}})_{y\in\mathds{R}}$. Since Proposition \ref{prop_uniform_conv} only describes convergence away from $-|x|-2\theta$ and $|x|+2\theta$, the previous arguments are enough to prove it. To prove the convergence in the Skorohod $M_1$ topology on $D(-\infty,+\infty)$ stated in Theorem \ref{thm_main}, we need to handle what happens around $-|x|-2\theta$ and $|x|+2\theta$ with more precision. We first have to bound the difference between the local times and the i.i.d. random variables of the coupling even where the coupling fails. Afterwards comes the most important part of the paper: defining parametric representations of $Y_N^\pm$ and of the sum of the i.i.d. random variables of the coupling, properly renormalized and set to 0 outside of $[-|x|-2\theta,|x|+2\theta)$, and then proving that they are close to each other. That allows to prove $Y_N^\pm$ is close in the Skorohod $M_1$ distance to a process that will converge in distribution to $(B_y^x \mathds{1}_{\{y\in[-|x|-2\theta,|x|+2\theta)\}})_{y\in\mathds{R}}$ in the Skorohod $M_1$ topology and to complete the proof of Theorem \ref{thm_main}. 

To prove Proposition \ref{prop_no_J1_conv}, that is that $Y_N^\pm$ does not converge in the $J_1$ topology, we first notice that since the $J_1$ topology is stronger than the $M_1$ topology, if $Y_N^\pm$ did converge in the $J_1$ topology its limit would be $(B_y^x \mathds{1}_{\{y\in[-|x|-2\theta,|x|+2\theta)\}})_{y\in\mathds{R}}$. However, it is not possible, as $(B_y^x \mathds{1}_{\{y\in[-|x|-2\theta,|x|+2\theta)\}})_{y\in\mathds{R}}$ has a jump at $|x|+2\theta$, while the jumps of $Y_N^\pm$ have typical size of order $\frac{1}{\sqrt{N}}$, so the jump in $(B_y^x \mathds{1}_{\{y\in[-|x|-2\theta,|x|+2\theta)\}})_{y\in\mathds{R}}$ is approximated in $Y_N^\pm$ by either a sequence of small jumps or a continuous slope, which prevents the convergence in the Skorohod $J_1$ topology. 

Finally, to prove Proposition \ref{prop_fluctuations_T} on the fluctuations of $T_{\lfloor N \theta\rfloor,\lfloor N x\rfloor}^\iota$, we use the fact that we have $T_{\lfloor N \theta\rfloor,\lfloor N x\rfloor}^\iota=\sum_{i\in\mathds{Z}}(\ell^+(T_{\lfloor N \theta\rfloor,\lfloor N x\rfloor}^\iota,i)+\ell^-(T_{\lfloor N \theta\rfloor,\lfloor N x\rfloor}^\iota,i))$. It can be checked that $|\ell^+(T_{\lfloor N \theta\rfloor,\lfloor N x\rfloor}^\iota,i)-\ell^-(T_{\lfloor N \theta\rfloor,\lfloor N x\rfloor}^\iota,i+1)|=0$ or 1, hence controlling the $\ell^-(T_{\lfloor N \theta\rfloor,\lfloor N x\rfloor}^\iota,i)$ is enough. By using the coupling for the $\ell^-(T_{\lfloor N \theta\rfloor,\lfloor N x\rfloor}^\iota,i+1)-\ell^-(T_{\lfloor N \theta\rfloor,\lfloor N x\rfloor}^\iota,i)$ when $\ell^-(T_{\lfloor N \theta\rfloor,\lfloor N x\rfloor}^\iota,i)$ is high enough and our estimates on the size of the window in which $\ell^-(T_{\lfloor N \theta\rfloor,\lfloor N x\rfloor}^\iota,i)$ is neither high enough nor 0, we can prove that $T_{\lfloor N \theta\rfloor,\lfloor N x\rfloor}^\iota$ is close to the integral of the sum of the i.i.d. random variables of the coupling, which will yield the convergence. 

\subsection{Organization of the paper}

In Section \ref{sec_coupling}, we define the coupling between the increments of the local time and i.i.d. random variables and prove some of its properties. In Section \ref{sec_hit0}, we control where the local times hit 0, as well as where the local times are too low for the coupling of Section \ref{sec_coupling} to be useful. In Section \ref{sec_Skorohod_dist}, we prove a bound on the Skorohod $M_1$ distance between $Y_N^\pm$ and the renormalized sum of the i.i.d. random variables of the coupling set to 0 outside of $[-|x|-2\theta,|x|+2\theta)$ by writing explicit parametric representations of the two functions. In Section \ref{sec_conv_processes}, we complete the proof of the convergence of $Y_N^\pm$ stated in Theorem \ref{thm_main} and Proposition \ref{prop_uniform_conv}. In Section \ref{sec_no_J1_conv}, we prove that as claimed in Proposition \ref{prop_no_J1_conv}, $Y_N^\pm$ does not converge in the $J_1$ topology. Finally, in Section \ref{sec_conv_T}, we prove the convergence of the fluctuations of $T_{\lfloor N \theta\rfloor,\lfloor N x\rfloor}^\pm$ stated in Proposition \ref{prop_fluctuations_T}. 

In what follows, we set $\theta>0$, $\iota\in\{+,-\}$ and $x > 0$ (the cases $x<0$ and $x=0$ can be dealt with in the same way). To shorten the notation, we denote $T_N=T_{\lfloor N \theta\rfloor,\lfloor N x\rfloor}^\iota$. Moreover, for any $a,b\in\mathds{R}$, we denote $a \vee b=\max(a,b)$ and $a \wedge b=\min(a,b)$.

\section{Coupling of the local times increments with i.i.d. random variables}\label{sec_coupling}

Our goal in this section will be to couple the $\ell^\pm(T_N,i+1)-\ell^\pm(T_N,i)$ with i.i.d. random variables and to prove some properties of this coupling. This part of the work is not very different from what was done in \cite{Toth_et_al2008}, but we still recall their concepts and definitions. If we fix $i\in\mathds{Z}$ and observe the evolution of $(\ell^-(n,i)-\ell^+(n,i))_{n\in\mathds{N}}$, and if we ignore the steps at which $\ell^-(n,i)-\ell^+(n,i)$ does not move (i.e. those at which the random walk is not at $i$), we obtain a Markov chain $\xi_i$ whose distribution $\xi$ has the following transition probabilities: for all $n \in \mathds{N}$, $\mathds{P}(\xi(n+1)=\xi(n)\pm1)=\frac{w(\mp\xi(n))}{w(\xi(n))+w(-\xi(n))}$, and so that $\xi_i(0)=0$. Now, we denote $\tau_{i,\pm}(0)=0$ and for any $n\in\mathds{N}$, we denote $\tau_{i,\pm}(n+1)=\inf\{m>\tau_{i,\pm}(n)\,|\,\xi_i(m)=\xi_i(m-1)\pm1\}$, so that $\tau_{i,+}(n)$ is the time of the $n$-th upwards step of $\xi_i$ and $\tau_{i,-}(n)$ is the time of the $n$-th downwards step of $\xi_i$. Then since the distribution of $\xi$ is symmetric, the processes $(\eta_{i,+}(n))_{n\in\mathds{N}}=(-\xi_{i}(\tau_{i,+}(n)))_{n\in\mathds{N}}$ and $(\eta_{i,-}(n))_{n\in\mathds{N}}=(\xi_{i}(\tau_{i,-}(n)))_{n\in\mathds{N}}$ have the same distribution, called $\eta$, and it can be checked that $\eta$ is a Markov chain. 

We are going to give an expression of $\ell^\pm(T_N,i+1)-\ell^\pm(T_N,i)$ depending on the $\eta_{i,-}$, $\eta_{i,+}$. We assume $N$ large enough (so that $\lfloor N x\rfloor-1>0$). By definition of $T_N$ we have $X_{T_N}=\lfloor N x\rfloor\iota 1$. If $i \leq 0$ we thus have $X_{T_N}>i$, so the last step of the walk at $i$ before $T_N$ was going to the right, so the last step of $\xi_i$ was a downwards step, and by definition of $\ell^+(T_N,i)$ we have that $\xi_i$ made $\ell^+(T_N,i)$ downwards steps, hence $\ell^-(T_N,i)-\ell^+(T_N,i)=\xi_i(\tau_{i,-}(\ell^+(T_N,i)))=\eta_{i,-}(\ell^+(T_N,i))$, which yields $\ell^-(T_N,i)-\ell^+(T_N,i)=\eta_{i,-}(\ell^+(T_N,i))$. In addition, $\ell^-(T_N,i)=\ell^+(T_N,i-1)$, hence $\ell^+(T_N,i-1)=\ell^+(T_N,i)+\eta_{i,-}(\ell^+(T_N,i))$. If $0 < i < \lfloor N x\rfloor$ (for $\iota=-$) or $0 < i \leq \lfloor N x\rfloor$ (for $\iota=+$), the last step of the walk at $i$ was also going to the right, so we also have $\ell^-(T_N,i)-\ell^+(T_N,i)=\eta_{i,-}(\ell^+(T_N,i))$. However, $\ell^-(T_N,i)=\ell^+(T_N,i-1)-1$, so $\ell^+(T_N,i-1)=\ell^+(T_N,i)+\eta_{i,-}(\ell^+(T_N,i))+1$. Finally, if $i \geq \lfloor N x\rfloor$ (for $\iota=-$) or $i > \lfloor N x\rfloor$ (for $\iota=+$), then the last step of the walk at $i$ was going to the left, so the last step of $\xi_i$ was an upwards step, and $\xi_i$ made $\ell^-(T_N,i)$ upwards steps, therefore $\ell^-(T_N,i)-\ell^+(T_N,i)=\xi_i(\tau_{i,+}(\ell^-(T_N,i)))=-\eta_{i,+}(\ell^-(T_N,i))$, which yields $\ell^-(T_N,i)-\ell^+(T_N,i)=-\eta_{i,+}(\ell^-(T_N,i))$. Moreover, $\ell^+(T_N,i)=\ell^-(T_N,i+1)$, hence $\ell^-(T_N,i+1)=\ell^-(T_N,i)+\eta_{i,+}(\ell^-(T_N,i))$. 

We are going to use these results to deduce an expression of the $\ell^\pm(T_N,i)$ which will be very useful throughout this work. Denoting $\chi(N)=\lfloor N x\rfloor$ if $\iota=-$ and $\chi(N)=\lfloor N x\rfloor+1$ if $\iota=+$, for $i \geq \chi(N)$ we have $\ell^-(T_N,i)=\ell^-(T_N,\chi(N))+\sum_{j=\chi(N)}^{i-1}\eta_{j,+}(\ell^-(T_N,j))$, and for $i < \chi(N)$ we have $\ell^+(T_N,i)=\ell^+(T_N,\chi(N)-1)+\sum_{j=i+1}^{\chi(N)-1}(\eta_{j,-}(\ell^+(T_N,j))+\mathds{1}_{\{j > 0\}})$. Now, we remember that the definition of $T_N$ implies $\ell^\iota(T_N,\lfloor N x\rfloor)=\lfloor N \theta\rfloor$, so if $\iota=-$ we have $\ell^-(T_N,\chi(N))=\lfloor N \theta\rfloor$ and $\ell^+(T_N,\chi(N)-1)=\ell^-(T_N,\chi(N))=\lfloor N \theta\rfloor$, and if $\iota=+$ we have $\ell^+(T_N,\chi(N)-1)=\lfloor N \theta\rfloor$ and $\ell^-(T_N,\chi(N))=\ell^+(T_N,\chi(N)-1)-1=\lfloor N \theta\rfloor-1$. Consequently, we have the following.
\begin{equation}\label{eq_local_times}
\begin{split}
 \text{If }i \geq \chi(N),& \quad \ell^-(T_N,i)=\lfloor N \theta\rfloor-\mathds{1}_{\{\iota=+\}}+\sum_{j=\chi(N)}^{i-1}\eta_{j,+}(\ell^-(T_N,j)). \\
 \text{If }i < \chi(N),& \quad \ell^+(T_N,i)=\lfloor N \theta\rfloor+\sum_{j=i+1}^{\chi(N)-1}(\eta_{j,-}(\ell^+(T_N,j))+\mathds{1}_{\{j > 0\}}).
 \end{split}
\end{equation}
We will also need to remember the following.
\begin{equation}\label{eq_local_times_2}
 \text{If }i \geq \chi(N),\quad \ell^-(T_N,i)-\ell^+(T_N,i)=-\eta_{i,+}(\ell^-(T_N,i)). 
 \quad\text{ If } i < \chi(N),\quad \ell^-(T_N,i)-\ell^+(T_N,i)=\eta_{i,-}(\ell^+(T_N,i)).
\end{equation}

To couple the $\ell^\pm(T_N,i+1)-\ell^\pm(T_N,i)$ with i.i.d. random variables, we need to understand the $\eta_{i,+}(\ell^-(T_N,i))$ and the $\eta_{i,-}(\ell^+(T_N,i))$. \cite{Toth_et_al2008} proved that the following measure $\rho_-$ is the unique invariant probability distribution of the Markov chain $\eta$:  
\begin{equation}\label{eq_rho}
 \forall i \in \mathds{Z}, \quad \rho_-(i)=\frac{1}{R} \prod_{j=1}^{\lfloor|2i+1|/2\rfloor}\frac{w(-j)}{w(j)}\quad\text{with}\quad R=\sum_{i\in\mathds{Z}}\prod_{j=1}^{\lfloor|2i+1|/2\rfloor}\frac{w(-j)}{w(j)}. 
\end{equation}
We also denote $\rho_0$ the measure on $\frac{1}{2}+\mathds{Z}$ defined by $\rho_0(\cdot)=\rho_-(\cdot-\frac{1}{2})$. 

We are now in position to construct the coupling of the $\ell^\pm(T_N,i+1)-\ell^\pm(T_N,i)$ with i.i.d. random variables $(\zeta_i)_{i\in\mathds{Z}}$. The idea is that $\eta$ can be expected to converge to its invariant distribution $\rho_-$, hence when $\ell^\pm(T_N,i)$ is large, $\eta_{i,\mp}(\ell^\pm(T_N,i))$ will be close to a random variable of law $\rho_-$. More rigorously, we begin by defining an i.i.d. sequence $(r_i)_{i\in\mathds{Z}}$ of random variables of distribution $\rho_-$ so that for $i \geq \chi(N)$ then $\mathds{P}(r_i \neq \eta_{i,+}(\lfloor N^{1/6}\rfloor))$ is minimal, and for $i < \chi(N)$ then $\mathds{P}(r_i \neq \eta_{i,-}(\lfloor N^{1/6}\rfloor))$ is minimal. We can then define i.i.d. Markov chains $(\bar\eta_{i,+}(n))_{n \geq \lfloor N^{1/6}\rfloor}$ for $i \geq \chi(N)$ and $(\bar\eta_{i,-}(n))_{n \geq \lfloor N^{1/6}\rfloor}$ for $i < \chi(N)$ so that $\bar\eta_{i,\pm}(\lfloor N^{1/6}\rfloor)=r_i$, $\bar\eta_{i,\pm}$ is a Markov chain of distribution that of $\eta$, and if $\bar\eta_{i,\pm}(\lfloor N^{1/6}\rfloor)=\eta_{i,\pm}(\lfloor N^{1/6}\rfloor)$ then $\bar\eta_{i,\pm}(n)=\eta_{i,\pm}(n)$ for any $n \geq \lfloor N^{1/6}\rfloor$. Since $\rho_-$ is invariant for $\eta$, if $n \geq \lfloor N^{1/6}\rfloor$, the $\bar\eta_{i,+}(n)$ for $i \geq \chi(N)$ and $\bar\eta_{i,-}(n)$ for $i < \chi(N)$ have distribution $\rho_-$. We define the random variables $(\zeta_i)_{i\in\mathds{Z}}$ as follows: for $i \geq \chi(N)$ we set $\zeta_i=\bar\eta_{i,+}(\ell^-(T_N,i)\vee\lfloor N^{1/6}\rfloor)+\frac{1}{2}$, and for $i < \chi(N)$ we set $\zeta_i=\bar\eta_{i,-}(\ell^+(T_N,i)\vee\lfloor N^{1/6}\rfloor)+\frac{1}{2}$. For $i \geq \chi(N)$, \eqref{eq_local_times} implies that $\ell^-(T_N,i)$ depends only on the $\eta_{j,+}$, $\chi(N) \leq j \leq i-1$, hence is independent from $\bar\eta_{i,+}$, which implies $\zeta_i$ has distribution $\rho_0$ and is independent from the $\zeta_j$, $\chi(N) \leq j \leq i-1$. This and a similar argument for $i < \chi(N)$ implies the $(\zeta_i)_{i\in\mathds{Z}}$ are i.i.d. with distribution $\rho_0$. 

We will prove several properties of $(\zeta_i)_{i\in\mathds{Z}}$ that we will use in the remainder of the proof. In order to do that, we need the following lemma of \cite{Toth_et_al2008}.

\begin{lemma}[Lemma 1 of \cite{Toth_et_al2008}]\label{lem_Toth}
 There exist two constants $\tilde c =\tilde c(w)>0$ and $\tilde C = \tilde C(w) < +\infty$ so that for any $n\in\mathds{N}$,
 \[
  \mathds{P}(\eta(n)=i|\eta(0)=0) \leq \tilde C e^{-\tilde c |i|} \quad \text{and} 
  \quad \sum_{i\in\mathds{Z}}|\mathds{P}(\eta(n)=i|\eta(0)=0)-\rho_-(i)| \leq \tilde C e^{-\tilde c n}.
 \]
\end{lemma}

Firstly, we want to prove that our coupling is actually useful: that the $\zeta_i$ are close to the $\ell^\pm(T_N,i+1)-\ell^\pm(T_N,i)$. More precisely, we will show that except on an event of probability tending to 0, if $\ell^\pm(T_N,i)$ is large then $\zeta_i=\eta_{i,\mp}(\ell^\pm(T_N,i))+1/2$, which \eqref{eq_local_times} relates to $\ell^\pm(T_N,i+1)-\ell^\pm(T_N,i)$. We denote 
\begin{equation}\label{eq_zeta=eta}
\begin{split}
 \mathcal{B}_1^-&=\{\exists\, i\in\{-\lceil2(|x|+2\theta)N\rceil,....,\chi(N)-1\}, \ell^+(T_N,i)\geq\lfloor N^{1/6}\rfloor\text{ and }\zeta_i\neq\eta_{i,-}(\ell^+(T_N,i))+1/2\}, \\
 \mathcal{B}_1^+&=\{\exists\, i\in\{\chi(N),....,\lceil2(|x|+2\theta)N\rceil\}, \ell^-(T_N,i)\geq\lfloor N^{1/6}\rfloor\text{ and }\zeta_i\neq\eta_{i,+}(\ell^-(T_N,i))+1/2\}.
 \end{split}
\end{equation}
Lemma \ref{lem_Toth} will allow us to prove the following.

\begin{lemma}\label{lem_zeta=eta}
$\mathds{P}(\mathcal{B}_1^-)$ and $\mathds{P}(\mathcal{B}_1^+)$ tend to 0 when $N \to +\infty$.
\end{lemma}

\begin{proof}
 By definition, for any $i\in\{-\lceil2(|x|+2\theta)N\rceil,....,\chi(N)-1\}$ we have $\zeta_i=\bar\eta_{i,-}(\ell^+(T_N,i)\vee\lfloor N^{1/6}\rfloor)+\frac{1}{2}$, which is $\bar\eta_{i,-}(\ell^+(T_N,i))+\frac{1}{2}$ when $\ell^+(T_N,i)\geq\lfloor N^{1/6}\rfloor$. Now, $\bar\eta_{i,-}=\eta_{i,-}$ if $\bar\eta_{i,-}(\lfloor N^{1/6}\rfloor)=\eta_{i,-}(\lfloor N^{1/6}\rfloor)$, that is $r_i=\eta_{i,-}(\lfloor N^{1/6}\rfloor)$. We deduce $\mathds{P}(\mathcal{B}_1^-) \leq \mathds{P}(\exists\, i\in\{-\lceil2(|x|+2\theta)N\rceil,....,\chi(N)-1\},r_i\neq\eta_{i,-}(\lfloor N^{1/6}\rfloor))$. Now, for any $i < \chi(N)$, we have $\mathds{P}(r_i\neq\eta_{i,-}(\lfloor N^{1/6}\rfloor))$ minimal, thus smaller than $\tilde C e^{-\tilde c \lfloor N^{1/6}\rfloor}$ by Lemma \ref{lem_Toth}. Consequently, when $N$ is large enough, $\mathds{P}(\mathcal{B}_1^-) \leq 3(|x|+2\theta)N\tilde C e^{-\tilde c \lfloor N^{1/6}\rfloor}$, which tends to 0 when $N \to +\infty$. The proof for $\mathds{P}(\mathcal{B}_1^+)$ is the same.
\end{proof}

Unfortunately, the previous lemma does not allow to control the local times when $\ell^\pm(T_N,i)$ is small. In order to do that, we show several additional properties. We have to control the probability of 
\begin{equation*}
\begin{split}
 \mathcal{B}_2=&\{\exists \, i \in \{-\lceil2(|x|+2\theta)N\rceil,....,\lceil2(|x|+2\theta)N\rceil\}, |\zeta_i| \geq N^{1/16}\} \\
 &\cup \{\exists \, i \in \{-\lceil2(|x|+2\theta)N\rceil,....,\chi(N)-1\}, |\eta_{i,-}(\ell^+(T_N,i))+1/2| \geq N^{1/16}\} \\
 &\cup \{\exists \, i \in \{\chi(N),....,\lceil2(|x|+2\theta)N\rceil\}, |\eta_{i,+}(\ell^-(T_N,i))+1/2| \geq N^{1/16}\}.
 \end{split}
\end{equation*}

\begin{lemma}\label{lem_zeta_eta_small}
 $\mathds{P}(\mathcal{B}_2)$ tends to 0 when $N$ tends to $+\infty$.
\end{lemma}

\begin{proof}
It is enough to find some constants $c>0$ and $C < +\infty$ so that for any $i \in \{-\lceil2(|x|+2\theta)N\rceil,....,\lceil2(|x|+2\theta)N\rceil\}$ we have $\mathds{P}(|\zeta_i| \geq N^{1/16}) \leq Ce^{-cN^{1/16}}$, for any $i \in \{-\lceil2(|x|+2\theta)N\rceil,....,\chi(N)-1\}$ we have $\mathds{P}(|\eta_{i,-}(\ell^+(T_N,i))+1/2| \geq N^{1/16}) \leq Ce^{-cN^{1/16}}$, and for all $i \in \{\chi(N),....,\lceil2(|x|+2\theta)N\rceil\}$ we have $\mathds{P}(|\eta_{i,+}(\ell^-(T_N,i))+1/2| \geq N^{1/16}) \leq Ce^{-cN^{1/16}}$. For all $i\in\mathds{Z}$, $\zeta_i$ has distribution $\rho_0$, which has exponential tails, hence there exists constants $c'=c'(w)>0$ and $C'=C'(w) < +\infty$ so that for $i \in \{-\lceil2(|x|+2\theta)N\rceil,....,\lceil2(|x|+2\theta)N\rceil\}$ we have $\mathds{P}(|\zeta_i| \geq N^{1/16}) \leq C'e^{-c'N^{1/16}}$. We now consider $i \in \{-\lceil2(|x|+2\theta)N\rceil,....,\chi(N)-1\}$ and $\mathds{P}(|\eta_{i,-}(\ell^+(T_N,i))+1/2| \geq N^{1/16})$ (the $\mathds{P}(|\eta_{i,+}(\ell^-(T_N,i))+1/2| \geq N^{1/16})$ can be dealt with in the same way). Equation \eqref{eq_local_times} implies $\ell^+(T_N,i)$ depends only on the $\eta_{j,-}$ for $j>i$, hence is independent of $\eta_{i,-}$. This implies $\mathds{P}(|\eta_{i,-}(\ell^+(T_N,i))+1/2| \geq N^{1/16}) = \sum_{k\in\mathds{N}}\mathds{P}(|\eta_{i,-}(k)+1/2| \geq N^{1/16})\mathds{P}(\ell^+(T_N,i)=k)$. Therefore the first part of Lemma \ref{lem_Toth} implies $\mathds{P}(|\eta_{i,-}(\ell^+(T_N,i))+1/2| \geq N^{1/16}) \leq \sum_{k\in\mathds{N}}\frac{2\tilde C e^{\tilde c /2}}{1-e^{-\tilde c}}e^{-\tilde c N^{1/16}}\mathds{P}(\ell^+(T_N,i)=k)=\frac{2\tilde C e^{\tilde c /2}}{1-e^{-\tilde c}}e^{-\tilde c N^{1/16}}$, which is enough. 
\end{proof}

We will also need the following, which is a rather standard result of large deviations.

\begin{lemma}\label{lem_large_deviations}
 For any $\alpha>0$, $\varepsilon>0$, $\mathds{P}(\max_{0 \leq i_1 \leq i_2 \leq \lceil N^\alpha\rceil}|\sum_{i=i_1}^{i_2}\zeta_i| \geq N^{\alpha/2+\varepsilon})$ tends to 0 when $N \to +\infty$.
\end{lemma}

\begin{proof}
 Let $0 \leq i_1 \leq i_2 \leq \lceil N^\alpha\rceil$, let us study $\mathds{P}(|\sum_{i=i_1}^{i_2}\zeta_i| \geq N^{\alpha/2+\varepsilon})$. We know the $\zeta_i$, $i\in\mathds{Z}$ are i.i.d. with distribution $\rho_0$, and it can be checked that $\rho_0$ is symmetric with respect to 0, so from that and the Markov inequality we get 
 \begin{equation}\label{eq_large_devations}
  \begin{split}
  \mathds{P}\left(\left|\sum_{i=i_1}^{i_2}\zeta_i\right| \geq N^{\alpha/2+\varepsilon}\right)
  \leq 2 \mathds{P}\left(\sum_{i=i_1}^{i_2}\zeta_i \geq N^{\alpha/2+\varepsilon}\right)
  =2 \mathds{P}\left(\exp\left(\frac{1}{N^{\alpha/2}}\sum_{i=i_1}^{i_2}\zeta_i\right) \geq \exp(N^{\varepsilon})\right) \\
  \leq 2e^{-N^{\varepsilon}}\mathds{E}\left(\exp\left(\frac{1}{N^{\alpha/2}}\sum_{i=i_1}^{i_2}\zeta_i\right)\right)
  \leq 2e^{-N^{\varepsilon}}\prod_{i=i_1}^{i_2}\mathds{E}\left(\exp\left(\frac{1}{N^{\alpha/2}}\zeta_i\right)\right).\qquad\qquad\quad
 \end{split}
 \end{equation}
 Now, if $\zeta$ has distribution $\rho_0$, we can write $\exp(\frac{1}{N^{\alpha/2}}\zeta)=1+\frac{1}{N^{\alpha/2}}\zeta+\frac{1}{2}(\frac{1}{N^{\alpha/2}}\zeta)^2\exp(\frac{1}{N^{\alpha/2}}\zeta')$ with $|\zeta'|\leq|\zeta|$. Since $\rho_0$ is symmetric with respect to 0, we have $\mathds{E}(\zeta)=0$, therefore  
 \[
  \mathds{E}\left(\exp\left(\frac{1}{N^{\alpha/2}}\zeta\right)\right)
  = 1+\mathds{E}\left(\frac{1}{2}\left(\frac{1}{N^{\alpha/2}}\zeta\right)^2\exp\left(\frac{1}{N^{\alpha/2}}\zeta'\right)\right)
  \leq 1+\frac{1}{2N^\alpha}\mathds{E}\left(\zeta^2\exp\left(\frac{1}{N^{\alpha/2}}|\zeta|\right)\right). 
 \]
Moreover, $\rho_0$ has exponential tails, hence there exists constants $C<+\infty$ and $c>0$ so that $\mathds{E}(\zeta^2e^{c \,|\zeta|}) \leq C$. When $N$ is large enough, $\frac{1}{N^{\alpha/2}} \leq c$, therefore $\mathds{E}(\exp(\frac{1}{N^{\alpha/2}}\zeta)) \leq 1+\frac{C}{2N^\alpha}\leq \exp(\frac{C}{2N^\alpha})$. Together with \eqref{eq_large_devations}, this yields $\mathds{P}(|\sum_{i=i_1}^{i_2}\zeta_i| \geq N^{\alpha/2+\varepsilon})\leq 2e^{-N^{\varepsilon}}e^{(i_2-i_1+1)\frac{C}{2N^\alpha}} \leq 2e^{-N^{\varepsilon}}e^{(\lceil N^\alpha\rceil+1)\frac{C}{2N^\alpha}} \leq 2e^{C}e^{-N^{\varepsilon}}$ when $N$ is large enough. We deduce that when $N$ is large enough, $\mathds{P}(\max_{0 \leq i_1 \leq i_2 \leq \lceil N^\alpha\rceil}|\sum_{i=i_1}^{i_2}\zeta_i| \geq N^{\alpha/2+\varepsilon}) \leq (\lceil N^\alpha\rceil+1)^2 2e^{C}e^{-N^{\varepsilon}}$, which tends to 0 when $N$ tends to $+\infty$.
\end{proof}

We also prove an immediate application of Lemma \ref{lem_large_deviations}, which we will use several times. If we define
\[
 \mathcal{B}_3^-=\left\{\max_{-\lfloor (|x|+2\theta)N\rfloor-N^{3/4} \leq i_1 \leq i_2 \leq -\lfloor (|x|+2\theta)N\rfloor+N^{3/4}} \left|\sum_{i=i_1}^{i_2} \zeta_i\right| \geq N^{19/48}\right\},
 \]
 \[
 \mathcal{B}_3^+=\left\{\max_{\lfloor (|x|+2\theta)N\rfloor-N^{3/4} \leq i_1 \leq i_2 \leq \lfloor (|x|+2\theta)N\rfloor+N^{3/4}} \left|\sum_{i=i_1}^{i_2} \zeta_i\right| \geq N^{19/48}\right\},
\]
we have the following lemma.

\begin{lemma}\label{lem_sum_zeta}
 $\mathds{P}(\mathcal{B}_3^-)$ and $\mathds{P}(\mathcal{B}_3^+)$ tend to 0 when $N$ tends to $+\infty$. 
\end{lemma}

\begin{proof}
 Since the $(\zeta_i)_{i\in\mathds{Z}}$ are i.i.d., $\mathds{P}(\mathcal{B}_3^+)=\mathds{P}(\mathcal{B}_3^-)=\mathds{P}(\max_{0 \leq i_1 \leq i_2 \leq 2 \lceil N^{3/4} \rceil} |\sum_{i=i_1}^{i_2} \zeta_i| \geq N^{19/48})$, which is smaller than $\mathds{P}(\max_{0 \leq i_1 \leq i_2 \leq \lceil N^{37/48} \rceil} |\sum_{i=i_1}^{i_2} \zeta_i| \geq N^{19/48})$ when $N$ is large enough. Moreover, Lemma \ref{lem_large_deviations}, used with $\alpha=37/48$ and $\varepsilon=1/96$, yields that the latter probability tends to 0 when $N$ tends to $+\infty$. 
\end{proof}

\section{Where the local times approach 0}\label{sec_hit0}

The aim of this section is twofold. Firstly, we need to control the place where $\ell^-(T_N,i)$ hits 0 when $i$ is at the right of 0, as well as the place where $\ell^+(T_N,i)$ hits 0 when $i$ is at the left of 0. Secondly, we have to show that even when $\ell^\pm(T_N,i)$ is close to 0, the local times do not stray too far away from the coupling. For any $N \in \mathds{N}$, we denote $I^+=\inf\{i \geq \chi(N)\,|\,\ell^-(T_N,i)=0\}$ and $I^-=\sup\{i < \chi(N)\,|\,\ell^+(T_N,i)=0\}$. We notice that $\ell^+(T_N,I^-) = 0$, and from the definition of $T_N$ we have $\ell^+(T_N,i) > 0$ for any $0 \leq i \leq \chi(N)-1$, hence $I^-<0$. We first state an elementary result that we will use many times in this work.

\begin{lemma}\label{lem_0_after_I}
 For any $i \geq I^+$ or $i \leq I^-$ we have $\ell^\pm(T_N,i)=0$.
\end{lemma}

\begin{proof}
 Since $\ell^+(T_N,I^-) = 0$ and the random walk is at $\lfloor Nx \rfloor \iota1>0$ at time $T_N$, the random walk did not reach $I^-$ before time $T_N$, thus $\ell^\pm(T_N,i)=0$ for any $i \leq I^-$. Moreover, $\ell^-(T_N,\chi(N)) > 0$ by definition of $T_N$, hence $I^+>\chi(N)$ thus $X_{T_N} < I^+$ hence $\ell^-(T_N,I^+) = 0$ implies the random walk did not reach $I^+$ before time $T_N$, thus $\ell^\pm(T_N,i)=0$ for any $i \geq I^+$. 
\end{proof}

We will also need the auxiliary random variables $\tilde I^+=\inf\{i \geq \chi(N)\,|\,\ell^-(T_N,i)\leq\lfloor N^{1/6}\rfloor\}$ and $\tilde I^-=\sup\{i < \chi(N)\,|\,\ell^+(T_N,i)\leq\lfloor N^{1/6}\rfloor\}$. 

\subsection{Place where we hit 0}

We have the following result of control on $I^+$ and $I^-$. 

\begin{lemma}\label{lem_control_I}
 For any $\delta>0$, $\mathds{P}(|I^-+(|x|+2\theta)N| \geq N^{\delta+1/2})$ and $\mathds{P}(|I^+-(|x|+2\theta)N| \geq N^{\delta+1/2})$ tend to 0 when $N$ tends to $+\infty$. 
\end{lemma}

\begin{proof}
The idea is to control the fluctuations of the local times around their deterministic limit: as long as $\ell^\pm(T_N,i)$ is large, the $\ell^\pm(T_N,i+1)-\ell^\pm(T_N,i)$ will be close to the i.i.d. random variables of the coupling, so the fluctuations of $\ell^\pm(T_N,i)$ around its deterministic limit are bounded and $\ell^\pm(T_N,i)$ can be small only when the deterministic limit is small, that is around $-(|x|+2\theta)N$ and $(|x|+2\theta)N$. We only spell out the proof for $I^-$, as the argument for $I^+$ is similar. The fact that $\mathds{P}(I^-+(|x|+2\theta)N \leq -N^{\delta+1/2})$ tends to 0 when $N$ tends to $+\infty$ comes from inequalities (51) and (53) of \cite{Toth_et_al2008}, so we only have to prove that $\mathds{P}(I^-+(|x|+2\theta)N \geq N^{\delta+1/2})$ tends to 0 when $N$ tends to $+\infty$. Since $I^- \leq \tilde I^-$, it is enough to prove that $\mathds{P}(\tilde I^-+(|x|+2\theta)N \geq N^{\delta+1/2})$ tends to 0 when $N$ tends to $+\infty$. Since by Lemma \ref{lem_zeta=eta} we have that $\mathds{P}(\mathcal{B}_1^-)$ tends to 0 when $N$ tends to $+\infty$, it is enough to prove $\mathds{P}(\tilde I^-+(|x|+2\theta)N \geq N^{\delta+1/2},(\mathcal{B}_1^-)^c)$ tends to 0 when $N$ tends to $+\infty$.  We now assume $N$ is large enough, $\tilde I^-+(|x|+2\theta)N \geq N^{\delta+1/2}$ and $(\mathcal{B}_1^-)^c$. Then there exists $i\in\{\lceil-(|x|+2\theta)N+N^{\delta+1/2}\rceil,...,\chi(N)-1\}$ so that $\ell^+(T_N,i) \leq \lfloor N^{1/6}\rfloor$ and $\ell^+(T_N,j) > \lfloor N^{1/6}\rfloor$ for all $j \in\{i+1,...,\chi(N)-1\}$. Thus, by \eqref{eq_local_times} we get $\lfloor N \theta\rfloor+\sum_{j=i+1}^{\chi(N)-1}(\eta_{j,-}(\ell^+(T_N,j))+\mathds{1}_{\{j > 0\}})=\ell^+(T_N,i) \leq \lfloor N^{1/6}\rfloor$. Furthermore, for all $j \in\{i+1,...,\chi(N)-1\}$, since $(\mathcal{B}_1^-)^c$ occurs and $\ell^+(T_N,j) > \lfloor N^{1/6}\rfloor$, we have $\eta_{j,-}(\ell^+(T_N,j))+1/2=\zeta_j$. We deduce $\lfloor N \theta\rfloor+\sum_{j=i+1}^{\chi(N)-1}(\zeta_j+(\mathds{1}_{\{j > 0\}}-\mathds{1}_{\{j \leq 0\}})/2) \leq \lfloor N^{1/6}\rfloor$, thus $\sum_{j=i+1}^{\chi(N)-1}\zeta_j+\lfloor N \theta\rfloor+\sum_{j=i+1}^{\chi(N)-1}(\mathds{1}_{\{j > 0\}}-\mathds{1}_{\{j \leq 0\}})/2 \leq \lfloor N^{1/6}\rfloor$. Moreover, since $i \in\{\lceil-(|x|+2\theta)N+N^{\delta+1/2}\rceil,...,\chi(N)-1\}$ we have $\sum_{j=i+1}^{\chi(N)-1}(\mathds{1}_{\{j > 0\}}-\mathds{1}_{\{j \leq 0\}})/2 =\frac{1}{2}(\chi(N)-1+i) \geq \frac{1}{2}(Nx-2-(|x|+2\theta)N+N^{\delta+1/2})=-\theta N+\frac{1}{2}N^{\delta+1/2}-1$. This yields $\sum_{j=i+1}^{\chi(N)-1}\zeta_j+\lfloor N \theta\rfloor-\theta N+\frac{1}{2}N^{\delta+1/2}-1 \leq \lfloor N^{1/6}\rfloor$, hence $\sum_{j=i+1}^{\chi(N)-1}\zeta_j \leq -\frac{1}{2}N^{\delta+1/2}+\lfloor N^{1/6}\rfloor+2 \leq -N^{(1+\delta)/2}$ since $N$ is large enough. Consequently, when $N$ is large enough, $\mathds{P}(\tilde I^-+(|x|+2\theta)N \geq N^{\delta+1/2},(\mathcal{B}_1^-)^c) \leq \mathds{P}(\exists\, i\in\{\lceil-(|x|+2\theta)N+N^{\delta+1/2}\rceil,...,\chi(N)-1\}, \sum_{j=i+1}^{\chi(N)-1}\zeta_j \leq -N^{(1+\delta)/2})$. Since the $\zeta_i$, $i\in\mathds{Z}$ are i.i.d., when $N$ is large enough this yields $\mathds{P}(\tilde I^-+(|x|+2\theta)N \geq N^{\delta+1/2},(\mathcal{B}_1^-)^c) \leq \mathds{P}(\max_{0 \leq i_1 \leq i_2 \leq \lceil N^{1+\delta/2}\rceil}|\sum_{i=i_1}^{i_2}\zeta_i|\geq N^{(1+\delta)/2})$, which tends to 0 when $N$ tends to $+\infty$ by Lemma \ref{lem_large_deviations} (applied with $\alpha=1+\delta/2$ and $\varepsilon=\delta/4$). This shows that $\mathds{P}(I^-+(|x|+2\theta)N \geq N^{\delta+1/2})$ converges to 0 when $N$ tends to $+\infty$, which ends the proof of Lemma \ref{lem_control_I}. 
\end{proof}

\subsection{Control of low local times}

We have to show that even when $\ell^\pm(T_N,i)$ is small, the local times are not too far from the random variables of the coupling. In order to do that, we first prove that the window where $\ell^\pm(T_N,i)$ is small but not zero, that is between $\tilde I^+$ and $I^+$ and between $I^-$ and $\tilde I^-$, is small. Afterwards, we will give bounds on what happens inside. We begin by showing the following easy result.

\begin{lemma}\label{lem_I^-<0}
 $\mathds{P}(\tilde I^- \geq 0)$ tends to 0 when $N \to +\infty$.
\end{lemma}

\begin{proof}
Let $N$ be large enough. If $\tilde I^- \geq 0$, there exists $i\in\{0,...,\lfloor Nx\rfloor\}$ so that $\ell^+(T_N,i)\leq\lfloor N^{1/6}\rfloor$. Since $N$ is large enough, this implies $\ell^+(T_N,i)\leq N\theta/2$, therefore $\sup_{y \in \mathds{R}}|\frac{1}{N}\ell^+(T_N,\lfloor N y\rfloor)-(\frac{|x|-|y|}{2}+\theta)_+| \geq \theta/2$. Moreover, by Theorem 1 of \cite{Toth_et_al2008}, $\sup_{y \in \mathds{R}}|\frac{1}{N}\ell^+(T_N,\lfloor N y\rfloor)-(\frac{|x|-|y|}{2}+\theta)_+|$ converges in probability to 0 when $N$ tends to $+\infty$, hence we deduce that $\mathds{P}(\sup_{y \in \mathds{R}}|\frac{1}{N}\ell^+(T_N,\lfloor N y\rfloor)-(\frac{|x|-|y|}{2}+\theta)_+| \geq \theta/2)$ tends to 0 when $N \to +\infty$. Therefore $\mathds{P}(\tilde I^- \geq 0)$ tends to 0 when $N \to +\infty$. 
\end{proof}

In order to control $I^+$, $I^-$, $\tilde I^+$ and $\tilde I^-$, we will use the fact the local times behave as the Markov chain $L$ from \cite{Toth_et_al2008}, defined as follows. We consider i.i.d. copies of the Markov chain $\eta$ starting at 0, called $(\eta_m)_{m\in\mathds{N}}$. For any $m\in\mathds{N}$, we then set $L(m+1)=L(m)+\eta_{m}(L(m))$. We denote $\tau=\inf\{m\in\mathds{N}\,|\,L(m) \leq 0\}$. The following was proven in \cite{Toth_et_al2008}.

\begin{lemma}[Lemma 2 of \cite{Toth_et_al2008}]\label{lem_Toth_tau}
 There exists a constant $K < +\infty$ so that for any $k\in \mathds{N}$ we have $\mathds{E}(\tau|L(0)=k) \leq 3k+K$. 
\end{lemma}

Since the local times will behave as $L$, Lemma \ref{lem_Toth_tau} implies that if the local time starts small, then the time at which it reaches 0 has small expectation hence is not too large. This will help us to prove the following control on the window where $\ell^\pm(T_N,i)$ is small but not zero. 

\begin{lemma}\label{lem_I-tildeI}
$\mathds{P}(I^+-\tilde I^+ \geq N^{1/4})$ and $\mathds{P}(\tilde I^--I^- \geq N^{1/4})$ tend to 0 when $N \to +\infty$.
\end{lemma}

\begin{proof}
Let $N$ be large enough. We deal only with $\mathds{P}(\tilde I^--I^- \geq N^{1/4})$, since $\mathds{P}(I^+-\tilde I^+ \geq N^{1/4})$ can be dealt with in the same way and with simpler arguments. Thanks to Lemma \ref{lem_I^-<0}, it is enough to prove that $\mathds{P}(\tilde I^--I^- \geq N^{1/4},\tilde I^-<0)$ tends to 0 when $N \to +\infty$. Moreover, if $\tilde I^- < 0$, thanks to \eqref{eq_local_times}, for any $i < \tilde I^-$ we get $\ell^+(T_N,i)=\ell^+(T_N,\tilde I^-)+\sum_{j=i+1}^{\tilde I^-}\eta_{j,-}(\ell^+(T_N,j))$, which allows to prove that $(\ell^+(T_N,\tilde I^--i))_{i \in \mathds{N}}$ is a Markov chain with the transition probabilities of $L$. Therefore we have (recalling the notations just before Lemma \ref{lem_Toth_tau}) 
 \[
  \mathds{P}\left(\tilde I^--I^- \geq N^{1/4},\tilde I^-<0\right)=\sum_{k=0}^{\lfloor N^{1/6} \rfloor}\mathds{P}\left(\tilde I^--I^- \geq N^{1/4},\tilde I^-<0\left|\ell^+(T_N,\tilde I^-)=k\right.\right)\mathds{P}\left(\ell^+(T_N,\tilde I^-)=k\right)
 \]
\[
 =\sum_{k=0}^{\lfloor N^{1/6} \rfloor}\mathds{P}\left(\left.\tau \geq N^{1/4}\right|L(0)=k\right)\mathds{P}\left(\ell^+(T_N,\tilde I^-)=k\right)
 \leq \sum_{k=0}^{\lfloor N^{1/6} \rfloor}\frac{1}{N^{1/4}}\mathds{E}(\tau|L(0)=k)\mathds{P}(\ell^+(T_N,\tilde I^-)=k).
\]
By Lemma \ref{lem_Toth_tau} we deduce 
 \[
  \mathds{P}\left(\tilde I^--I^- \geq N^{1/4},\tilde I^-<0\right)
   \leq \frac{1}{N^{1/4}}\sum_{k=0}^{\lfloor N^{1/6} \rfloor}(3k+K)\mathds{P}(\ell^+(T_N,\tilde I^-)=k)
   \leq \frac{3N^{1/6}+K}{N^{1/4}} \leq 4 N^{-1/12}
 \]
 since $N$ is large enough, hence $\mathds{P}(\tilde I^--I^- \geq N^{1/4},\tilde I^-<0)$ tends to 0 when $N \to +\infty$, which ends the proof.
\end{proof}

We are now going to prove that even when $\ell^\pm(T_N,i)$ is small, the local times are not too far from the random variables of the coupling. More precisely, for any $n\in\mathds{N}$, we define the following events.
\[
 \mathcal{B}_4^-=\left\{\exists \, i \in \{I^-,...,\chi(N)-1\},\left|\sum_{j=i+1}^{\chi(N)-1}(\eta_{j,-}(\ell^+(T_N,j))+1/2)-\sum_{j=i+1}^{\chi(N)-1}\zeta_j\right| \geq N^{1/3} \right\},
\]
\[
 \mathcal{B}_4^+=\left\{\exists \, i \in \{\chi(N),...,I^+\},\left|\sum_{j=\chi(N)}^{i-1}(\eta_{j,+}(\ell^-(T_N,j))+1/2)-\sum_{j=\chi(N)}^{i-1}\zeta_j\right| \geq N^{1/3} \right\}.
\]

\begin{lemma}\label{lem_sum_eta_zeta}
 $\mathds{P}(\mathcal{B}_4^-)$ and $\mathds{P}(\mathcal{B}_4^+)$ tend to 0 when $N$ tend to $+\infty$.
\end{lemma}

\begin{proof}
The idea of the argument is that when $\ell^\pm(T_N,i)$ is large, $\eta_{i,\mp}(\ell^\pm(T_N,i))+1/2=\zeta_i$ thanks to Lemma \ref{lem_zeta=eta}, that the window where $\ell^\pm(T_N,i)$ is small is bounded by Lemma \ref{lem_I-tildeI}, and that inside this window the $\eta_{i,\mp}(\ell^\pm(T_N,i))+1/2$, $\zeta_i$ are also bounded by Lemma \ref{lem_zeta_eta_small}. We only spell out the proof for $\mathds{P}(\mathcal{B}_4^-)$, since the proof for $\mathds{P}(\mathcal{B}_4^+)$ is the same. By Lemma \ref{lem_control_I}, we have that $\mathds{P}(I^- \leq -2(|x|+\theta)N)$ tends to 0 when $N$ tends to $+\infty$. Furthermore, Lemma \ref{lem_I-tildeI} implies that $\mathds{P}(\tilde I^--I^- \geq N^{1/4})$ tends to 0 when $N$ tends to $+\infty$. In addition, by Lemmas \ref{lem_zeta=eta} and \ref{lem_zeta_eta_small} we have that $\mathds{P}(\mathcal{B}_1^-)$ and $\mathds{P}(\mathcal{B}_2)$ tend to 0 when $N$ tends to $+\infty$. Consequently, it is enough to prove that for $N$ large enough, if $(\mathcal{B}_1^-)^c$, $(\mathcal{B}_2)^c$ occur, if $\tilde I^--I^- < N^{1/4}$ and if $I^- > -2(|x|+\theta)N$, then $(\mathcal{B}_4^-)^c$ occurs. We assume $(\mathcal{B}_1^-)^c$, $(\mathcal{B}_2)^c$, $\tilde I^--I^- < N^{1/4}$ and $I^- > -2(|x|+\theta)N$. Since $(\mathcal{B}_1^-)^c$ occurs and $\tilde I^- \geq I^- > -2(|x|+\theta)N$, we get $\zeta_i=\eta_{j,-}(\ell^+(T_N,j))+1/2$ for any $i\in\{\tilde I^-+1,...,\chi(N)-1\}$. Therefore, if $i\in\{\tilde I^-,...,\chi(N)-1\}$ we get $\sum_{j=i+1}^{\chi(N)-1}(\eta_{j,-}(\ell^+(T_N,j))+1/2)-\sum_{j=i+1}^{\chi(N)-1}\zeta_j=0$, and for $i \in \{I^-,...,\tilde I^--1\}$ we have 
\[
 \left|\sum_{j=i+1}^{\chi(N)-1}(\eta_{j,-}(\ell^+(T_N,j))+1/2)-\sum_{j=i+1}^{\chi(N)-1}\zeta_j\right|
 = \left|\sum_{j=i+1}^{\tilde I^-}(\eta_{j,-}(\ell^+(T_N,j))+1/2)-\sum_{j=i+1}^{\tilde I^-}\zeta_j\right|
 \]
 \[
 \leq \sum_{j=i+1}^{\tilde I^-}\left(|\eta_{j,-}(\ell^+(T_N,j))+1/2|+|\zeta_j|\right)
 \leq 2(\tilde I^- -I^-)N^{1/16}
\]
since $(\mathcal{B}_2^-)^c$ occurs, $i+1 \geq I^- > -2(|x|+\theta)N$ and by definition $\tilde I^- \leq \chi(N)-1 \leq 2(|x|+\theta)N$. Moreover, we assumed $\tilde I^--I^- < N^{1/4}$, which implies $|\sum_{j=i+1}^{\chi(N)-1}(\eta_{j,-}(\ell^+(T_N,j))+1/2)-\sum_{j=i+1}^{\chi(N)-1}\zeta_j| \leq 2 N^{1/4}N^{1/16}=2N^{5/16} < N^{1/3}$ when $N$ is large enough. Consequently, for any $i\in\{I^-,...,\chi(N)-1\}$ we have $|\sum_{j=i+1}^{\chi(N)-1}(\eta_{j,-}(\ell^+(T_N,j))+1/2)-\sum_{j=i+1}^{\chi(N)-1}\zeta_j| < N^{1/3}$, therefore $(\mathcal{B}_4^-)^c$ occurs, which ends the proof. 
\end{proof}

\section{Skorohod $M_1$ distance}\label{sec_Skorohod_dist}

The goal of this section is to prove that when $N$ is large, $Y_N^\pm$ is close in the Skorohod $M_1$ distance to the function $Y_N$ defined as follows. For any $N$ large enough, for $y\in\mathds{R}$, we set $Y_N(y)=\frac{1}{\sqrt{N}}\sum_{i=\lfloor N y\rfloor+1}^{\chi(N)-1}\zeta_i$ if $y \in [-|x|-2\theta,\frac{\chi(N)}{N})$, $Y_N(y)=\frac{1}{\sqrt{N}}\sum_{i=\chi(N)}^{\lfloor N y\rfloor-1}\zeta_i$ if $y \in [\frac{\chi(N)}{N},|x|+2\theta)$, and $Y_N(y)=0$ otherwise. We want to prove the following proposition.

\begin{proposition}\label{prop_distance_param}
 $\mathds{P}(d_{M_1}(Y_N^\pm,Y_N) > 3N^{-1/12})$ tends to 0 when $N$ tends to $+\infty$. 
\end{proposition}

If we denote 
\[
 \mathcal{B}=\mathcal{B}_1^- \cup \mathcal{B}_1^+ \cup \mathcal{B}_2 \cup \mathcal{B}_3^- \cup \mathcal{B}_3^+ \cup \mathcal{B}_4^- \cup \mathcal{B}_4^+ \cup \{|I^-+(|x|+2\theta)N| \geq N^{3/4}\} \cup \{|I^+-(|x|+2\theta)N| \geq N^{3/4}\},
\]
it will be enough to prove the following proposition.

\begin{proposition}\label{prop_distance_param_bis}
 When $N$ is large enough, for all $a >0$ with $|(|x|+2\theta)-a|>N^{-1/8}$, we have that $\mathcal{B}^c \subset \{d_{M_1,a}(Y_N^\pm|_{[-a,a]},Y_N|_{[-a,a]}) \leq 2N^{-1/12}\}$.  
\end{proposition}

\begin{proof}[Proof of Proposition \ref{prop_distance_param} given Proposition \ref{prop_distance_param_bis}.] We assume Proposition \ref{prop_distance_param_bis} holds. Then, when $N$ is large enough, if $\mathcal{B}^c$ occurs, for all $a >0$ with $|(|x|+2\theta)-a|>N^{-1/8}$ we have $d_{M_1,a}(Y_N^\pm|_{[-a,a]},Y_N|_{[-a,a]}) \leq 2N^{-1/12}$, which yields $d_{M_1}(Y_N^\pm,Y_N)=\int_0^{+\infty}e^{-a}(d_{M_1,a}(Y_N^\pm|_{[-a,a]},Y_N|_{[-a,a]}) \wedge 1 )\mathrm{d}a \leq \int_0^{+\infty}e^{-a}2N^{-1/12}\mathrm{d}a + 2 N^{-1/8}= 2N^{-1/12}+ 2 N^{-1/8} \leq 3 N^{-1/12}$. This implies $\mathds{P}(d_{M_1}(Y_N^\pm,Y_N) > 3 N^{-1/12}) \leq \mathds{P}(\mathcal{B})$ when $N$ is large enough. In addition,
\begin{align*}
 \mathds{P}(\mathcal{B}) \leq & \mathds{P}(\mathcal{B}_1^-) +\mathds{P}(\mathcal{B}_1^+) + \mathds{P}(\mathcal{B}_2) + \mathds{P}(\mathcal{B}_3^-) + \mathds{P}(\mathcal{B}_3^+) + \mathds{P}(\mathcal{B}_4^-) + \mathds{P}(\mathcal{B}_4^+) \\
 & +\mathds{P}(|I^-+(|x|+2\theta)N| \geq N^{3/4}) + \mathds{P}(|I^+-(|x|+2\theta)N| \geq N^{3/4}).
\end{align*}
Applying Lemmas \ref{lem_zeta=eta}, \ref{lem_zeta_eta_small}, \ref{lem_sum_zeta}, \ref{lem_control_I} and \ref{lem_sum_eta_zeta} implies $\mathds{P}(\mathcal{B})$ tends to 0 when $N$ tends to $+\infty$, hence $\mathds{P}(d_{M_1}(Y_N^\pm,Y_N) > 3 N^{-1/12})$ tends to 0 when $N$ tends to $+\infty$, which is Proposition \ref{prop_distance_param}. 
 
\end{proof}

The remainder of this section is devoted to the proof of Proposition \ref{prop_distance_param_bis}. The first thing we do is showing that between $\frac{(-(|x|+2\theta)N) \vee I^-}{N}$ and $\frac{((|x|+2\theta)N) \wedge I^+}{N}$, the functions $Y_N^\pm$ and $Y_N$ are close in uniform distance, which is the following lemma. 

\begin{lemma}\label{lem_Y_close}
 When $N$ is large enough, if $(\mathcal{B}_2)^c$, $(\mathcal{B}_4^-)^c$ and $(\mathcal{B}_4^+)^c$ occur, then if $I^+ < (|x|+2\theta)N$ then for any $y\in[\frac{(-(|x|+2\theta)N) \vee I^-}{N},\frac{((|x|+2\theta)N) \wedge I^+}{N}]$ we have $|Y_N^\pm(y)-Y_N(y)| \leq N^{-1/12}$, while if $I^+ \geq (|x|+2\theta)N$ we have $|Y_N^\pm(y)-Y_N(y)| \leq N^{-1/12}$ for $y\in[\frac{(-(|x|+2\theta)N) \vee I^-}{N},\frac{((|x|+2\theta)N) \wedge I^+}{N})$. 
\end{lemma}

\begin{proof}[Proof of Lemma \ref{lem_Y_close}.]
 Writing down the proof is only a technical matter, as the meaning of $(\mathcal{B}_4^\pm)^c$ is that the local times are close to the process formed from the random variables of the coupling. $(\mathcal{B}_2)^c$ is there to ensure that the difference terms that appear will be small. We only spell out the proof for $Y^-$, as the proof for $Y^+$ is similar. We assume $(\mathcal{B}_2)^c$, $(\mathcal{B}_4^-)^c$ and $(\mathcal{B}_4^+)^c$. Then if $y \in [\frac{\chi(N)}{N},\frac{((|x|+2\theta)N) \wedge I^+}{N}]$ (if $I^+\geq(|x|+2\theta)N$ we exclude the case $y=\frac{((|x|+2\theta)N) \wedge I^+}{N}$) we have $y \in [\frac{\chi(N)}{N},|x|+2\theta)$, so $|Y_N^-(y)-Y_N(y)| = \frac{1}{\sqrt{N}}|\ell^-(T_N,\lfloor N y\rfloor)-N(\frac{|x|-|y|}{2}+\theta)_+-\sum_{i=\chi(N)}^{\lfloor N y\rfloor-1}\zeta_i|$, thus by \eqref{eq_local_times} we obtain the following: 
 \[
 |Y_N^-(y)-Y_N(y)| = \frac{1}{\sqrt{N}}\left|\lfloor N \theta\rfloor-\mathds{1}_{\{\iota=+\}}+\sum_{i=\chi(N)}^{\lfloor N y\rfloor-1}\eta_{i,+}(\ell^-(T_N,i))-N\left(\frac{|x|-|y|}{2}+\theta\right)_+-\sum_{i=\chi(N)}^{\lfloor N y\rfloor-1}\zeta_i\right| 
 \]
 \[
 \leq \frac{1}{\sqrt{N}}\left|\sum_{i=\chi(N)}^{\lfloor N y\rfloor-1}\eta_{i,+}(\ell^-(T_N,i))+\frac{\lfloor N y\rfloor-\chi(N)}{2}-\sum_{i=\chi(N)}^{\lfloor N y\rfloor-1}\zeta_i\right|+\frac{3}{\sqrt{N}} 
 \]
 \[
 = \frac{1}{\sqrt{N}}\left|\sum_{i=\chi(N)}^{\lfloor N y\rfloor-1}(\eta_{i,+}(\ell^-(T_N,i))+1/2)-\sum_{i=\chi(N)}^{\lfloor N y\rfloor-1}\zeta_i\right|+\frac{3}{\sqrt{N}}.
 \]
 Now, $y \in [\frac{\chi(N)}{N},\frac{((|x|+2\theta)N) \wedge I^+}{N}]$ implies $\lfloor N y\rfloor \in \{\chi(N),...,I^+\}$, thus $(\mathcal{B}_4^+)^c$ yields $|Y_N^-(y)-Y_N(y)| \leq \frac{1}{\sqrt{N}}N^{1/3}+\frac{3}{\sqrt{N}} \leq N^{-1/12}$ when $N$ is large enough. We now consider the case $y \in [\frac{(-(|x|+2\theta)N) \vee I^-}{N},\frac{\chi(N)}{N})$. Then $y \in [-|x|-2\theta,\frac{\chi(N)}{N})$, hence $|Y_N^-(y)-Y_N(y)| = \frac{1}{\sqrt{N}}|\ell^-(T_N,\lfloor N y\rfloor)-N(\frac{|x|-|y|}{2}+\theta)_+-\sum_{i=\lfloor N y\rfloor+1}^{\chi(N)-1}\zeta_i|$. Now, \eqref{eq_local_times_2} yields $|\ell^-(T_N,\lfloor N y\rfloor)-\ell^+(T_N,\lfloor N y\rfloor)|=|\eta_{\lfloor N y\rfloor,-}(\ell^+(T_N,\lfloor N y\rfloor))|$, which is smaller than $N^{1/16}+1/2$ thanks to $(\mathcal{B}_2)^c$. We deduce that $|Y_N^-(y)-Y_N(y)| \leq \frac{1}{\sqrt{N}}|\ell^+(T_N,\lfloor N y\rfloor)-N(\frac{|x|-|y|}{2}+\theta)_+-\sum_{i=\lfloor N y\rfloor+1}^{\chi(N)-1}\zeta_i|+\frac{N^{1/16}+1/2}{\sqrt{N}}$, thus \eqref{eq_local_times} implies 
 \[
 |Y_N^-(y)-Y_N(y)| \leq \frac{1}{\sqrt{N}}\left|\lfloor N \theta\rfloor+\sum_{i=\lfloor N y\rfloor+1}^{\chi(N)-1}(\eta_{i,-}(\ell^+(T_N,i))+\mathds{1}_{\{i > 0\}})-N\left(\frac{|x|-|y|}{2}+\theta\right)_+-\sum_{i=\lfloor N y\rfloor+1}^{\chi(N)-1}\zeta_i\right|+\frac{N^{1/16}+1/2}{\sqrt{N}}
 \]
 \[
 \leq \frac{1}{\sqrt{N}}\left|\sum_{i=\lfloor N y\rfloor+1}^{\chi(N)-1}(\eta_{i,-}(\ell^+(T_N,i))+\mathds{1}_{\{i > 0\}})+\frac{\lfloor N y\rfloor+1-\chi(N)}{2}-\sum_{i=\lfloor N y\rfloor+1}^{\chi(N)-1}\zeta_i\right|+\frac{N^{1/16}+3}{\sqrt{N}}
 \]
 \[
 \leq \frac{1}{\sqrt{N}}\left|\sum_{i=\lfloor N y\rfloor+1}^{\chi(N)-1}(\eta_{i,-}(\ell^+(T_N,i))+1/2)-\sum_{i=\lfloor N y\rfloor+1}^{\chi(N)-1}\zeta_i\right|+\frac{N^{1/16}+3}{\sqrt{N}}.
 \]
  Furthermore, $y \in [\frac{(-(|x|+2\theta)N) \vee I^-}{N},\frac{\chi(N)}{N})$ implies $\lfloor N y\rfloor \in \{I^-,...,\chi(N)-1\}$, hence $(\mathcal{B}_4^-)^c$ yields $|Y_N^-(y)-Y_N(y)| \leq \frac{1}{\sqrt{N}}N^{1/3}+\frac{N^{1/16}+3}{\sqrt{N}} \leq N^{-1/12}$ when $N$ is large enough. Consequently, for any $y\in[\frac{(-(|x|+2\theta)N) \vee I^-}{N},\frac{((|x|+2\theta)N) \wedge I^+}{N}]$ we have $|Y_N^-(y)-Y_N(y)| \leq N^{-1/12}$, which ends the proof of Lemma \ref{lem_Y_close}. 
\end{proof}

We now prove Proposition \ref{prop_distance_param_bis}. Let $a>0$ so that $|(|x|+2\theta)-a|>N^{-1/8}$, we will prove that when $N$ is large enough, $\mathcal{B}^c \subset \{d_{M_1,a}(Y_N^\pm|_{[-a,a]},Y_N|_{[-a,a]}) \leq 2N^{-1/12}\}$, and the threshold for $N$ given by the proof will not depend on the value of $a$. There will be two cases depending on if $a$ is smaller than $|x|+2\theta$ or not. 

\subsection{Case $\mathbf{a\in(0,|x|+2\theta-N^{-1/8})}$}

 This is the easier case. Indeed, the interval $[-a,a]$ will then be contained in $[\frac{(-(|x|+2\theta)N) \vee I^-}{N},\frac{((|x|+2\theta)N) \wedge I^+}{N})$, inside which $Y_N^\pm$ and $Y_N$ are close for the uniform norm by Lemma \ref{lem_Y_close}. We may then define parametric representations $(u_N^\pm,r_N^\pm)$ and $(u_N,r_N)$ of $Y_N^-|_{[-a,a]}$ and $Y_N|_{[-a,a]}$ ``following the graphs of $Y_N^\pm|_{[-a,a]}$ and $Y_N|_{[-a,a]}$ together'' so that $u_N^\pm(t)=u_N(t)$ for all $t\in[0,1]$, and $\|r_N^\pm-r_N\|_\infty \leq \sup_{y\in[-a,a]}|Y_N^\pm(y)-Y_N(y)|$ (an explicit construction of these representations can be found in the first arXiv version of this paper \cite{Mareche2022v1}). We deduce $d_{M_1,a}(Y_N^\pm|_{[-a,a]},Y_N|_{[-a,a]}) \leq \sup_{y\in[-a,a]}|Y_N^\pm(y)-Y_N(y)|$. Moreover, if $\mathcal{B}^c$ occurs, since $a\in(0,|x|+2\theta-N^{-1/8})$, for any $y \in [-a,a]$ we have $y \in (-|x|-2\theta+N^{-1/8},|x|+2\theta-N^{-1/8})$ thus $-(|x|+2\theta)N+N^{3/4} \leq N y \leq (|x|+2\theta)N-N^{3/4}$, hence $I^- < N y < I^+$, hence $y\in(\frac{(-(|x|+2\theta)N) \vee I^-}{N},\frac{((|x|+2\theta)N) \wedge I^+}{N})$, so by Lemma \ref{lem_Y_close} we have $|Y_N^\pm(y)-Y_N(y)| \leq N^{-1/12}$. Consequently, if $\mathcal{B}^c$ occurs, $d_{M_1,a}(Y_N^\pm|_{[-a,a]},Y_N|_{[-a,a]}) \leq N^{-1/12}$.

 \subsection{Case $\mathbf{a>|x|+2\theta+N^{-1/8}}$}
 
This is the harder case, as we have to deal with what happens around $|x|+2\theta$ and $-|x|-2\theta$. We only write down the proof for $Y_N^-$, since the proof for $Y_N^+$ is similar (one may remember that \eqref{eq_local_times_2} allows to bound the $\ell^-(T_N,i)-\ell^+(T_N,i)$ when $(\mathcal{B}_2)^c$ occurs, hence when $\mathcal{B}^c$ occurs). Once again, we will define parametric representations $(u_N^-,r_N^-)$ and $(u_N,r_N)$ of $Y_N^-|_{[-a,a]}$ and $Y_N|_{[-a,a]}$. The definition will depend on whether $I^+ \leq \lfloor(|x|+2\theta)N\rfloor$ or not, and also on whether $I^- \geq -\lfloor(|x|+2\theta)N\rfloor$ or not. We explain it for abscissas in $[0,a]$ depending on whether $I^+ \leq \lfloor(|x|+2\theta)N\rfloor$ or not; the construction for abscissas in $[-a,0]$ are similar depending on whether $I^- \geq -\lfloor(|x|+2\theta)N\rfloor$ or not. We first assume $I^+ \leq \lfloor(|x|+2\theta)N\rfloor$. Between $0$ and $\frac{I^+}{N}$, the parametric representations will be, as in the case $a\in(0,|x|+2\theta-N^{-1/8})$, following the completed graphs of $Y_N^-$ and $Y_N$ in parallel (see Figure \ref{fig_I+small}(a)). The next step, once $(u_N^-,r_N^-)$ reached $(\frac{I^+}{N},Y_N^-(\frac{I^+}{N}))$, is to freeze it there while $(u_N,r_N)$ follows the graph of $Y_N$ from $(\frac{I^+}{N},Y_N(\frac{I^+}{N}))$ to $(|x|+2\theta,Y_N((|x|+2\theta))^-)$ (see Figure \ref{fig_I+small}(b)). For $y \geq \frac{I^+}{N}$ we have $\ell^-(T_N,\lfloor N y\rfloor)=0$ (see Lemma \ref{lem_0_after_I}) thus $Y_N(y)=-N(\frac{|x|-|y|}{2}+\theta)_+$, hence $Y_N^-:[\frac{I^+}{N},|x|+2\theta] \mapsto \mathds{R}$ is affine. Therefore, the following step is to move at the same time $(u_N^-,r_N^-)$ from $(\frac{I^+}{N},Y_N^-(\frac{I^+}{N}))$ to $(|x|+2\theta,Y_N^-(|x|+2\theta))=(|x|+2\theta,0)$ and $(u_N,r_N)$ from $(|x|+2\theta,Y_N((|x|+2\theta)^-))$ to $(|x|+2\theta,0)$ (see Figure \ref{fig_I+small}(c)), and the two parametric representations will remain close. After this step, both parametric representations are at $(|x|+2\theta,0)$, and they will go together to $(a,0)$ (see Figure \ref{fig_I+small}(d)). We now assume $I^+ > \lfloor(|x|+2\theta)N\rfloor$. We also assume $\frac{I^+}{N} \leq a$ (if $\frac{I^+}{N} > a$, we may choose anything for $(u_N^-,r_N^-)$, $(u_N,r_N)$; it will not happen if $\mathcal{B}^c$ occurs). Between $0$ and $|x|+2\theta$, the parametric representations will follow the completed graphs of $Y_N^-$ and $Y_N$ in parallel (see Figure \ref{fig_I+large}(a)). Once abscissa $|x|+2\theta$ is reached, the next step is to move $(u_N^-,r_N^-)$ from $(|x|+2\theta,Y_N^-(|x|+2\theta))$ to $(\frac{I^+}{N},Y_N^-(\frac{I^+}{N}))$, which is $(\frac{I^+}{N},0)$, and to move at the same time $(u_N,r_N)$ from $(|x|+2\theta,Y_N(|x|+2\theta))$ to $(|x|+2\theta,0)$ (see Figure \ref{fig_I+large}(b)). We will prove the two representations are close by controlling the local times. At the next step we freeze $(u_N^-,r_N^-)$ at $(\frac{I^+}{N},0)$ while $(u_N,r_N)$ goes from $(|x|+2\theta,0)$ to $(\frac{I^+}{N},0)$ (see Figure \ref{fig_I+large}(c)). After this step, both parametric representations are at $(\frac{I^+}{N},0)$, and they will go together from $(\frac{I^+}{N},0)$ to $(a,0)$ (see Figure \ref{fig_I+large}(d)). Again, a more rigorous definition of the parametric representations is available in the first arXiv version of this paper \cite{Mareche2022v1}.
 
 \begin{figure}
 \begin{center}
\parbox{0.24\textwidth}{
 \begin{tikzpicture}[scale=0.25]
  \draw[->,>=latex] (0,-14)--(0,2);
  \draw (-1,0)--(1,0);
  \draw (0.75,-0.5)--(1.25,0.5);
  \draw (1.75,-0.5)--(2.25,0.5);
  \draw[->,>=latex] (2,0)--(17,0);
  \draw[ultra thick,red] (2,-10)--(3,-8);
  \draw[ultra thick,red] (3,-8)--(3,-11);
  \draw[ultra thick,red] (3,-11)--(4,-9);
  \draw[ultra thick,red] (4,-9)--(4,-10);
  \draw[ultra thick,red] (4,-10)--(5,-8);
  \draw[ultra thick,red] (5,-8)--(5,-10);
  \draw[ultra thick,red] (5,-10)--(6,-8);
  \draw[ultra thick,red] (6,-8)--(6,-9);
  \draw[ultra thick,red] (6,-9)--(7,-7);
  \draw[ultra thick,red] (7,-7)--(7,-8);
  \draw (7,-8)--(11,0);
  \draw (11,0)--(16,0);
  \draw[ultra thick,red] (2,-12)--(3,-12);
  \draw[ultra thick,red] (3,-12)--(3,-13);
  \draw[ultra thick,red] (3,-13)--(4,-13);
  \draw[ultra thick,red] (4,-13)--(4,-12.5);
  \draw[ultra thick,red] (4,-12.5)--(5,-12.5);
  \draw[ultra thick,red] (5,-12.5)--(5,-11.5);
  \draw[ultra thick,red] (5,-11.5)--(6,-11.5);
  \draw[ultra thick,red] (6,-11.5)--(6,-12);
  \draw[ultra thick,red] (6,-12)--(7,-12);
  \draw[ultra thick,red] (7,-12)--(7,-10.5);
  \draw (7,-10.5)--(8,-10.5);
  \draw[dotted] (8,-10.5)--(8,-11.5);
  \draw (8,-11.5)--(9,-11.5);
  \draw[dotted] (9,-11.5)--(9,-12);
  \draw (9,-12)--(10,-12);
  \draw[dotted] (10,-12)--(10,-11.5);
  \draw (10,-11.5)--(11,-11.5);
  \draw[dotted] (11,-11.5)--(11,0);
  \draw (0,0) node[above left] {$0$};
  \draw (7,0) node {$\mid$} node[above] {\small$\frac{I^+}{N}$} ;
  \draw (11,0) node {$\mid$} node[above] {\small$|x|+2\theta$} ;
  \draw (16,0) node {$\mid$} node[above] {$a$} ;
  \draw (4.5,-6.5) node {$Y_N^-$};
  \draw (6.5,-12.5) node[below] {$Y_N$};
 \end{tikzpicture}
 
 \center (a)
 }
 \parbox{0.24\textwidth}{
 \begin{tikzpicture}[scale=0.25]
  \draw[->,>=latex] (0,-14)--(0,2);
  \draw (-1,0)--(1,0);
  \draw (0.75,-0.5)--(1.25,0.5);
  \draw (1.75,-0.5)--(2.25,0.5);
  \draw[->,>=latex] (2,0)--(17,0);
  \draw (2,-10)--(3,-8);
  \draw[dotted] (3,-8)--(3,-11);
  \draw (3,-11)--(4,-9);
  \draw[dotted] (4,-9)--(4,-10);
  \draw (4,-10)--(5,-8);
  \draw[dotted] (5,-8)--(5,-10);
  \draw (5,-10)--(6,-8);
  \draw[dotted] (6,-8)--(6,-9);
  \draw (6,-9)--(7,-7);
  \draw[dotted] (7,-7)--(7,-8);
  \fill[red] (7,-8) circle(0.3) ;
  \draw (7,-8)--(11,0);
  \draw (11,0)--(16,0);
  \draw (2,-12)--(3,-12);
  \draw[dotted] (3,-12)--(3,-13);
  \draw (3,-13)--(4,-13);
  \draw[dotted] (4,-13)--(4,-12.5);
  \draw (4,-12.5)--(5,-12.5);
  \draw[dotted] (5,-12.5)--(5,-11.5);
  \draw (5,-11.5)--(6,-11.5);
  \draw[dotted] (6,-11.5)--(6,-12);
  \draw (6,-12)--(7,-12);
  \draw[dotted] (7,-12)--(7,-10.5);
  \draw[ultra thick,red] (7,-10.5)--(8,-10.5);
  \draw[ultra thick,red] (8,-10.5)--(8,-11.5);
  \draw[ultra thick,red] (8,-11.5)--(9,-11.5);
  \draw[ultra thick,red] (9,-11.5)--(9,-12);
  \draw[ultra thick,red] (9,-12)--(10,-12);
  \draw[ultra thick,red] (10,-12)--(10,-11.5);
  \draw[ultra thick,red] (10,-11.5)--(11,-11.5);
  \draw[dotted] (11,-11.5)--(11,0);
  \draw (0,0) node[above left] {$0$};
  \draw (7,0) node {$\mid$} node[above] {\small$\frac{I^+}{N}$} ;
  \draw (11,0) node {$\mid$} node[above] {\small$|x|+2\theta$} ;
  \draw (16,0) node {$\mid$} node[above] {$a$} ;
  \draw (4.5,-6.5) node {$Y_N^-$};
  \draw (6.5,-12.5) node[below] {$Y_N$};
 \end{tikzpicture}
 
 \center (b)
 }
 \parbox{0.24\textwidth}{
 \begin{tikzpicture}[scale=0.25]
  \draw[->,>=latex] (0,-14)--(0,2);
  \draw (-1,0)--(1,0);
  \draw (0.75,-0.5)--(1.25,0.5);
  \draw (1.75,-0.5)--(2.25,0.5);
  \draw[->,>=latex] (2,0)--(17,0);
  \draw (2,-10)--(3,-8);
  \draw[dotted] (3,-8)--(3,-11);
  \draw (3,-11)--(4,-9);
  \draw[dotted] (4,-9)--(4,-10);
  \draw (4,-10)--(5,-8);
  \draw[dotted] (5,-8)--(5,-10);
  \draw (5,-10)--(6,-8);
  \draw[dotted] (6,-8)--(6,-9);
  \draw (6,-9)--(7,-7);
  \draw[dotted] (7,-7)--(7,-8);
  \draw[ultra thick,red] (7,-8)--(11,0);
  \draw (11,0)--(16,0);
  \draw (2,-12)--(3,-12);
  \draw[dotted] (3,-12)--(3,-13);
  \draw (3,-13)--(4,-13);
  \draw[dotted] (4,-13)--(4,-12.5);
  \draw (4,-12.5)--(5,-12.5);
  \draw[dotted] (5,-12.5)--(5,-11.5);
  \draw (5,-11.5)--(6,-11.5);
  \draw[dotted] (6,-11.5)--(6,-12);
  \draw (6,-12)--(7,-12);
  \draw[dotted] (7,-12)--(7,-10.5);
  \draw (7,-10.5)--(8,-10.5);
  \draw[dotted] (8,-10.5)--(8,-11.5);
  \draw (8,-11.5)--(9,-11.5);
  \draw[dotted] (9,-11.5)--(9,-12);
  \draw (9,-12)--(10,-12);
  \draw[dotted] (10,-12)--(10,-11.5);
  \draw (10,-11.5)--(11,-11.5);
  \draw[ultra thick,red] (11,-11.5)--(11,0);
  \draw (0,0) node[above left] {$0$};
  \draw (7,0) node {$\mid$} node[above] {\small$\frac{I^+}{N}$} ;
  \draw (11,0) node {$\mid$} node[above] {\small$|x|+2\theta$} ;
  \draw (16,0) node {$\mid$} node[above] {$a$} ;
  \draw (4.5,-6.5) node {$Y_N^-$};
  \draw (6.5,-12.5) node[below] {$Y_N$};
 \end{tikzpicture}
 
 \center (c)
 }
 \parbox{0.24\textwidth}{
 \begin{tikzpicture}[scale=0.25]
  \draw[->,>=latex] (0,-14)--(0,2);
  \draw (-1,0)--(1,0);
  \draw (0.75,-0.5)--(1.25,0.5);
  \draw (1.75,-0.5)--(2.25,0.5);
  \draw[->,>=latex] (2,0)--(17,0);
  \draw (2,-10)--(3,-8);
  \draw[dotted] (3,-8)--(3,-11);
  \draw (3,-11)--(4,-9);
  \draw[dotted] (4,-9)--(4,-10);
  \draw (4,-10)--(5,-8);
  \draw[dotted] (5,-8)--(5,-10);
  \draw (5,-10)--(6,-8);
  \draw[dotted] (6,-8)--(6,-9);
  \draw (6,-9)--(7,-7);
  \draw[dotted] (7,-7)--(7,-8);
  \draw (7,-8)--(11,0);
  \draw[ultra thick,red] (11,0)--(16,0);
  \draw (2,-12)--(3,-12);
  \draw[dotted] (3,-12)--(3,-13);
  \draw (3,-13)--(4,-13);
  \draw[dotted] (4,-13)--(4,-12.5);
  \draw (4,-12.5)--(5,-12.5);
  \draw[dotted] (5,-12.5)--(5,-11.5);
  \draw (5,-11.5)--(6,-11.5);
  \draw[dotted] (6,-11.5)--(6,-12);
  \draw (6,-12)--(7,-12);
  \draw[dotted] (7,-12)--(7,-10.5);
  \draw (7,-10.5)--(8,-10.5);
  \draw[dotted] (8,-10.5)--(8,-11.5);
  \draw (8,-11.5)--(9,-11.5);
  \draw[dotted] (9,-11.5)--(9,-12);
  \draw (9,-12)--(10,-12);
  \draw[dotted] (10,-12)--(10,-11.5);
  \draw (10,-11.5)--(11,-11.5);
  \draw[dotted] (11,-11.5)--(11,0);
  \draw (0,0) node[above left] {$0$};
  \draw (7,0) node {$\mid$} node[above] {\small$\frac{I^+}{N}$} ;
  \draw (11,0) node {$\mid$} node[above] {\small$|x|+2\theta$} ;
  \draw (16,0) node {$\mid$} node[above] {$a$} ;
  \draw (4.5,-6.5) node {$Y_N^-$};
  \draw (6.5,-12.5) node[below] {$Y_N$};
 \end{tikzpicture}
 
 \center (d)
 }
 \end{center}
 \caption{The successive steps of the parametric representations of $Y_N^-|_{[-a,a]}$ and $Y_N|_{[-a,a]}$ if $I^+ \leq \lfloor(|x|+2\theta)N\rfloor$. At each step, the parts of the graphs the parametric representations travel through are thickened.}\label{fig_I+small}
 \end{figure}
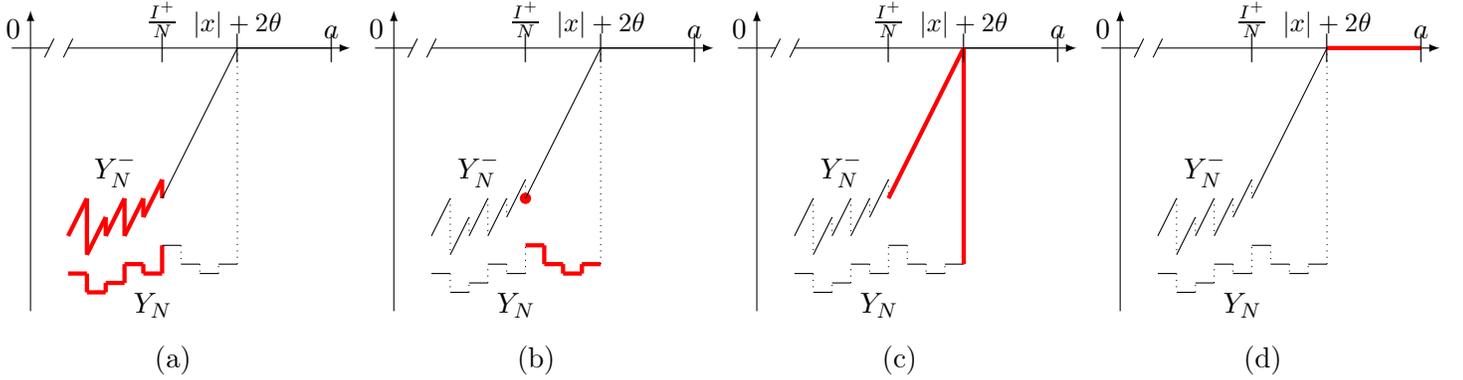
 
 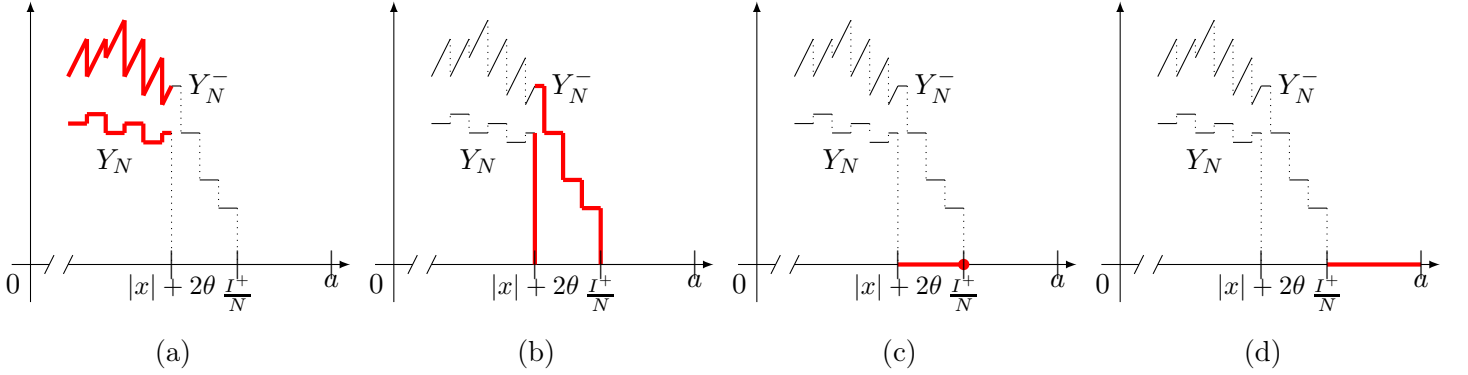
\begin{figure}
  \begin{center}
  \parbox{0.24\textwidth}{
   \begin{tikzpicture}[scale=0.25]
  \draw[->,>=latex] (0,-2)--(0,14);
  \draw (-1,0)--(1,0);
  \draw (0.75,-0.5)--(1.25,0.5);
  \draw (1.75,-0.5)--(2.25,0.5);
  \draw[->,>=latex] (2,0)--(17,0);
  \draw[ultra thick,red] (2,10)--(3,12);
  \draw[ultra thick,red] (3,12)--(3,10);
  \draw[ultra thick,red] (3,10)--(4,12);
  \draw[ultra thick,red] (4,12)--(4,11);
  \draw[ultra thick,red] (4,11)--(5,13);
  \draw[ultra thick,red] (5,13)--(5,10);
  \draw[ultra thick,red] (5,10)--(6,12);
  \draw[ultra thick,red] (6,12)--(6,9);
  \draw[ultra thick,red] (6,9)--(7,11);
  \draw[ultra thick,red] (7,11)--(7,8.5);
  \draw[ultra thick,red] (7,8.5)--(7.5,9.5);
  \draw (7.5,9.5)--(8,9.5);
  \draw[dotted] (8,9.5)--(8,7);
  \draw (8,7)--(9,7);
  \draw[dotted] (9,7)--(9,4.5);
  \draw (9,4.5)--(10,4.5);
  \draw[dotted] (10,4.5)--(10,3);
  \draw (10,3)--(11,3);
  \draw[dotted] (11,3)--(11,0);
  \draw (11,0)--(16,0);
  \draw[ultra thick,red] (2,7.5)--(3,7.5);
  \draw[ultra thick,red] (3,7.5)--(3,8);
  \draw[ultra thick,red] (3,8)--(4,8);
  \draw[ultra thick,red] (4,8)--(4,7);
  \draw[ultra thick,red] (4,7)--(5,7);
  \draw[ultra thick,red] (5,7)--(5,7.5);
  \draw[ultra thick,red] (5,7.5)--(6,7.5);
  \draw[ultra thick,red] (6,7.5)--(6,6.5);
  \draw[ultra thick,red] (6,6.5)--(7,6.5);
  \draw[ultra thick,red] (7,6.5)--(7,7);
  \draw[ultra thick,red] (7,7)--(7.5,7);
  \draw[dotted] (7.5,7)--(7.5,0);
  \draw (0,0) node[below left] {$0$};
  \draw (7.5,0) node {$\mid$} node[below] {\small$|x|+2\theta$} ;
  \draw (11,0) node {$\mid$} node[below] {\small$\frac{I^+}{N}$} ;
  \draw (16,0) node {$\mid$} node[below] {$a$} ;
   \draw (9.5,9.5) node {$Y_N^-$};
  \draw (4.5,5.5) node {$Y_N$};
   \end{tikzpicture}
   
   \center (a)
   }
   \parbox{0.24\textwidth}{
   \begin{tikzpicture}[scale=0.25]
    \draw[->,>=latex] (0,-2)--(0,14);
  \draw (-1,0)--(1,0);
  \draw (0.75,-0.5)--(1.25,0.5);
  \draw (1.75,-0.5)--(2.25,0.5);
  \draw[->,>=latex] (2,0)--(17,0);
  \draw (2,10)--(3,12);
  \draw[dotted] (3,12)--(3,10);
  \draw (3,10)--(4,12);
  \draw[dotted] (4,12)--(4,11);
  \draw (4,11)--(5,13);
  \draw[dotted] (5,13)--(5,10);
  \draw (5,10)--(6,12);
  \draw[dotted] (6,12)--(6,9);
  \draw (6,9)--(7,11);
  \draw[dotted] (7,11)--(7,8.5);
  \draw (7,8.5)--(7.5,9.5);
  \draw[ultra thick,red] (7.5,9.5)--(8,9.5);
  \draw[ultra thick,red] (8,9.5)--(8,7);
  \draw[ultra thick,red] (8,7)--(9,7);
  \draw[ultra thick,red] (9,7)--(9,4.5);
  \draw[ultra thick,red] (9,4.5)--(10,4.5);
  \draw[ultra thick,red] (10,4.5)--(10,3);
  \draw[ultra thick,red] (10,3)--(11,3);
  \draw[ultra thick,red] (11,3)--(11,0);
  \draw (11,0)--(16,0);
  \draw (2,7.5)--(3,7.5);
  \draw[dotted] (3,7.5)--(3,8);
  \draw (3,8)--(4,8);
  \draw[dotted] (4,8)--(4,7);
  \draw (4,7)--(5,7);
  \draw[dotted] (5,7)--(5,7.5);
  \draw (5,7.5)--(6,7.5);
  \draw[dotted] (6,7.5)--(6,6.5);
  \draw (6,6.5)--(7,6.5);
  \draw[dotted] (7,6.5)--(7,7);
  \draw (7,7)--(7.5,7);
  \draw[ultra thick,red] (7.5,7)--(7.5,0);
  \draw (0,0) node[below left] {$0$};
  \draw (7.5,0) node {$\mid$} node[below] {\small$|x|+2\theta$} ;
  \draw (11,0) node {$\mid$} node[below] {\small$\frac{I^+}{N}$} ;
  \draw (16,0) node {$\mid$} node[below] {$a$} ;
   \draw (9.5,9.5) node {$Y_N^-$};
  \draw (4.5,5.5) node {$Y_N$};
   \end{tikzpicture}
   
   \center (b)
   }
   \parbox{0.24\textwidth}{
   \begin{tikzpicture}[scale=0.25]
    \draw[->,>=latex] (0,-2)--(0,14);
  \draw (-1,0)--(1,0);
  \draw (0.75,-0.5)--(1.25,0.5);
  \draw (1.75,-0.5)--(2.25,0.5);
  \draw[->,>=latex] (2,0)--(17,0);
  \draw (2,10)--(3,12);
  \draw[dotted] (3,12)--(3,10);
  \draw (3,10)--(4,12);
  \draw[dotted] (4,12)--(4,11);
  \draw (4,11)--(5,13);
  \draw[dotted] (5,13)--(5,10);
  \draw (5,10)--(6,12);
  \draw[dotted] (6,12)--(6,9);
  \draw (6,9)--(7,11);
  \draw[dotted] (7,11)--(7,8.5);
  \draw (7,8.5)--(7.5,9.5);
  \draw (7.5,9.5)--(8,9.5);
  \draw[dotted] (8,9.5)--(8,7);
  \draw (8,7)--(9,7);
  \draw[dotted] (9,7)--(9,4.5);
  \draw (9,4.5)--(10,4.5);
  \draw[dotted] (10,4.5)--(10,3);
  \draw (10,3)--(11,3);
  \draw[dotted] (11,3)--(11,0);
  \draw (11,0)--(16,0);
  \fill[red] (11,0) circle(0.3) ;
  \draw (2,7.5)--(3,7.5);
  \draw[dotted] (3,7.5)--(3,8);
  \draw (3,8)--(4,8);
  \draw[dotted] (4,8)--(4,7);
  \draw (4,7)--(5,7);
  \draw[dotted] (5,7)--(5,7.5);
  \draw (5,7.5)--(6,7.5);
  \draw[dotted] (6,7.5)--(6,7);
  \draw (6,6.5)--(7,6.5);
  \draw[dotted] (7,6.5)--(7,7);
  \draw (7,7)--(7.5,7);
  \draw[dotted] (7.5,7)--(7.5,0);
  \draw[ultra thick,red] (7.5,0)--(11,0);
  \draw (0,0) node[below left] {$0$};
  \draw (7.5,0) node {$\mid$} node[below] {\small$|x|+2\theta$} ;
  \draw (11,0) node {$\mid$} node[below] {\small$\frac{I^+}{N}$} ;
  \draw (16,0) node {$\mid$} node[below] {$a$} ;
   \draw (9.5,9.5) node {$Y_N^-$};
  \draw (4.5,5.5) node {$Y_N$};
   \end{tikzpicture}
   
   \center (c)
   }
   \parbox{0.24\textwidth}{
   \begin{tikzpicture}[scale=0.25]
    \draw[->,>=latex] (0,-2)--(0,14);
  \draw (-1,0)--(1,0);
  \draw (0.75,-0.5)--(1.25,0.5);
  \draw (1.75,-0.5)--(2.25,0.5);
  \draw[->,>=latex] (2,0)--(17,0);
  \draw (2,10)--(3,12);
  \draw[dotted] (3,12)--(3,10);
  \draw (3,10)--(4,12);
  \draw[dotted] (4,12)--(4,11);
  \draw (4,11)--(5,13);
  \draw[dotted] (5,13)--(5,10);
  \draw (5,10)--(6,12);
  \draw[dotted] (6,12)--(6,9);
  \draw (6,9)--(7,11);
  \draw[dotted] (7,11)--(7,8.5);
  \draw (7,8.5)--(7.5,9.5);
  \draw (7.5,9.5)--(8,9.5);
  \draw[dotted] (8,9.5)--(8,7);
  \draw (8,7)--(9,7);
  \draw[dotted] (9,7)--(9,4.5);
  \draw (9,4.5)--(10,4.5);
  \draw[dotted] (10,4.5)--(10,3);
  \draw (10,3)--(11,3);
  \draw[dotted] (11,3)--(11,0);
  \draw[ultra thick,red] (11,0)--(16,0);
  \draw (2,7.5)--(3,7.5);
  \draw[dotted] (3,7.5)--(3,8);
  \draw (3,8)--(4,8);
  \draw[dotted] (4,8)--(4,7);
  \draw (4,7)--(5,7);
  \draw[dotted] (5,7)--(5,7.5);
  \draw (5,7.5)--(6,7.5);
  \draw[dotted] (6,7.5)--(6,6.5);
  \draw (6,6.5)--(7,6.5);
  \draw[dotted] (7,6.5)--(7,7);
  \draw (7,7)--(7.5,7);
  \draw[dotted] (7.5,7)--(7.5,0);
  \draw (0,0) node[below left] {$0$};
  \draw (7.5,0) node {$\mid$} node[below] {\small$|x|+2\theta$} ;
  \draw (11,0) node {$\mid$} node[below] {\small$\frac{I^+}{N}$} ;
  \draw (16,0) node {$\mid$} node[below] {$a$} ;
   \draw (9.5,9.5) node {$Y_N^-$};
  \draw (4.5,5.5) node {$Y_N$};
   \end{tikzpicture}
   
   \center (d)
   }
  \end{center}
\caption{The successive steps of the parametric representations of $Y_N^-|_{[-a,a]}$ and $Y_N|_{[-a,a]}$ if $I^+ > \lfloor(|x|+2\theta)N\rfloor$. At each step, the parts of the graphs the parametric representations travel through are thickened.}\label{fig_I+large}
 \end{figure}

We can now bound the Skorohod $M_1$ distance between $Y_N^-|_{[-a,a]}$ and $Y_N|_{[-a,a]}$. From its definition, we have $d_{M_1,a}(Y_N^-|_{[-a,a]},Y_N|_{[-a,a]}) \leq \max(\|u_N^--u_N\|_\infty,\|r_N^--r_N\|_\infty)$, hence we only have to prove $\mathcal{B}^c \subset \{\max(\|u_N^--u_N\|_\infty,\|r_N^--r_N\|_\infty) \leq 2N^{-1/12}\}$ when $N$ is large enough. We are going to break down $\{\max(\|u_N^--u_N\|_\infty,\|r_N^--r_N\|_\infty) \leq 2N^{-1/12}\}$ into several events. We may write 
\[
 \{\max(\|u_N^--u_N\|_\infty,\|r_N^--r_N\|_\infty) \leq 2N^{-1/12}\} 
 \]
 \begin{align*}
 = &\left\{\text{between }\frac{(-(|x|+2\theta)N) \vee I^-}{N}\text{ and }\frac{((|x|+2\theta)N) \wedge I^+}{N},\|u_N^--u_N\|_\infty,\|r_N^--r_N\|_\infty \leq 2N^{-1/12}\right\} \\
 &\cap \left\{\text{between }\frac{((|x|+2\theta)N)\wedge I^+}{N}\text{ and }a,\|u_N^--u_N\|_\infty,\|r_N^--r_N\|_\infty \leq 2N^{-1/12}\right\} \\
 &\cap \left\{\text{between }-a\text{ and }\frac{(-(|x|+2\theta)N) \vee I^-}{N},\|u_N^--u_N\|_\infty,\|r_N^--r_N\|_\infty \leq 2N^{-1/12}\right\}.
\end{align*}
Consequently, to prove that $\mathcal{B}^c \subset \{\max(\|u_N^--u_N\|_\infty,\|r_N^--r_N\|_\infty) \leq 2N^{-1/12}\}$ when $N$ is large enough and thus end the proof of Proposition \ref{prop_distance_param_bis}, we only have to prove the following claims.

\begin{claim}\label{claim_param_center}
 $\mathcal{B}^c \subset \{$between $\frac{(-(|x|+2\theta)N) \vee I^-}{N}$ and $\frac{((|x|+2\theta)N) \wedge I^+}{N},\|u_N^--u_N\|_\infty,\|r_N^--r_N\|_\infty \leq 2N^{-1/12}\}$ when $N$ is large enough. 
\end{claim}

\begin{claim}\label{claim_param_early_stop}
 $\mathcal{B}^c \cap \{I^+ \leq \lfloor(|x|+2\theta)N\rfloor\} \subset \{$between $\frac{((|x|+2\theta)N) \wedge I^+}{N}$ and $a$, $\|u_N^--u_N\|_\infty,\|r_N^--r_N\|_\infty \leq 2N^{-1/12}\}$ and $\mathcal{B}^c \cap \{I^- \geq -\lfloor(|x|+2\theta)N\rfloor\} \subset \{$between $-a$ and $\frac{(-(|x|+2\theta)N) \vee I^-}{N}$, $\|u_N^--u_N\|_\infty,\|r_N^--r_N\|_\infty \leq 2N^{-1/12}\})$ when $N$ is large enough. 
\end{claim}

\begin{claim}\label{claim_param_late_stop}
 $\mathcal{B}^c \cap \{I^+ > \lfloor(|x|+2\theta)N\rfloor\} \subset \{$between $\frac{((|x|+2\theta)N) \wedge I^+}{N}$ and $a$, $\|u_N^--u_N\|_\infty,\|r_N^--r_N\|_\infty \leq 2N^{-1/12}\}$ and $\mathcal{B}^c \cap \{I^- < -\lfloor(|x|+2\theta)N\rfloor\} \subset \{$between $-a$ and $\frac{(-(|x|+2\theta)N) \vee I^-}{N}$, $\|u_N^--u_N\|_\infty,\|r_N^--r_N\|_\infty \leq 2N^{-1/12}\})$ when $N$ is large enough. 
\end{claim}

We now prove Claims \ref{claim_param_center}, \ref{claim_param_early_stop} and \ref{claim_param_late_stop}.

\begin{proof}[Proof of Claim \ref{claim_param_center}.] 
We assume $\mathcal{B}^c$ occurs. In the part of the parametric representations between $\frac{(-(|x|+2\theta)N) \vee I^-}{N}$ and $\frac{((|x|+2\theta)N) \wedge I^+}{N}$, corresponding to Figures \ref{fig_I+small}(a) and \ref{fig_I+large}(a), we follow the completed graphs of $Y_N^-$ and $Y_N$ in parallel. Therefore we have $u_N^-(t)=u_N(t)$ and $|r_N^-(t)-r_N(t)| \leq \sup\{|Y_N^-(y)-Y_N(y)|:y\in[\frac{(-(|x|+2\theta)N) \vee I^-}{N},\frac{((|x|+2\theta)N) \wedge I^+}{N}]\}$. If $(|x|+2\theta)N$ is not an integer or $I^+<(|x|+2\theta)N$, this is smaller than $N^{-12}$ when $N$ is large enough by Lemma \ref{lem_Y_close}, and we are done. If $(|x|+2\theta)N$ is an integer and $I^+\geq(|x|+2\theta)N$, there is a small complication, since the parametric representations follow the graph of $Y_N$ until $Y_N((|x|+2\theta)^-)$, but should follow the graph of $Y_N^-$ until $Y_N^-(|x|+2\theta)$. The solution is to freeze the representation of $Y_N$ at $Y_N((|x|+2\theta)^-)$ while that of $Y_N^-$ goes from $Y_N^-((|x|+2\theta)^-)$ to $Y_N^-(|x|+2\theta)$. Then, between $\frac{(-(|x|+2\theta)N) \vee I^-}{N}$ and $(|x|+2\theta)^-$ we have $|r_N^-(t)-r_N(t)| \leq \sup\{|Y_N^-(y)-Y_N(y)|:y\in[\frac{(-(|x|+2\theta)N) \vee I^-}{N},(|x|+2\theta)N)\} \leq N^{-1/12}$ by Lemma \ref{lem_Y_close} when $N$ is large enough. Furthermore, when going from $Y_N^-((|x|+2\theta)^-)$ to $Y_N^-(|x|+2\theta)$, we have $|r_N^-(t)-r_N(t)| \leq |Y_N^-((|x|+2\theta)^-)-Y_N((|x|+2\theta)^-)|+|Y_N^-((|x|+2\theta)^-)-Y_N^-(|x|+2\theta)| \leq N^{-1/12}+|Y_N^-((|x|+2\theta)^-)-Y_N^-(|x|+2\theta)|$ when $N$ is large enough. In addition, when $N$ is large enough \eqref{eq_local_times} yields $|Y_N^-(|x|+2\theta)-Y_N^-((|x|+2\theta)^-)|=\frac{1}{\sqrt{N}}|\ell^-(T_N,(|x|+2\theta)N)-\ell^-(T_N,(|x|+2\theta)N-1)|=\frac{1}{\sqrt{N}}|\eta_{(|x|+2\theta)N-1,+}(\ell^-(T_N,(|x|+2\theta)N-1))|\leq \frac{N^{1/16}+1/2}{\sqrt{N}}$ since $(\mathcal{B}_2)^c$ occurs. This yields $|r_N^-(t)-r_N(t)| \leq N^{-1/12}+\frac{N^{1/16}+1/2}{\sqrt{N}} \leq 2N^{-1/12}$ when $N$ is large enough, which ends the proof. 
\end{proof}

\begin{proof}[Proof of Claim \ref{claim_param_early_stop}.] 
This claim deals with the ``right part'' of the parametric representations in the case $I^+ \leq \lfloor(|x|+2\theta)N\rfloor$, and with the ``left part'' in the case $I^- \geq -\lfloor(|x|+2\theta)N\rfloor$, corresponding to Figure \ref{fig_I+small}(b), (c) and (d). The idea of the argument is that in the step of Figure \ref{fig_I+small}(b), the representation of $Y_N$ does not move much horizontally as $\frac{I^+}{N}$ is close to $|x|+2\theta$ by Lemma \ref{lem_control_I}, so it does not have time to move too much vertically. In the step of Figure \ref{fig_I+small}(c), the representations of $Y_N^-$ and $Y_N$ will thus start from points that are close and go to the same point, hence stay close to each other. We now give the rigorous argument. We only spell out the proof for $\mathcal{B}^c \cap \{I^+ \leq \lfloor(|x|+2\theta)N\rfloor\}$, as the other case is similar. Let us assume $\mathcal{B}^c$ occurs and $I^+ \leq \lfloor(|x|+2\theta)N\rfloor$. Firstly, we notice that in the part of the parametric representations corresponding to Figure \ref{fig_I+small}(d) we have $(u_N^-(t),r_N^-(t))=(u_N(t),r_N(t))$, so we only consider the parts corresponding to Figure \ref{fig_I+small}(b) and Figure \ref{fig_I+small}(c). We first consider the case in which $(|x|+2\theta)N$ is not an integer or $I^+ < \lfloor(|x|+2\theta)N\rfloor$. We begin by dealing with $|u_N^-(t)-u_N(t)|$. By the definition of our parametric representations, $|u_N^-(t)-u_N(t)| \leq ||x|+2\theta-\frac{I^+}{N}|$. Furthermore, $\mathcal{B}^c$ occurs, thus we have $|I^+-(|x|+2\theta)N| < N^{3/4}$, hence $|u_N^-(t)-u_N(t)| \leq N^{-1/4}$. We now deal with $|r_N^-(t)-r_N(t)|$. Remembering the definition of our parametric representations, we notice that in the part corresponding to Figure \ref{fig_I+small}(c), $r_N^-$ and $r_N$ are affine functions, so the maximum value of $|r_N^-(t)-r_N(t)|$ on this part is reached either at the beginning or at the end of the part. Moreover, at the end of the part we have $r_N^-(t)=r_N(t)=0$, so the maximum is reached at the beginning. Therefore, if $|r_N^-(t)-r_N(t)| \leq 2N^{-1/12}$ in the part corresponding to Figure \ref{fig_I+small}(b), then $|r_N^-(t)-r_N(t)| \leq 2N^{-1/12}$ in the part corresponding to Figure \ref{fig_I+small}(c), and this ends the proof when $(|x|+2\theta)N$ is not an integer or $I^+ < \lfloor(|x|+2\theta)N\rfloor$. 

We thus have to study the part corresponding to Figure \ref{fig_I+small}(b). By the definition of our parametric representations, $|r_N^-(t)-r_N(t)| \leq \sup\{|Y_N^-(\frac{I^+}{N})-Y_N(y)|:y\in[\frac{I^+}{N},|x|+2\theta)\}$, so it is enough to prove that when $N$ is large enough, $\sup\{|Y_N^-(\frac{I^+}{N})-Y_N(y)|:y\in[\frac{I^+}{N},|x|+2\theta)\} \leq 2N^{-1/12}$. Moreover, for any $y\in[\frac{I^+}{N},|x|+2\theta)$, we have $|Y_N^-(\frac{I^+}{N})-Y_N(y)| \leq |Y_N^-(\frac{I^+}{N})-Y_N(\frac{I^+}{N})|+|Y_N(\frac{I^+}{N})-Y_N(y)|$. Since $\mathcal{B}^c$ occurs, we have that $(\mathcal{B}_2)^c$, $(\mathcal{B}_4^-)^c$ and $(\mathcal{B}_4^+)^c$ occur, hence Lemma \ref{lem_Y_close} implies $|Y_N^-(\frac{I^+}{N})-Y_N(\frac{I^+}{N})| \leq N^{-1/12}$ when $N$ is large enough, thus $|Y_N^-(\frac{I^+}{N})-Y_N(y)| \leq |Y_N(\frac{I^+}{N})-Y_N(y)|+N^{-1/12} \leq \frac{1}{\sqrt{N}}|\sum_{i=I^+}^{\lfloor Ny \rfloor-1}\zeta_i|+N^{-1/12}$. We deduce $\sup\{|Y_N^-(\frac{I^+}{N})-Y_N(y)|:y\in[\frac{I^+}{N},|x|+2\theta)\} \leq \sup\{\frac{1}{\sqrt{N}}|\sum_{i=I^+}^{\lfloor Ny \rfloor-1}\zeta_i|:y\in[\frac{I^+}{N},|x|+2\theta)\}+N^{-1/12}$. Furthermore, $\mathcal{B}^c$ occurs hence $|I^+-(|x|+2\theta)N| < N^{3/4}$, thus $\sup\{|Y_N^-(\frac{I^+}{N})-Y_N(y)|:y\in[\frac{I^+}{N},|x|+2\theta)\} \leq \frac{1}{\sqrt{N}}\max_{\lfloor(|x|+2\theta)N\rfloor-N^{3/4} \leq i_1 \leq i_2 \leq \lfloor (|x|+2\theta)N\rfloor}\frac{1}{\sqrt{N}}|\sum_{i=i_1}^{i_2}\zeta_i|+N^{-1/12}$. Since $\mathcal{B}^c$ occurs, $(\mathcal{B}_3^+)^c$ occurs, hence $\sup\{|Y_N^-(\frac{I^+}{N})-Y_N(y)|:y\in[\frac{I^+}{N},|x|+2\theta)\} \leq \frac{N^{19/48}}{\sqrt{N}}+N^{-1/12}=N^{-5/48}+N^{-1/12} \leq 2 N^{-1/12}$, which is enough.  

We now consider the case in which $(|x|+2\theta)N$ is an integer and $I^+ = \lfloor(|x|+2\theta)N\rfloor$. Then the step of Figure \ref{fig_I+small}(b) does not exist, we only have to deal with that of Figure \ref{fig_I+small}(c), which comes mostly from Lemma \ref{lem_Y_close} as this lemma ensures $Y_N^-((|x|+2\theta)^-)$ and $Y_N((|x|+2\theta)^-)$ are close (we will actually prove they are both close to 0). Since $I^+ = \lfloor(|x|+2\theta)N\rfloor$, we have $u_N^-(t)=u_N(t)$. Moreover, $\ell^-(T_N,\lfloor N(|x|+2\theta) \rfloor)=\ell^-(T_N,I^+)=0$, so $Y_N^-(|x|+2\theta)=0$, hence $r_N^-(t)=0$. Furthermore, $|r_N(t)|\leq |Y_N((|x|+2\theta)^-)|$. Therefore we only have to prove that $|Y_N((|x|+2\theta)^-)| \leq 2N^{-1/12}$ when $N$ is large enough. In addition, $\mathcal{B}^c$ occurs, thus by Lemma \ref{lem_Y_close} we have $|Y_N((|x|+2\theta)^-)-Y_N^-((|x|+2\theta)^-)| \leq N^{-1/12}$ when $N$ is large enough. Moreover by the definition of $Y_N^-$ and by \eqref{eq_local_times}, we have $Y_N^-(|x|+2\theta)=Y_N^-((|x|+2\theta)^-)+\frac{1}{\sqrt{N}}\eta_{(|x|+2\theta)N-1,+}(\ell^-(T_N,(|x|+2\theta)N-1))$, and since $\mathcal{B}$ occurs, $(\mathcal{B}_2)^c$ occurs, hence we get $|Y_N^-(|x|+2\theta)-Y_N^-((|x|+2\theta)^-)|\leq \frac{1}{\sqrt{N}}(N^{1/16}+1/2) \leq N^{-1/4}$. Since $Y_N^-(|x|+2\theta)=0$, this yields $|Y_N^-((|x|+2\theta)^-)|\leq N^{-1/4}$, which yields $|Y_N((|x|+2\theta)^-)| \leq N^{-1/12}+N^{-1/4}< 2N^{-1/12}$, which is enough and ends the proof of Claim \ref{claim_param_early_stop}.
\end{proof}

\begin{proof}[Proof of Claim \ref{claim_param_late_stop}.] 
This claim deals with the ``right part'' of the parametric representations in the case $I^+ > \lfloor(|x|+2\theta)N\rfloor$, and with the ``left part'' in the case $I^- < -\lfloor(|x|+2\theta)N\rfloor$, corresponding to Figure \ref{fig_I+large}(b), (c) and (d). We first give an idea of the argument. The most important part of the proof is to deal with the step corresponding to Figure \ref{fig_I+large}(b). In this step, the function $Y_N^-(y)=\frac{1}{\sqrt{N}}\ell^-(T_N,\lfloor Ny\rfloor)$ evolves as a sum of $\frac{1}{\sqrt{N}}\eta_{j,+}(\ell^-(T_N,j))$ by \eqref{eq_local_times}, which is close to the sum of $\frac{1}{\sqrt{N}}(\zeta_j-\frac{1}{2})$ as $(\mathcal{B}_4^+)^c$ occurs. Since the $\zeta_j$ are i.i.d. with mean 0, the sum of $\frac{1}{\sqrt{N}}\zeta_j$ will be small, and the evolution of $Y_N^-$ will be close to that of a deterministic sum of $-\frac{1}{2\sqrt{N}}$, thus it reaches 0 at constant speed, which is also what our parametric representation of $Y_N$ does. We now give the proof, beginning with the detail of the argument to deal with $\mathcal{B}^c \cap \{I^- < -\lfloor(|x|+2\theta)N\rfloor\}$. Let us assume $\mathcal{B}^c$ occurs and $I^- < -\lfloor(|x|+2\theta)N\rfloor$. We first see that $\frac{I^-}{N} \geq -a$, as since $\mathcal{B}^c$ occurs we have $|I^-+(|x|+2\theta)N| < N^{3/4}$, hence $\frac{I^-}{N} > -|x|-2\theta - N^{-1/4}$, and by assumption $a>|x|+2\theta+N^{-1/8}$, so $-a<-|x|-2\theta-N^{-1/8}<\frac{I^-}{N}$, hence $\frac{I^-}{N} \geq -a$. Moreover, in the part of the parametric representations corresponding to Figure \ref{fig_I+large}(d), we have $(u_N^-(t),r_N^-(t))=(u_N(t),r_N(t))$. We now consider the equivalent of Figure \ref{fig_I+large}(c). Then $r_N^-(t)=r_N(t)=0$, and $|u_N^-(t)-u_N(t)| \leq |\frac{I^-}{N}+(|x|+2\theta)|$, which is strictly smaller than $2N^{-1/12}$ since $|I^-+(|x|+2\theta)N| < N^{3/4}$. It remains to consider the equivalent of Figure \ref{fig_I+large}(b). Then $|u_N^-(t)-u_N(t)| \leq |\frac{I^-}{N}+(|x|+2\theta)|$, which is strictly smaller than $2N^{-1/12}$, so we only have to prove $|r_N^-(t)-r_N(t)| \leq 2N^{-1/12}$. 
 
 We are going to study $\sup_{y\in[\frac{I^-}{N},-|x|-2\theta]}|Y_N^-(y)-Y_N^-(-|x|-2\theta)+\frac{\lfloor (|x|+2\theta)N \rfloor-\lfloor Ny\rfloor}{2\sqrt{N}}|$. Let $y\in[\frac{I^-}{N},-|x|-2\theta]$. By the definition of $Y_N^-$ we have $Y_N^-(y)-Y_N^-(-|x|-2\theta)=\frac{1}{\sqrt{N}}(\ell^-(T_N,\lfloor Ny\rfloor)-\ell^-(T_N,\lfloor -(|x|+2\theta)N \rfloor))$. By \eqref{eq_local_times_2} and since $(\mathcal{B}_2)^c$ occurs (remembering $\lfloor Ny\rfloor \geq I^- \geq -(|x|+2\theta)N-N^{3/4} \geq -\lceil2(|x|+2\theta)N\rceil$), we deduce 
 \[
 \left|Y_N^-(y)-Y_N^-(-|x|-2\theta)-\frac{1}{\sqrt{N}}(\ell^+(T_N,\lfloor Ny\rfloor)-\ell^+(T_N,\lfloor -(|x|+2\theta)N \rfloor))\right|
 \]
 \[
 =\left|\frac{\eta_{\lfloor Ny\rfloor,-}(\ell^+(T_N,\lfloor Ny\rfloor))-\eta_{\lfloor -(|x|+2\theta)N\rfloor,-}(\ell^+(T_N,\lfloor -(|x|+2\theta)N\rfloor))}{\sqrt{N}}\right|\leq \frac{2N^{1/16}}{\sqrt{N}}.
 \]
 In addition, \eqref{eq_local_times} yields the following: 
 \[
 \ell^+(T_N,\lfloor Ny\rfloor)-\ell^+(T_N,\lfloor -(|x|+2\theta)N \rfloor)=\sum_{i=\lfloor Ny\rfloor+1}^{\chi(N)-1}(\eta_{i,-}(\ell^+(T_N,i))+\mathds{1}_{\{i>0\}})-\sum_{i=\lfloor -(|x|+2\theta)N \rfloor+1}^{\chi(N)-1}(\eta_{i,-}(\ell^+(T_N,i))+\mathds{1}_{\{i>0\}})
 \]
 \[
 =\sum_{i=\lfloor Ny\rfloor+1}^{\chi(N)-1}(\eta_{i,-}(\ell^+(T_N,i))+1/2)-\sum_{i=\lfloor -(|x|+2\theta)N \rfloor+1}^{\chi(N)-1}(\eta_{i,-}(\ell^+(T_N,i))+1/2)-\frac{\lfloor -(|x|+2\theta)N \rfloor-\lfloor Ny\rfloor}{2}.
 \]
  Since $(\mathcal{B}_4^-)^c$ occurs, this yields $|\ell^+(T_N,\lfloor Ny\rfloor)-\ell^+(T_N,\lfloor -(|x|+2\theta)N \rfloor)+\frac{\lfloor-(|x|+2\theta)N \rfloor-\lfloor Ny\rfloor}{2}| \leq |\sum_{i=\lfloor Ny\rfloor+1}^{\chi(N)-1}\zeta_i-\sum_{i=\lfloor -(|x|+2\theta)N \rfloor+1}^{\chi(N)-1}\zeta_i|+2N^{1/3}=|\sum_{i=\lfloor Ny\rfloor+1}^{\lfloor -(|x|+2\theta)N \rfloor}\zeta_i|+2N^{1/3}$. As we also have $|Y_N^-(y)-Y_N^-(-|x|-2\theta)-\frac{1}{\sqrt{N}}(\ell^+(T_N,\lfloor Ny\rfloor)-\ell^+(T_N,\lfloor -(|x|+2\theta)N \rfloor))|\leq \frac{2N^{1/16}}{\sqrt{N}}$, this implies $\sup_{y\in[\frac{I^-}{N},-|x|-2\theta]}|Y_N^-(y)-Y_N^-(-|x|-2\theta)+\frac{\lfloor (|x|+2\theta)N \rfloor-\lfloor Ny\rfloor}{2\sqrt{N}}|\leq \max_{I^-+1 \leq i \leq \lfloor -(|x|+2\theta)N \rfloor}\frac{1}{\sqrt{N}}|\sum_{j=i}^{\lfloor -(|x|+2\theta)N \rfloor} \zeta_j| +\frac{2N^{1/16}}{\sqrt{N}}+\frac{2N^{1/3}}{\sqrt{N}}$. Moreover, $\mathcal{B}^c$ occurs, hence $|I^-+(|x|+2\theta)N| < N^{3/4}$ and $(\mathcal{B}_3^-)^c$ occurs, therefore we obtain that $\sup_{y\in[\frac{I^-}{N},-|x|-2\theta]}|Y_N^-(y)-Y_N^-(-|x|-2\theta)+\frac{\lfloor (|x|+2\theta)N \rfloor-\lfloor Ny\rfloor}{2\sqrt{N}}|\leq \max_{ -\lfloor(|x|+2\theta)N \rfloor - N^{3/4} \leq i \leq \lfloor -(|x|+2\theta)N\rfloor}\frac{1}{\sqrt{N}}|\sum_{j=i}^{\lfloor -(|x|+2\theta)N \rfloor} \zeta_j| +\frac{2N^{1/16}}{\sqrt{N}}+\frac{2N^{1/3}}{\sqrt{N}} \leq \frac{N^{19/48}}{\sqrt{N}}+\frac{2N^{1/16}}{\sqrt{N}}+\frac{2N^{1/3}}{\sqrt{N}} \leq 2 N^{-5/48}$ when $N$ is large enough. This yields $\sup_{y\in[\frac{I^-}{N},-|x|-2\theta]}|Y_N^-(y)-Y_N^-(-|x|-2\theta)+\frac{\lfloor (|x|+2\theta)N \rfloor-\lfloor Ny\rfloor}{2\sqrt{N}}|\leq 2 N^{-5/48}$ when $N$ is large enough. 
 
We also need an explicit expression of the parametric representations. Assume the part of $[0,1]$ devoted to the equivalent of Figure \ref{fig_I+large}(b) in the parametric representations is $[a_N,a_N']$. We set $\phi$ the affine function mapping $a_N$ to $-\frac{2I^-}{N}$ and $a_N'$ to $-(|x|+2\theta)N$. Then, if $\phi(t)$ belongs to some $[\frac{2i}{N},\frac{2i+1}{N})$ with $i\in\{I^-,...,-\lfloor(|x|+2\theta)N\rfloor-1\}$, we set $(u_N^-(t),r_N^-(t))=(\phi(t)-\frac{i}{N},Y_N^-(\phi(t)-\frac{i}{N}))$, while if $\phi(t)$ belongs to some $[\frac{2i+1}{N},\frac{2i+2}{N}]$ for $i\in\{I^-,...,-\lfloor(|x|+2\theta)N\rfloor-1\}$, we set $(u_N^-(t),r_N^-(t))=(\frac{i+1}{N},(-N\phi(t)+2i+2)Y_N^-((\frac{i+1}{N})^-)+(N\phi(t)-2i-1)Y_N^-(\frac{i+1}{N}))$. In addition, we set $(u_N(t),r_N(t))=(-|x|-2\theta,\hat\phi(\phi(t)))$, where $\hat\phi$ is the affine function mapping $-|x|-2\theta-\frac{\lfloor (|x|+2\theta)N\rfloor+1}{N}$ to $Y_N(-|x|-2\theta)$ and $\frac{2I^-}{N}$ to 0.
 
 We recall that it is enough to prove $|r_N^-(t)-r_N(t)| < 2N^{-1/12}$. We are going to study $|r_N^-(t)-Y_N^-(-|x|-2\theta)+\frac{\sqrt{N}}{4}\phi(t)-\frac{\lfloor -(|x|+2\theta)N \rfloor}{2\sqrt{N}}|$. We first suppose that $\phi(t)\in[\frac{2i}{N},\frac{2i+1}{N})$ with $i\in\{I^-,...,-\lfloor (|x|+2\theta)N \rfloor-1\}$. In this case, $r_N^-(t)=Y_N^-(\phi(t)-\frac{i}{N})$ and $|\frac{\phi(t)}{2}-\frac{1}{N}\lfloor N(\phi(t)-\frac{i}{N})\rfloor| = |\frac{\phi(t)}{2}-\frac{i}{N}| \leq \frac{1}{2N}$, hence 
 \[
 \left|r_N^-(t)-Y_N^-(-|x|-2\theta)+\frac{\sqrt{N}}{4}\phi(t)-\frac{\lfloor -(|x|+2\theta)N \rfloor}{2\sqrt{N}}\right| 
 \]
 \[
 = \left|Y_N^-\left(\phi(t)-\frac{i}{N}\right)-Y_N^-(-|x|-2\theta)+\frac{\lfloor N(\phi(t)-\frac{i}{N})\rfloor-\lfloor -(|x|+2\theta)N \rfloor}{2\sqrt{N}}+\frac{\frac{N}{2}\phi(t)-\lfloor N(\phi(t)-\frac{i}{N})\rfloor}{2\sqrt{N}}\right|
 \]
 \[
 \leq \left|Y_N^-\left(\phi(t)-\frac{i}{N}\right)-Y_N^-(-|x|-2\theta)+\frac{\lfloor N(\phi(t)-\frac{i}{N})\rfloor-\lfloor -(|x|+2\theta)N \rfloor}{2\sqrt{N}}\right|+\frac{\sqrt{N}}{2}\left|\frac{\phi(t)}{2}-\frac{1}{N}\left\lfloor N\left(\phi(t)-\frac{i}{N}\right)\right\rfloor\right|
 \]
 is smaller than $2N^{-5/48}+\frac{1}{4\sqrt{N}}$, thus $|r_N^-(t)-Y_N^-(-|x|-2\theta)+\frac{\sqrt{N}}{4}\phi(t)-\frac{\lfloor -(|x|+2\theta)N \rfloor}{2\sqrt{N}}| \leq 2N^{-5/48}+\frac{1}{4\sqrt{N}}$. We now consider the case $\phi(t)\in[\frac{2i+1}{N},\frac{2i+2}{N}]$ with $i\in\{I^-,...,-\lfloor (|x|+2\theta)N \rfloor-1\}$. We temporarily denote $N\phi(t)-2i-1$ by $\varepsilon$ for short, with $\varepsilon \in [0,1]$. Then we have $r_N^-(t)=(1-\varepsilon)Y_N^-((\frac{i+1}{N})^-)+\varepsilon Y_N^-(\frac{i+1}{N})$, $|\frac{\phi(t)}{2}-\frac{i}{N}| \leq \frac{1}{N}$ and $|\frac{\phi(t)}{2}-\frac{i+1}{N}| \leq \frac{1}{2N}$, therefore 
 \[
 \left|r_N^-(t)-Y_N^-(-|x|-2\theta)+\frac{\sqrt{N}}{4}\phi(t)-\frac{\lfloor -(|x|+2\theta)N \rfloor}{2\sqrt{N}}\right|
 \]
 \begin{align*}
 =&\left|(1-\varepsilon)\left(Y_N^-\left(\left(\frac{i+1}{N}\right)^-\right)-Y_N^-(-|x|-2\theta)+\frac{\sqrt{N}}{4}\phi(t)-\frac{\lfloor -(|x|+2\theta)N \rfloor}{2\sqrt{N}}\right)\right. \\
  & +\left.\varepsilon \left(Y_N^-\left(\frac{i+1}{N}\right)-Y_N^-(-|x|-2\theta)+\frac{\sqrt{N}}{4}\phi(t)-\frac{\lfloor -(|x|+2\theta)N \rfloor}{2\sqrt{N}}\right)\right| \\
 \leq & (1-\varepsilon)\left|Y_N^-\left(\left(\frac{i+1}{N}\right)^-\right)-Y_N^-(-|x|-2\theta)+\frac{\sqrt{N}}{4}\phi(t)-\frac{\lfloor -(|x|+2\theta)N \rfloor}{2\sqrt{N}}\right| \\
 & +\varepsilon\left|Y_N^-\left(\frac{i+1}{N}\right)-Y_N^-(-|x|-2\theta)+\frac{\sqrt{N}}{4}\phi(t)-\frac{\lfloor -(|x|+2\theta)N \rfloor}{2\sqrt{N}}\right| \\
 \leq &(1-\varepsilon)\left|Y_N^-\left(\left(\frac{i+1}{N}\right)^-\right)-Y_N^-(-|x|-2\theta)+\frac{i-\lfloor -(|x|+2\theta)N \rfloor}{2\sqrt{N}}\right|+(1-\varepsilon)\left|\frac{\sqrt{N}}{4}\phi(t)-\frac{i}{2\sqrt{N}}\right| \\
 & +\varepsilon\left|Y_N^-\left(\frac{i+1}{N}\right)-Y_N^-(-|x|-2\theta)+\frac{i+1-\lfloor -(|x|+2\theta)N \rfloor}{2\sqrt{N}}\right|+\varepsilon\left|\frac{\sqrt{N}}{4}\phi(t)-\frac{i+1}{2\sqrt{N}}\right| \\
 \leq & (1-\varepsilon)\sup_{y \in[\frac{I^-}{N},-|x|-2\theta]}\left|Y_N^-(y)-Y_N^-(-|x|-2\theta)+\frac{\lfloor Ny\rfloor-\lfloor -(|x|+2\theta)N \rfloor}{2\sqrt{N}}\right|+(1-\varepsilon)\frac{\sqrt{N}}{2}\left|\frac{\phi(t)}{2}-\frac{i}{N}\right| \\
 & +\varepsilon\sup_{y \in[\frac{I^-}{N},-|x|-2\theta]}\left|Y_N^-(y)-Y_N^-(-|x|-2\theta)+\frac{\lfloor Ny\rfloor-\lfloor -(|x|+2\theta)N \rfloor}{2\sqrt{N}}\right|+\varepsilon\frac{\sqrt{N}}{2}\left|\frac{\phi(t)}{2}-\frac{i+1}{N}\right| 
 \end{align*}
 \[
 \leq \sup_{y \in[\frac{I^-}{N},-|x|-2\theta]}\left|Y_N^-(y)-Y_N^-(-|x|-2\theta)+\frac{\lfloor Ny\rfloor-\lfloor -(|x|+2\theta)N \rfloor}{2\sqrt{N}}\right|+\frac{1}{2\sqrt{N}} \leq 2N^{-5/48}+\frac{1}{2\sqrt{N}}
 \]
 thanks to our bound on the sup. Since this was also true for $\phi(t)\in[\frac{2i}{N},\frac{2i+1}{N})$ with $i\in\{I^-,...,-\lfloor (|x|+2\theta)N \rfloor-1\}$, we have $|r_N^-(t)-Y_N^-(-|x|-2\theta)+\frac{\sqrt{N}}{4}\phi(t)-\frac{\lfloor -(|x|+2\theta)N \rfloor}{2\sqrt{N}}|\leq 2N^{-5/48}+\frac{1}{2\sqrt{N}}$. 
 
 The latter expression yields $|r_N^-(t)-r_N(t)| \leq |r_N(t)-Y_N^-(-|x|-2\theta)+\frac{\sqrt{N}}{4}\phi(t)-\frac{\lfloor -(|x|+2\theta)N \rfloor}{2\sqrt{N}}|+2N^{-5/48}+\frac{1}{2\sqrt{N}}=|\hat\phi(\phi(t))-Y_N^-(-|x|-2\theta)+\frac{\sqrt{N}}{4}\phi(t)-\frac{\lfloor -(|x|+2\theta)N \rfloor}{2\sqrt{N}}|+2N^{-5/48}+\frac{1}{2\sqrt{N}}$, where $\hat\phi$ is the affine function mapping $-|x|-2\theta-\frac{\lfloor(|x|+2\theta)N\rfloor+1}{N}$ to $Y_N(-|x|-2\theta)$ and $\frac{2I^-}{N}$ to 0. Therefore it is enough to prove $|\hat\phi(\phi(t))-Y_N^-(-|x|-2\theta)+\frac{\sqrt{N}}{4}\phi(t)-\frac{\lfloor -(|x|+2\theta)N \rfloor}{2\sqrt{N}}| \leq N^{-1/12}+\frac{1}{2\sqrt{N}}$ to end the proof. Now, $\hat\phi(\phi(t))-Y_N^-(-|x|-2\theta)+\frac{\sqrt{N}}{4}\phi(t)-\frac{\lfloor -(|x|+2\theta)N \rfloor}{2\sqrt{N}}$ is an affine function of $\phi(t)$, so it is enough to prove the bound for $\phi(t)=-|x|-2\theta-\frac{\lfloor(|x|+2\theta)N\rfloor+1}{N}$ and for $\phi(t)=\frac{2I^-}{N}$. We first consider $\phi(t)=\frac{2I^-}{N}$. By Lemma \ref{lem_0_after_I}, $\ell^-(T_N,I^-) = 0$. Moreover, $I^- < -\lfloor(|x|+2\theta)N\rfloor$, hence $Y_N^-(\frac{I^-}{N})=0$. We deduce 
 \[
 \left|\hat\phi(\phi(t))-Y_N^-(-|x|-2\theta)+\frac{\sqrt{N}}{4}\phi(t)-\frac{\lfloor -(|x|+2\theta)N \rfloor}{2\sqrt{N}}\right|=\left|-Y_N^-(-|x|-2\theta)+\frac{\sqrt{N}}{4}\frac{2I^-}{N}-\frac{\lfloor -(|x|+2\theta)N \rfloor}{2\sqrt{N}}\right|
 \]
 \[
 =\left|Y_N^-\left(\frac{I^-}{N}\right)-Y_N^-(-|x|-2\theta)+\frac{I^--\lfloor -(|x|+2\theta)N \rfloor}{2\sqrt{N}}\right|
 \]
 \[
 \leq \sup_{y \in[\frac{I^-}{N},-|x|-2\theta]}\left|Y_N^-(y)-Y_N^-(-|x|-2\theta)+\frac{\lfloor Ny\rfloor-\lfloor -(|x|+2\theta)N \rfloor}{2\sqrt{N}}\right|\leq 2 N^{-5/48},
 \]
  which is enough. We now consider $\phi(t)=-|x|-2\theta-\frac{\lfloor(|x|+2\theta)N\rfloor+1}{N}$. Then $|\hat\phi(\phi(t))-Y_N^-(-|x|-2\theta)+\frac{\sqrt{N}}{4}\phi(t)-\frac{\lfloor -(|x|+2\theta)N \rfloor}{2\sqrt{N}}|$ is equal to
  \[
  \left|Y_N(-|x|-2\theta)-Y_N^-(-|x|-2\theta)+\frac{\sqrt{N}}{4}\left(-|x|-2\theta-\frac{\lfloor(|x|+2\theta)N\rfloor+1}{N}\right)-\frac{\lfloor -(|x|+2\theta)N \rfloor}{2\sqrt{N}}\right|
  \]
  \[
  \leq |Y_N(-|x|-2\theta)-Y_N^-(-|x|-2\theta)|+\left|\frac{1}{4\sqrt{N}}(-(|x|+2\theta)N-\lfloor (|x|+2\theta)N \rfloor-1-2\lfloor -(|x|+2\theta)N \rfloor\right|
  \]
  \[
  \leq |Y_N(-|x|-2\theta)-Y_N^-(-|x|-2\theta)|+\frac{1}{2\sqrt{N}} \leq N^{-1/12}+\frac{1}{2\sqrt{N}}
  \]
  by Lemma \ref{lem_Y_close}, which ends the proof for $\mathcal{B}^c \cap \{I^- < -\lfloor(|x|+2\theta)N\rfloor\}$. 
 
 The argument to show $\mathcal{B}^c \cap \{I^+ > \lfloor(|x|+2\theta)N\rfloor\} \subset \{$between $\frac{((|x|+2\theta)N) \wedge I^+}{N}$ and $a$, $\|u_N^--u_N\|_\infty,\|r_N^--r_N\|_\infty \leq 2N^{-1/12}\}$ is similar and simpler, except for the end of the argument, which we give here. In a similar way as in the previous case, we must bound $|Y_N((|x|+2\theta)^-)-Y_N^-(|x|+2\theta)+\frac{\sqrt{N}}{4}(|x|+2\theta+\frac{\lfloor(|x|+2\theta)N\rfloor}{N})-\frac{\lfloor (|x|+2\theta)N \rfloor}{2\sqrt{N}}| \leq |Y_N((|x|+2\theta)^-)-Y_N^-(|x|+2\theta)|+|\frac{(|x|+2\theta)N-\lfloor (|x|+2\theta)N \rfloor}{4\sqrt{N}}| \leq |Y_N((|x|+2\theta)^-)-Y_N^-((|x|+2\theta)^-)|+|Y_N^-((|x|+2\theta)^-)-Y_N^-(|x|+2\theta)|+\frac{1}{4\sqrt{N}}$, hence Lemma \ref{lem_Y_close} yields $|Y_N((|x|+2\theta)^-)-Y_N^-(|x|+2\theta)+\frac{\sqrt{N}}{4}(|x|+2\theta+\frac{\lfloor(|x|+2\theta)N\rfloor}{N})-\frac{\lfloor (|x|+2\theta)N \rfloor}{2\sqrt{N}}| \leq N^{-1/12}+|Y_N^-((|x|+2\theta)^-)-Y_N^-(|x|+2\theta)|+\frac{1}{4\sqrt{N}}$. In addition, the definition of $Y_N^-$ and \eqref{eq_local_times} yield that if $(|x|+2\theta)N$ is not an integer, then $Y_N^-((|x|+2\theta)^-)=Y_N^-(|x|+2\theta)$, while if $(|x|+2\theta)N$ is an integer then $|Y_N^-((|x|+2\theta)^-)-Y_N^-(|x|+2\theta)|=\frac{1}{\sqrt{N}}|\ell^-(T_N,(|x|+2\theta)N-1)-\ell^-(T_N,(|x|+2\theta)N)|=\frac{1}{\sqrt{N}}|\eta_{(|x|+2\theta)N-1,+}(\ell^-(T_N,(|x|+2\theta)N-1))| \leq \frac{N^{1/16}+1/2}{\sqrt{N}}$ since $(\mathcal{B}_2)^c$ occurs. In all cases we obtain $|Y_N^-((|x|+2\theta)^-)-Y_N^-(|x|+2\theta)|\leq \frac{N^{1/16}+1/2}{\sqrt{N}}$, therefore $|Y_N((|x|+2\theta)^-)-Y_N^-(|x|+2\theta)+\frac{\sqrt{N}}{4}(|x|+2\theta+\frac{\lfloor(|x|+2\theta)N\rfloor}{N})-\frac{\lfloor (|x|+2\theta)N \rfloor}{2\sqrt{N}}| \leq N^{-1/12}+\frac{N^{1/16}+1/2}{\sqrt{N}}+\frac{1}{4\sqrt{N}}$, which is a bound small enough to end the proof of the claim. 
\end{proof}

\section{Convergence of the local times process: proof of Theorem \ref{thm_main} and Proposition \ref{prop_uniform_conv}}\label{sec_conv_processes}

\subsection{Proof of Theorem \ref{thm_main}}\label{subsec_proof_thm}

Our aim is to prove that $Y_N^\pm$ converges in distribution to $(B_y^x \mathds{1}_{\{y\in[-|x|-2\theta,|x|+2\theta)\}})_{y\in\mathds{R}}$ in the Skorohod $M_1$ topology on $D(-\infty,+\infty)$ when $N$ tends to $+\infty$. Proposition \ref{prop_distance_param} yields that $Y_N^\pm$ is close to the function $Y_N$ defined by $Y_N(y)=\frac{1}{\sqrt{N}}\sum_{i=\lfloor N y\rfloor+1}^{\chi(N)-1}\zeta_i$ if $y \in [-|x|-2\theta,\frac{\chi(N)}{N})$, $Y_N(y)=\frac{1}{\sqrt{N}}\sum_{i=\chi(N)}^{\lfloor N y\rfloor-1}\zeta_i$ if $y \in [\frac{\chi(N)}{N},|x|+2\theta)$, and $Y_N(y)=0$ otherwise. One has the feeling that by Donsker's Invariance Principle, $Y_N$ should converge to $(B_y^x \mathds{1}_{\{y\in[-|x|-2\theta,|x|+2\theta)\}})_{y\in\mathds{R}}$ and so we should be able to conclude quickly, but proving rigorously the convergence in the Skorohod $M_1$ topology on $D(-\infty,+\infty)$ is harder than it looks. We are instead going to use a similar argument with a new process $Y_N''$ which will be ``like $Y_N$, but continuous in $[-|x|-2\theta,|x|+2\theta)$''. We will define it as follows. We first set a process $Y_N'$ thus: if $Ny\in\mathds{Z}$ then $Y_N'(y)=\frac{1}{\sqrt{N}}\sum_{i= N y+1}^{\chi(N)-1}\zeta_i$ if $y \in (-\infty,\frac{\chi(N)}{N})$ and $Y_N'(y)=\frac{1}{\sqrt{N}}\sum_{i=\chi(N)}^{N y-1}\zeta_i$ if $y \in [\frac{\chi(N)}{N},+\infty)$, and in-between $Y_N'$ is linearly interpoled. 
We then define $Y_N''$ by $Y_N''(y)=Y_N'(y)\mathds{1}_{\{y\in[-|x|-2\theta,|x|+2\theta)\}}$ for any $y\in\mathds{R}$. Then $Y_N''$ will converge to $(B_y^x \mathds{1}_{\{y\in[-|x|-2\theta,|x|+2\theta)\}})_{y\in\mathds{R}}$ and be close to $Y_N$, which is stated in the two following lemmas.

\begin{lemma}\label{lem_conv_Y''_M_1}
 $Y_N''$ converges to $(B_y^x \mathds{1}_{\{y\in[-|x|-2\theta,|x|+2\theta)\}})_{y\in\mathds{R}}$ in distribution when $N$ tends to $+\infty$ for the Skorohod $M_1$ topology in $D(-\infty,\infty)$.
\end{lemma}

\begin{lemma}\label{lem_Y_close_Y''}
 $\mathds{P}(d_{M_1}(Y_N^\pm,Y_N'')>N^{-7/16})$ tends to 0 when $N$ tends to $+\infty$.
\end{lemma}

Given these two lemmas, the proof of Theorem \ref{thm_main} is rather standard. One may for example look at the end of the proof of the Donsker invariance principle in \cite{Morters_Peres_Brownian_Motion} (here $Y_N''$ converges to the desired distribution instead of having it outright, but this convergence yields that the probability $Y_N''$ is in a closed set has the right limit). Thus we only have to prove Lemmas \ref{lem_conv_Y''_M_1} and \ref{lem_Y_close_Y''}. In order to do this, we first need two easy lemmas which will also be used later in this work. If we denote $C[-|x|-2\theta,|x|+2\theta]$ the space of continuous functions $:[-|x|-2\theta,|x|+2\theta] \mapsto \mathds{R}$, since the $(\zeta_i)_{i\in\mathds{Z}}$ are i.i.d. with law $\rho_0$ which is symmetric so has zero mean, Donsker's Invariance Principle yields the following.

\begin{lemma}\label{lem_conv_Y'_C}
 $Y_N'|_{[-|x|-2\theta,|x|+2\theta]}$ converges in distribution to $B^x|_{[-|x|-2\theta,|x|+2\theta]}$ when $N$ tends to $+\infty$ for the topology defined on $C[-|x|-2\theta,|x|+2\theta]$ by the uniform norm.
\end{lemma}

The following lemma is also easy to prove.

\begin{lemma}\label{lem_Y_close_Y'}
 If $(\mathcal{B}_2)^c$ occurs, $\sup\{|Y_N(y)-Y_N''(y)|:y\in[-|x|-2\theta,|x|+2\theta)\} \leq N^{-7/16}$.
\end{lemma}

\begin{proof}
 By the definition of $Y_N$ and $Y_N''$, we have $\sup\{|Y_N(y)-Y_N''(y)|:y\in[-|x|-2\theta,|x|+2\theta)\} \leq \frac{1}{\sqrt{N}}\sup\{|\zeta_i|:-(|x|+2\theta)N \leq i \leq (|x|+2\theta)N\}$, which is smaller than $\frac{N^{1/16}}{\sqrt{N}}=N^{-7/16}$ if $(\mathcal{B}_2)^c$ occurs. 
\end{proof}

We also need the following technical lemma in order to deduce results on the Skorohod $M_1$ topology from Lemmas \ref{lem_conv_Y'_C} and \ref{lem_Y_close_Y'}. 

\begin{lemma}\label{lem_sup_to_M1}
 Let $N>0$ and $Z_1,Z_2 \in D(-\infty,+\infty)$ whose possible discontinuities belong to $\frac{1}{N}\mathds{Z}$, then we have $d_{M_1}((Z_1(y)\mathds{1}_{\{y\in[-|x|-2\theta,|x|+2\theta)\}})_{y\in\mathds{R}},(Z_2(y)\mathds{1}_{\{y\in[-|x|-2\theta,|x|+2\theta)\}})_{y\in\mathds{R}}) \leq \sup\{|Z_1(y)-Z_2(y)|:y\in[-|x|-2\theta,|x|+2\theta)\}$.
\end{lemma}

\begin{proof}
  Lemma \ref{lem_sup_to_M1} can be shown by writing for each $a \neq |x|+2\theta$ parametric representations of the two processes on $[-a,a]$ ``following their completed graphs together'' (one can find an explicit construction of such representations in the first arXiv version of this paper \cite{Mareche2022v1}). 
 \end{proof}
 
 Lemma \ref{lem_sup_to_M1} will allow us to deduce Lemma \ref{lem_conv_Y''_M_1} from Lemma \ref{lem_conv_Y'_C}, and Lemma \ref{lem_Y_close_Y''} from Lemma \ref{lem_Y_close_Y'} and Proposition \ref{prop_distance_param}, which will end the proof of Theorem \ref{thm_main}.

\begin{proof}[Proof of Lemma \ref{lem_conv_Y''_M_1}.]
Let $f:D(-\infty,+\infty) \mapsto \mathds{R}$ be bounded and continuous with respect to the Skorohod $M_1$ topology on $D(-\infty,+\infty)$, we need to prove that $\mathds{E}(f(Y_N''))$ converges to $\mathds{E}(f((B_y^x \mathds{1}_{\{y\in[-|x|-2\theta,|x|+2\theta)\}})_{y\in\mathds{R}}))$ when $N$ tends to $+\infty$. We define $g : C[-|x|-2\theta,|x|+2\theta] \mapsto \mathds{R}$ by $g(Z)=f((Z(y)\mathds{1}_{\{y\in[-|x|-2\theta,|x|+2\theta)\}})_{y\in\mathds{R}})$ for any $Z \in C[-|x|-2\theta,|x|+2\theta]$. We then have $\mathds{E}(f(Y_N''))=\mathds{E}(g(Y_N'|_{[-|x|-2\theta,|x|+2\theta]}))$ and $\mathds{E}(f((B_y^x \mathds{1}_{\{y\in[-|x|-2\theta,|x|+2\theta)\}})_{y\in\mathds{R}}))=\mathds{E}(g(B^x |_{[-|x|-2\theta,|x|+2\theta]}))$, hence it is enough to prove $\mathds{E}(g(Y_N'|_{[-|x|-2\theta,|x|+2\theta]}))$ converges to $\mathds{E}(g(B^x |_{[-|x|-2\theta,|x|+2\theta]}))$ when $N$ tends to $+\infty$. Furthermore, Lemma \ref{lem_conv_Y'_C} yields that $Y_N'|_{[-|x|-2\theta,|x|+2\theta]}$ converges in distribution to $B^x|_{[-|x|-2\theta,|x|+2\theta]}$ when $N$ tends to $+\infty$ for the topology defined on $C[-|x|-2\theta,|x|+2\theta]$ by the uniform norm. Consequently, we only have to prove that $g$ is continuous for this topology.

Let $(Z_k)_{k\in\mathds{N}}$ be a sequence in $C[-|x|-2\theta,|x|+2\theta]$ converging uniformly to $Z \in C[-|x|-2\theta,|x|+2\theta]$ when $k$ tends to $+\infty$. Then Lemma \ref{lem_sup_to_M1} states that for all $k\in\mathds{N}$, $d_{M_1}((Z_k(y)\mathds{1}_{\{y\in[-|x|-2\theta,|x|+2\theta)\}})_{y\in\mathds{R}},(Z(y)\mathds{1}_{\{y\in[-|x|-2\theta,|x|+2\theta)\}})_{y\in\mathds{R}}) \leq \sup\{|Z_k(y)-Z(y)|:y\in[-|x|-2\theta,|x|+2\theta)\} \leq \|Z_k-Z\|_\infty$. Since the latter tends to 0 when $k$ tends to $+\infty$, we deduce $(Z_k(y)\mathds{1}_{\{y\in[-|x|-2\theta,|x|+2\theta)\}})_{y\in\mathds{R}}$ converges to $(Z(y)\mathds{1}_{\{y\in[-|x|-2\theta,|x|+2\theta)\}})_{y\in\mathds{R}}$ when $k$ tends to $+\infty$ with respect to the Skorohod $M_1$ topology on $D(-\infty,+\infty)$. Since $f$ is continuous with respect to this topology, $(g(Z_k))_{k\in\mathds{N}}$ converges to $g(Z)$ when $k$ tends to $+\infty$. Consequently $g$ is continuous for the topology defined on $C[-|x|-2\theta,|x|+2\theta]$ by the uniform norm, which ends the proof.
\end{proof}
 
 \begin{proof}[Proof of Lemma \ref{lem_Y_close_Y''}.]
 $\mathds{P}(d_{M_1}(Y_N^\pm,Y_N'')>4N^{-1/12}) \leq \mathds{P}(d_{M_1}(Y_N^\pm,Y_N)>3N^{-1/12})+\mathds{P}(d_{M_1}(Y_N,Y_N'')>N^{-7/16})$ when $N$ is large enough. By Lemmas \ref{lem_Y_close_Y'} and \ref{lem_sup_to_M1} $\mathds{P}(d_{M_1}(Y_N,Y_N'')>N^{-7/16}) \leq \mathds{P}(\mathcal{B}_2)$. Therefore $\mathds{P}(d_{M_1}(Y_N^\pm,Y_N'')>4N^{-1/12}) \leq \mathds{P}(d_{M_1}(Y_N^\pm,Y_N)>3N^{-1/12})+\mathds{P}(\mathcal{B}_2)$, which tends to 0 when $N$ tends to $+\infty$ by Proposition \ref{prop_distance_param} and Lemma \ref{lem_zeta_eta_small}. 
 \end{proof}
 
 \subsection{Proof of Proposition \ref{prop_uniform_conv}}
 
 Our goal is to prove that for any closed interval $I\in\mathds{R}$ that does not contain $-|x|-2\theta$ or $|x|+2\theta$, the process $(Y_N^\pm(y))_{y\in I}$ converges in distribution to $(B_y^x \mathds{1}_{\{y\in[-|x|-2\theta,|x|+2\theta)\}})_{y\in I}$ in the topology on $D I$ given by the uniform norm when $N$ tends to $+\infty$. We first assume $I=[a,b]$ or $[a,+\infty)$ with $a>|x|+2\theta$ (the case $I=[a,b]$ or $(-\infty,b]$ with $b<-|x|-2\theta$ can be dealt with in the same way). We are going to prove that outside an event of small probability, $(Y_N^\pm(y))_{y\in I}=0=(B_y^x \mathds{1}_{\{y\in[-|x|-2\theta,|x|+2\theta)\}})_{y\in I}$. For any $y \geq (|x|+2\theta)\vee\frac{I^+}{N}$, by Lemma \ref{lem_0_after_I} we have $\ell^\pm(T_N,\lfloor Ny \rfloor)=0$, thus $Y_N^\pm(y)=0$. We deduce that as soon as $\frac{I^+}{N} \leq a$, we have $(Y_N^\pm(y))_{y\in I}=0=(B_y^x \mathds{1}_{\{y\in[-|x|-2\theta,|x|+2\theta)\}})_{y\in I}$. In addition, when $N$ is large enough we have $a \geq |x|+2\theta+N^{-1/4}$. Therefore, when $N$ is large enough, $\mathds{P}((Y_N^\pm(y))_{y\in I}\neq(B_y^x \mathds{1}_{\{y\in[-|x|-2\theta,|x|+2\theta)\}})_{y\in I}) \leq \mathds{P}(|I^+-(|x|+2\theta)N|\geq N^{3/4})$, which tends to 0 when $N$ tends to $+\infty$ by Lemma \ref{lem_control_I}. This yields that $(Y_N^\pm(y))_{y\in I}$ converges in distribution to $(B_y^x \mathds{1}_{\{y\in[-|x|-2\theta,|x|+2\theta)\}})_{y\in I}$ in the topology on $D I$ given by the uniform norm.
 
 We now deal with the case $I=[a,b]$ with $-|x|-2\theta < a < b < |x|+2\theta$. The idea is that we will be far from the problems at $-|x|-2\theta$ and $|x|+2\theta$, thus $Y_N^\pm$ will be close to $Y_N'$ in all $I$, and $Y_N'$ converges to the right limit, hence $Y_N^\pm$ too. We first prove the following lemma.
 
 \begin{lemma}\label{lem_Ypm_close_Y'}
  For any $-|x|-2\theta < a < b < |x|+2\theta$, we have that $\mathds{P}(\|Y_N^\pm|_{[a,b]}-Y_N'|_{[a,b]}\|_\infty > 2N^{-1/12})$ tends to 0 when $N$ tends to $+\infty$. 
 \end{lemma}

 \begin{proof}
We assume $(\mathcal{B}_2)^c$, $(\mathcal{B}_4^-)^c$, $(\mathcal{B}_4^+)^c$ occurs, as well as $|I^-+(|x|+2\theta)N|< N^{3/4}$, $|I^+-(|x|+2\theta)N|< N^{3/4}$. When $N$ is large enough, we have $a \geq -|x|-2\theta+N^{-1/4} > \frac{I^-}{N}$ and $b \leq |x|+2\theta-N^{-1/4} < \frac{I^+}{N}$, hence $[a,b] \subset (\frac{I^-}{N},\frac{I^+}{N})$. Therefore, for any $y\in[a,b]$, Lemma \ref{lem_Y_close} yields $|Y_N^\pm(y)-Y_N(y)| \leq N^{-1/12}$, and Lemma \ref{lem_Y_close_Y'} gives $|Y_N(y)-Y_N'(y)| \leq N^{-7/16}$, hence we get $|Y_N^\pm(y)-Y_N'(y)| \leq 2N^{-1/12}$, and we deduce $\|Y_N^\pm|_{[a,b]}-Y_N'|_{[a,b]}\|_\infty \leq 2N^{-1/12}$. This implies $\mathds{P}(\|Y_N^\pm|_{[a,b]}-Y_N'|_{[a,b]}\|_\infty > 2N^{-1/12}) \leq \mathds{P}(\mathcal{B}_2\cup\mathcal{B}_4^-\cup\mathcal{B}_4^+ \cup \{|I^-+(|x|+2\theta)N| \geq N^{3/4}\} \cup \{|I^+-(|x|+2\theta)N| \geq N^{3/4}\})$, which tends to 0 when $N$ tends to $+\infty$ thanks to Lemmas \ref{lem_zeta_eta_small}, \ref{lem_control_I} and \ref{lem_sum_eta_zeta}.
 \end{proof}
 
 Moreover, for any $-|x|-2\theta < a < b < |x|+2\theta$, by Donsker’s Invariance Principle, $Y_N'|_{[a,b]}$ converges in distribution to $B^x|_{[a,b]}$ when $N$ tends to $+\infty$ for the topology defined on $D[a,b]$ by the uniform norm. The proof of Proposition \ref{prop_uniform_conv} from this is standard, as was the proof of Theorem \ref{thm_main} from Lemmas \ref{lem_conv_Y''_M_1} and \ref{lem_Y_close_Y''}. 
 
 \section{No convergence in the Skorohod $J_1$ topology: proof of Proposition \ref{prop_no_J1_conv}}\label{sec_no_J1_conv}
 
 In this section, our aim is to prove that $Y_N^\pm$ does not converge in distribution in the Skorohod $J_1$ topology on $D(-\infty,+\infty)$ when $N$ tends to $+\infty$. We will first prove that if $Y_N^\pm$ converges in the Skorohod $J_1$ topology, the limit has to be the same as in the Skorohod $M_1$ topology, that is $(B_y^x \mathds{1}_{\{y\in[-|x|-2\theta,|x|+2\theta)\}})_{y\in\mathds{R}}$ by Theorem \ref{thm_main} (this will be Lemma \ref{lem_lim_J1_M1}). Afterwards, we will prove that $Y_N^\pm$ does not converge in distribution in the Skorohod $J_1$ topology to $(B_y^x \mathds{1}_{\{y\in[-|x|-2\theta,|x|+2\theta)\}})_{y\in\mathds{R}}$ by finding some closed set $\Xi$ so that $\limsup_{N\to+\infty}\mathds{P}(Y_N^\pm \in \Xi) > \mathds{P}((B_y^x \mathds{1}_{\{y\in[-|x|-2\theta,|x|+2\theta)\}})_{y\in\mathds{R}}\in \Xi)$, which is enough by the Portmanteau Theorem.
 
 \begin{lemma}\label{lem_lim_J1_M1}
  If $Y_N^\pm$ converges in distribution in the Skorohod $J_1$ topology on $D(-\infty,+\infty)$ when $N$ tends to $+\infty$, the limit is $(B_y^x \mathds{1}_{\{y\in[-|x|-2\theta,|x|+2\theta)\}})_{y\in\mathds{R}}$.
 \end{lemma}

 \begin{proof}
 The idea is that the Skorohod $J_1$ topology is stronger than the Skorohod $M_1$ topology. We assume $Y_N^\pm$ converges in distribution to some $Z$ in the Skorohod $J_1$ topology on $D(-\infty,+\infty)$ when $N$ tends to $+\infty$. It can be proven that for any $a>0$ we have $d_{M_1,a} \leq d_{J_1,-a,a}$. Indeed, this is Theorem 12.3.2 of \cite{Whitt_Skorohod_topologies}, whose proof is in the Internet supplement of that book (just replace the discontinuity points of $x_1$ with their image by $\lambda^{-1}$). This implies $d_{M_1} \leq d_{J_1}$. Therefore a function $g : D(-\infty,+\infty) \mapsto \mathds{R}$ bounded and continuous for the Skorohod $M_1$ topology is also continuous for the Skorohod $J_1$ topology. We deduce that $\mathds{E}(g(Y_N^\pm))$ converges to $\mathds{E}(g(Z))$ when $N$ tends to $+\infty$, thus $Y_N^\pm$ converges in distribution to $Z$ in the Skorohod $M_1$ topology when $N$ tends to $+\infty$. By Theorem \ref{thm_main}, the limit has to be $(B_y^x \mathds{1}_{\{y\in[-|x|-2\theta,|x|+2\theta)\}})_{y\in\mathds{R}}$. 
 \end{proof}

 We now define our closed set $\Xi$. The idea behind this definition is that with high probability, $B_{|x|+2\theta}^x$ is at some distance from 0, hence at some point around $|x|+2\theta$, $Y_N^\pm$ will be close to $B_{|x|+2\theta}^x$, thus at some distance from 0. Furthermore, at $|x|+2\theta$ the process $(B_y^x \mathds{1}_{\{y\in[-|x|-2\theta,|x|+2\theta)\}})_{y\in\mathds{R}}$ will jump directly from $B_{|x|+2\theta}^x$ to 0, while $Y_N^\pm$, which can make only jumps of order $\frac{1}{\sqrt{N}}$, will have to cross the distance separating $B_{|x|+2\theta}^x$ from 0 without bigs jumps. Therefore if $\delta_1>0$ is much smaller than $B_{|x|+2\theta}^x$, then $Y_N^\pm(y)$ will enter the interval $[\delta_1,2\delta_1]$ for $y$ near $|x|+2\theta$, while $(B_y^x \mathds{1}_{\{y\in[-|x|-2\theta,|x|+2\theta)\}})_{y\in\mathds{R}}$ will not. We thus set $\Xi$ to be roughly ``the function enters $[\delta_1,2\delta_1]$ around $|x|+2\theta$''. More rigorously, by the definition of $B^x$, the random variable $B_{|x|+2\theta}^x$ has distribution $\mathcal{N}(0,2\theta)$, hence there exists $\delta_1>0$ so that $\mathds{P}(|B_{|x|+2\theta}^x| \leq 4\delta_1) \leq 1/8$. Moreover, $B^x$ is continuous, hence there exists $0 < \delta_2 <\theta$ so that $\mathds{P}(\exists\, y \in [|x|+2\theta-\delta_2,|x|+2\theta], |B_y^x| \leq 3\delta_1) \leq 1/4$. We then define $\Xi = \{Z\in D(-\infty,+\infty)\,|\,\exists y \in [|x|+2\theta-\delta_2,|x|+2\theta+\delta_2],|Z(y)|\in[\delta_1,2\delta_1]$ or $|Z(y^-)|\in[\delta_1,2\delta_1]\}$ (the inclusion of $Z(y^-)$ was necessary for $\Xi$ to be closed). Then $\mathds{P}((B_y^x \mathds{1}_{\{y\in[-|x|-2\theta,|x|+2\theta)\}})_{y\in\mathds{R}} \in \Xi) \leq 1/4$. We will prove the two following lemmas.
 
 \begin{lemma}\label{lem_J1_proba_bound}
  When $N$ is large enough, $\mathds{P}(Y_N^\pm \in \Xi) \geq 1/2$.
 \end{lemma}
 
 \begin{lemma}\label{lem_J1_closed}
  $\Xi$ is closed in the Skorohod $J_1$ topology on $D(-\infty,+\infty)$. 
 \end{lemma}
 
 With these two lemmas, the proof of Proposition \ref{prop_no_J1_conv} becomes easy. 
 
 \begin{proof}[Proof of Proposition \ref{prop_no_J1_conv}.]
 Lemma \ref{lem_J1_proba_bound} yields $\limsup_{N\to+\infty}\mathds{P}(Y_N^\pm \in \Xi) \geq 1/2$, and the definition of $\Xi$ ensures that $\mathds{P}((B_y^x \mathds{1}_{\{y\in[-|x|-2\theta,|x|+2\theta)\}})_{y\in\mathds{R}} \in \Xi) \leq 1/4$, hence $\limsup_{N\to+\infty}\mathds{P}(Y_N^\pm \in \Xi) > \mathds{P}((B_y^x \mathds{1}_{\{y\in[-|x|-2\theta,|x|+2\theta)\}})_{y\in\mathds{R}}\in \Xi)$. Since Lemma \ref{lem_J1_closed} yields $\Xi$ is closed in the Skorohod $J_1$ topology on $D(-\infty,+\infty)$, the Portmanteau Theorem implies $Y_N^\pm$ does not converge in distribution in the Skorohod $J_1$ topology on $D(-\infty,+\infty)$ to $(B_y^x \mathds{1}_{\{y\in[-|x|-2\theta,|x|+2\theta)\}})_{y\in\mathds{R}}$ when $N$ tends to $+\infty$. Hence Lemma \ref{lem_lim_J1_M1} yields that $Y_N^\pm$ does not converge in distribution in the Skorohod $J_1$ topology on $D(-\infty,+\infty)$ when $N$ tends to $+\infty$, which is Proposition \ref{prop_no_J1_conv}. 
 \end{proof}
 
 Thus it remains only to prove Lemmas \ref{lem_J1_proba_bound} and \ref{lem_J1_closed}.  
 
 \begin{proof}[Proof of Lemma \ref{lem_J1_proba_bound}.]
 The idea is that with good probability, when $y$ is a bit smaller than $|x|+2\theta$, we have $Y_N^\pm(y)$ of the same order as $B_{|x|+2\theta}^x$, thus away from 0, while when $y$ is a bit larger than $|x|+2\theta$, we have $Y_N^\pm(y)=0$, so since $Y_N$ can only make jumps of order $\frac{1}{\sqrt{N}}$, it will enter $[\delta_1,2\delta_1]$. We now give the rigorous argument. We begin by assuming that $|Y_N^\pm(|x|+2\theta-\delta_2)| > 3\delta_1$ (that is $Y_N(y)$ is indeed away from 0 when $y$ is a bit smaller than $|x|+2\theta$), $(\mathcal{B}_2)^c$ occurs and $|I^+-(|x|+2\theta)N| < N^{3/4}$, and proving that when $N$ is large enough, $Y_N^\pm \in \Xi$. We first show $Y_N^\pm(|x|+2\theta+\delta_2)=0$. When $N$ is large enough, $\frac{I^+}{N} \leq |x|+2\theta+N^{-1/4} \leq |x|+2\theta+\delta_2$. Moreover, Lemma \ref{lem_0_after_I} implies $\ell^\pm(T_N,\lfloor Ny \rfloor)=0$ for any $y \geq \frac{I^+}{N}$, hence for $y=|x|+2\theta+\delta_2$. This yields $Y_N^\pm(|x|+2\theta+\delta_2)=0$. Moreover, we assumed $|Y_N^\pm(|x|+2\theta-\delta_2)| > 3\delta_1$. Furthermore, equations \eqref{eq_local_times} and \eqref{eq_local_times_2} yield that the jumps of $Y_N^\pm$ in $[|x|+2\theta-\delta_2,|x|+2\theta+\delta_2]$ are either $\frac{1}{\sqrt{N}}\eta_{i,+}(\ell^-(T_N,i))$ (if we deal with $Y_N^-$) or $\frac{1}{\sqrt{N}}\eta_{i+1,+}(\ell^-(T_N,i+1))$ (if we deal with $Y_N^+$) with $i\in\{\lfloor(|x|+2\theta-\delta_2)N\rfloor,...,\lfloor(|x|+2\theta+\delta_2)N\rfloor-1\}$. Since $(\mathcal{B}_2)^c$ occurs, the jumps of $Y_N^\pm$ in $[|x|+2\theta-\delta_2,|x|+2\theta+\delta_2]$ have size at most $\frac{1}{\sqrt{N}}(N^{1/16}+1/2)$, which tends to 0 when $N$ tends to $+\infty$. Therefore, when $N$ is large enough, there exists $y \in [|x|+2\theta-\delta_2,|x|+2\theta+\delta_2]$ so that $|Y_N^\pm(y)|\in[\delta_1,2\delta_1]$, hence $Y_N^\pm\in\Xi$. Consequently, when $N$ is large enough, if $|Y_N^\pm(|x|+2\theta-\delta_2)| > 3\delta_1$, $(\mathcal{B}_2)^c$ and $|I^+-(|x|+2\theta)N| < N^{3/4}$ then $Y_N^\pm\in\Xi$. This implies $\mathds{P}(Y_N^\pm \not\in \Xi) \leq \mathds{P}(|Y_N^\pm(|x|+2\theta-\delta_2)| \leq 3\delta_1)+\mathds{P}(\mathcal{B}_2)+\mathds{P}(|I^+-(|x|+2\theta)N| \geq N^{3/4})$. In addition, Lemma \ref{lem_zeta_eta_small} and Lemma \ref{lem_control_I} yield respectively that $\mathds{P}(\mathcal{B}_2)$ and $\mathds{P}(|I^+-(|x|+2\theta)N| \geq N^{3/4})$ tend to 0 when $N$ tends to $+\infty$. Therefore it is enough to prove that $\mathds{P}(|Y_N^\pm(|x|+2\theta-\delta_2)| \leq 3\delta_1) \leq 3/8$ when $N$ is large enough to deduce that $\mathds{P}(Y_N^\pm \not\in \Xi) \leq 1/2$ when $N$ is large enough and end the proof of Lemma \ref{lem_J1_proba_bound}.
 
 We now prove $\mathds{P}(|Y_N^\pm(|x|+2\theta-\delta_2)| \leq 3\delta_1) \leq 3/8$ when $N$ is large enough, by noticing $Y_N^\pm(|x|+2\theta-\delta_2)$ is close to $Y_N'(|x|+2\theta-\delta_2)$, which will converge in distribution to $B_{|x|+2\theta-\delta_2}^x$ when $N$ tends to $+\infty$. Lemma \ref{lem_Ypm_close_Y'} implies $\mathds{P}(\|Y_N^\pm|_{[0,|x|+2\theta-\delta_2]}-Y_N'|_{[0,|x|+2\theta-\delta_2]}\|_\infty > 2N^{-1/12})$ tends to 0 when $N$ tends to $+\infty$, hence $\mathds{P}(|Y_N^\pm(|x|+2\theta-\delta_2)-Y_N'(|x|+2\theta-\delta_2)| > 2N^{-1/12})$ tends to 0 when $N$ tends to $+\infty$, which implies $Y_N^\pm(|x|+2\theta-\delta_2)-Y_N'(|x|+2\theta-\delta_2)$ converges in probability to 0 when $N$ tends to $+\infty$. In addition, Lemma \ref{lem_conv_Y'_C} states $Y_N'|_{[-|x|-2\theta,|x|+2\theta]}$ converges in distribution to $B^x|_{[-|x|-2\theta,|x|+2\theta]}$ when $N$ tends to $+\infty$ for the topology defined on $C[-|x|-2\theta,|x|+2\theta]$ by the uniform norm, hence $Y_N'(|x|+2\theta-\delta_2)$ converges in distribution to $B_{|x|+2\theta-\delta_2}^x$ when $N$ tends to $+\infty$. Therefore Slutsky's Theorem yields that $Y_N^\pm(|x|+2\theta-\delta_2)$ converges in distribution to $B_{|x|+2\theta-\delta_2}^x$ when $N$ tends to $+\infty$. Moreover, we defined $\Xi$ so that $\mathds{P}(\exists\, y \in [|x|+2\theta-\delta_2,|x|+2\theta], |B_y^x| \leq 3\delta_1) \leq 1/4$, hence $\mathds{P}(|B_{|x|+2\theta-\delta_2}^x| \leq 3\delta_1) \leq 1/4$. This implies that when $N$ is large enough, $\mathds{P}(|Y_N^\pm(|x|+2\theta-\delta_2)| \leq 3\delta_1) \leq 3/8$.
 \end{proof}

 \begin{proof}[Proof of Lemma \ref{lem_J1_closed}.]
  Let $(Z_N)_{N\in\mathds{N}}$ be a sequence of elements of $\Xi$ converging to $Z$ in the Skorohod $J_1$ topology on $D(-\infty,+\infty)$, we will prove $Z \in \Xi$. By taking a subsequence, we may assume $d_{J_1}(Z,Z_N) < e^{-|x|-2\theta-\delta_2-1}/N$ for any $N\in\mathds{N}^*$. Then for any $N\in\mathds{N}^*$, some $a_N > |x|+2\theta+\delta_2+1$ so that $d_{J_1,-a_N,a_N}(Z|_{[-a_N,a_N]},Z_N|_{[-a_N,a_N]}) \leq 1/N$ will exist. Indeed, if it was not the case, for some $N$ we would have $d_{J_1}(Z,Z_N)=\int_0^{+\infty}e^{-a}(d_{J_1,-a,a}(Z|_{[-a,a]},Z_N|_{[-a,a]}) \wedge 1 )\mathrm{d}a \geq \int_{|x|+2\theta+\delta_2+1}^{+\infty}e^{-a}\frac{1}{N}\mathrm{d}a=e^{-|x|-2\theta-\delta_2-1}/N$, which does not happen. For all $N\in \mathds{N}^*$, the fact that we have $d_{J_1,-a_N,a_N}(Z|_{[-a_N,a_N]},Z_N|_{[-a_N,a_N]}) \leq 1/N$ implies there exists $\lambda_N \in \Lambda_{-a_N,a_N}$ with $\|Z|_{[-a_N,a_N]} \circ \lambda_N - Z_N|_{[-a_N,a_N]}\|_\infty\leq 2/N$ and $\|\lambda_N-\mathrm{Id}_{-a_N,a_N}\|_\infty \leq 2/N$. Moreover, $Z_N\in\Xi$, hence there exists $y_N \in [|x|+2\theta-\delta_2,|x|+2\theta+\delta_2],|Z_N(y_N)|\in[\delta_1,2\delta_1]$ or $|Z_N(y_N^-)|\in[\delta_1,2\delta_1]$. We now define $y_N'$ as follows: if $|Z_N(y_N)|\in[\delta_1,2\delta_1]$ we set $y_N'=y_N$. Otherwise, since $|Z_N(y_N^-)|\in[\delta_1,2\delta_1]$ we can take some $y_N'$ in $[y_N-\frac{1}{N},y_N]$ so that $|Z_N(y_N')|\in[\delta_1-1/N,2\delta_1+1/N]$. In both cases, we have $y_N' \in [|x|+2\theta-\delta_2-1/N,|x|+2\theta+\delta_2]$ and $|Z_N(y_N')|\in[\delta_1-1/N,2\delta_1+1/N]$. Furthermore, $\|\lambda_N-\mathrm{Id}_{-a_N,a_N}\|_\infty \leq 2/N$, hence $|\lambda_N(y_N')-y_N'| \leq 2/N$, thus $\lambda_N(y_N') \in [|x|+2\theta-\delta_2-3/N,|x|+2\theta+\delta_2+2/N]$. In addition, $\|Z(\lambda_N(y_N')) - Z_N(y_N')\|_\infty\leq 2/N$, hence $|Z(\lambda_N(y_N'))|\in[\delta_1-3/N,2\delta_1+3/N]$. By taking a subsequence, we may assume that $\lambda_N(y_N')$ converges to some $y_\infty\in [|x|+2\theta-\delta_2,|x|+2\theta+\delta_2]$. In addition, $Z$ is càdlàg, hence there is a subsequence of $(Z(\lambda_N(y_N'))_{N\in\mathds{N}^*}$ that converges to either $Z(y_\infty)$ or $Z(y_\infty^-)$. Since $|Z(\lambda_N(y_N'))|\in[\delta_1-3/N,2\delta_1+3/N]$, we have $|Z(y_\infty)|$ or $|Z(y_\infty^-)|$ in $[\delta_1,2\delta_1]$. Therefore $Z\in \Xi$, which ends the proof. 
 \end{proof}
 
 \section{Convergence of the stopping time: proof of Proposition \ref{prop_fluctuations_T}}\label{sec_conv_T}
 
 We want to prove Proposition \ref{prop_fluctuations_T}, that is the convergence in distribution of $\frac{1}{N^{3/2}}(T_N-N^2(|x|+2\theta)^2)$ to the law $\mathcal{N}(0,\frac{32}{3}\mathrm{Var}(\rho_-)((|x|+\theta)^3+\theta^3))$ when $N$ tends to $+\infty$. In order to do that, we will prove that $\frac{1}{N^{3/2}}(T_N-N^2(|x|+2\theta)^2)$ is close to $2\int_{-|x|-2\theta}^{|x|+2\theta}Y_N'(y)\mathrm{d}y$ (where $Y_N'$ was defined at the beginning of Section \ref{subsec_proof_thm}), then show that $2\int_{-|x|-2\theta}^{|x|+2\theta}Y_N'(y)\mathrm{d}y$ converges to the desired distribution. 
 
 \begin{proposition}\label{prop_T_N_conv_int}
  $\mathds{P}(|\frac{1}{N^{3/2}}(T_N-N^2(|x|+2\theta)^2)-2\int_{-|x|-2\theta}^{|x|+2\theta}Y_N'(y)\mathrm{d}y| > 5(|x|+2\theta)N^{-1/12})$ tends to 0 when $N$ tends to $+\infty$. 
 \end{proposition}
 
 \begin{proof}
 The result will come from the fact that $T_N$ can be written as the sum of the local times, which is itself related to the integrals of $Y_N^-$ and $Y_N^+$, which are close to $Y_N$ by Lemma \ref{lem_Y_close} hence to $Y_N'$ by Lemma \ref{lem_Y_close_Y'}. It is enough to prove that if $(\mathcal{B}_2)^c$, $(\mathcal{B}_4^-)^c$ and $(\mathcal{B}_4^+)^c$ occur and if $|I^-+(|x|+2\theta)N| < N^{5/8}$, $|I^+-(|x|+2\theta)N| < N^{5/8}$, then $|\frac{1}{N^{3/2}}(T_N-N^2(|x|+2\theta)^2)-2\int_{-|x|-2\theta}^{|x|+2\theta}Y_N'(y)\mathrm{d}y| \leq 5(|x|+2\theta)N^{-1/12}$, since Lemma \ref{lem_zeta_eta_small} implies $\mathds{P}(\mathcal{B}_2)$ tends to 0 when $N$ tends to $+\infty$, Lemma \ref{lem_sum_eta_zeta} implies $\mathds{P}(\mathcal{B}_4^-)$ and $\mathds{P}(\mathcal{B}_4^+)$ tend to 0 when $N$ tends to $+\infty$, and Lemma \ref{lem_control_I} implies $\mathds{P}(|I^-+(|x|+2\theta)N| \geq N^{5/8})$ and $\mathds{P}(|I^+-(|x|+2\theta)N| \geq N^{5/8})$ tend to 0 when $N$ tends to $+\infty$. We assume $(\mathcal{B}_2)^c$, $(\mathcal{B}_4^-)^c$ and $(\mathcal{B}_4^+)^c$ occur and $|I^-+(|x|+2\theta)N| < N^{5/8}$, $|I^+-(|x|+2\theta)N| < N^{5/8}$, let us study $T_N$. 
 
 In order to do that, we first need to prove an auxiliary result, more precisely that the following holds when $N$ is large enough:
 \begin{equation}\label{eq_small_local_times}
 \text{if }|i-I^+| \leq N^{5/8}+1 \text{ or } |i-I^-| \leq N^{5/8}+1 \text{ then } \ell^+(T_N,i) \leq 4N^{11/16} \text{ and }\ell^-(T_N,i) \leq 4N^{11/16}.
 \end{equation}
 We prove \eqref{eq_small_local_times} for the case $|i-I^-| \leq N^{5/8}+1$, since the other is similar. Let $i \in \mathds{Z}$ so that $|i-I^-| < N^{5/8}+1$. We notice that since $|I^-+(|x|+2\theta)N| < N^{5/8}$ we have $I^-,i < 0$ when $N$ is large enough, so \eqref{eq_local_times} yields $|\ell^+(T_N,i)-\ell^+(T_N,I^-)| \leq \sum_{|j-I^-| < N^{5/8}+1}|\eta_{j,-}(\ell^+(T_N,j))|$, thus since $\ell^+(T_N,I^-)=0$ we have $\ell^+(T_N,i) \leq \sum_{|j-I^-| < N^{5/8}+1}|\eta_{j,-}(\ell^+(T_N,j))|$. In addition, we assumed $(\mathcal{B}_2)^c$, hence $\ell^+(T_N,i) \leq \sum_{|j-I^-| < N^{5/8}+1}(N^{1/16}+1/2) \leq 3 N^{5/8}N^{1/16}=3N^{11/16}$ when $N$ is large enough. Furthermore, \eqref{eq_local_times_2} implies $|\ell^-(T_N,i)-\ell^+(T_N,i)|=|\eta_{i,-}(\ell^+(T_N,i))| \leq N^{1/16}+1/2$ thanks to $(\mathcal{B}_2)^c$, hence $\ell^-(T_N,i) \leq 3N^{11/16}+N^{1/16}+1/2 \leq 4 N^{11/16}$ when $N$ is large enough, which ends the proof of \eqref{eq_small_local_times}.
 
 We now write $T_N$ as the sum of the local times and relate $\frac{1}{N^{3/2}}(T_N-N^2(|x|+2\theta)^2)$ to the integral of $Y^+$ and $Y^-$. We have $T_N = \sum_{i\in\mathds{Z}}(\ell^+(T_N,i)+\ell^-(T_N,i))$. Moreover, Lemma \ref{lem_0_after_I} implies that for all $i \geq I^+$ and $i \leq I^-$ we have $\ell^+(T_N,i)=\ell^-(T_N,i)=0$. Consequently, $T_N = \sum_{i=I^- \wedge (-\lfloor (|x|+2\theta)N\rfloor)}^{I^+ \vee \lfloor (|x|+2\theta)N\rfloor}(\ell^+(T_N,i)+\ell^-(T_N,i))$. We thus have $|\frac{1}{N^{3/2}}(T_N-N^2(|x|+2\theta)^2)-\int_{(I^-\wedge (-(|x|+2\theta)N))/N}^{(I^+\vee (|x|+2\theta)N)/N}(Y_N^+(y)+Y_N^-(y))\mathrm{d}y| \leq \frac{1}{N^{3/2}}(\ell^+(T_N,I^+ \vee \lfloor (|x|+2\theta)N\rfloor)+\ell^-(T_N,I^+ \vee \lfloor (|x|+2\theta)N\rfloor)+\ell^+(T_N,-\lfloor (|x|+2\theta)N\rfloor-1)+\ell^-(T_N,-\lfloor (|x|+2\theta)N\rfloor-1))$. Since we assumed $|I^-+(|x|+2\theta)N| < N^{5/8}$ and $|I^+-(|x|+2\theta)N| < N^{5/8}$, equation \eqref{eq_small_local_times} yields 
 \begin{equation}\label{eq_T_close_int}
 \left|\frac{1}{N^{3/2}}(T_N-N^2(|x|+2\theta)^2)-\int_{(I^-\wedge (-(|x|+2\theta)N))/N}^{(I^+\vee (|x|+2\theta)N)/N}(Y_N^+(y)+Y_N^-(y))\mathrm{d}y\right| \leq \frac{1}{N^{3/2}}16 N^{11/16}=16N^{-13/16}.
 \end{equation}
 
 We now prove that $\int_{(I^-\wedge (-(|x|+2\theta)N))/N}^{(I^+\vee (|x|+2\theta)N)/N}(Y_N^+(y)+Y_N^-(y))\mathrm{d}y$ is close to $2\int_{-(|x|+2\theta)}^{|x|+2\theta}Y_N(y)\mathrm{d}y$. We begin by considering $\int_{\chi(N)/N}^{(I^+\vee (|x|+2\theta)N)/N}(Y_N^+(y)+Y_N^-(y))\mathrm{d}y$. We first assume $I^+ \geq (|x|+2\theta)N$. Since we assumed $(\mathcal{B}_2)^c$, $(\mathcal{B}_4^-)^c$ and $(\mathcal{B}_4^+)^c$ occur, Lemma \ref{lem_Y_close} yields  $|\int_{\chi(N)/N}^{(I^+\vee (|x|+2\theta)N)/N}(Y_N^+(y)+Y_N^-(y))\mathrm{d}y-2\int_{\chi(N)/N}^{|x|+2\theta}Y_N(y)\mathrm{d}y| \leq 2(|x|+2\theta-\frac{\chi(N)}{N})N^{-1/12}+\int_{|x|+2\theta}^{I^+/N}|Y_N^+(y)+Y_N^-(y))|\mathrm{d}y$. In addition, we know $I^+ - (|x|+2\theta)N \leq N^{5/8}$ and \eqref{eq_small_local_times}, hence 
 \[
 \int_{|x|+2\theta}^{I^+/N}|Y_N^+(y)+Y_N^-(y))|\mathrm{d}y \leq N^{-3/8}\frac{1}{\sqrt{N}}\left(\max_{\lfloor (|x|+2\theta)N\rfloor\leq i \leq I^+}\ell^+(T_N,i)+\max_{\lfloor (|x|+2\theta)N\rfloor\leq i \leq I^+}\ell^-(T_N,i)\right)
 \]
 \[
 \leq N^{-3/8}N^{-1/2}8N^{11/16} = 8N^{-3/16}.
 \]
 We deduce 
 \[
 \left|\int_{\chi(N)/N}^{(I^+\vee (|x|+2\theta)N)/N}(Y_N^+(y)+Y_N^-(y))\mathrm{d}y-2\int_{\chi(N)/N}^{|x|+2\theta}Y_N(y)\mathrm{d}y\right| \leq 2\left(|x|+2\theta-\frac{\chi(N)}{N}\right)N^{-1/12}+8N^{-3/16}.
 \]
 We now assume $I^+ < (|x|+2\theta)N$. In this case, we have 
 \[
 \left|\int_{\chi(N)/N}^{(I^+\vee (|x|+2\theta)N)/N}(Y_N^+(y)+Y_N^-(y))\mathrm{d}y-2\int_{\chi(N)/N}^{|x|+2\theta}Y_N(y)\mathrm{d}y\right| 
 \]
 \[
 \leq \int_{\chi(N)/N}^{I^+/N}|Y_N^+(y)+Y_N^-(y)-2Y_N(y)|\mathrm{d}y+\int_{I^+/N}^{|x|+2\theta}|Y_N^+(y)+Y_N^-(y)|\mathrm{d}y+\int_{I^+/N}^{|x|+2\theta}|2Y_N(y)|\mathrm{d}y.
 \]
 Moreover, Lemma \ref{lem_Y_close} yields $\int_{\chi(N)/N}^{I^+/N}|Y_N^+(y)+Y_N^-(y)-2Y_N(y)|\mathrm{d}y \leq 2(|x|+2\theta-\frac{\chi(N)}{N})N^{-1/12}$. Furthermore, for $y \geq \frac{I^+}{N}$ we have $\ell^\pm(T_N,\lfloor Ny\rfloor)=0$. Since $|I^+ - (|x|+2\theta)N| < N^{5/8}$ this yields $|Y_N^\pm(y)| \leq \frac{1}{2}N^{1/8}$. Thus $\int_{I^+/N}^{|x|+2\theta}|Y_N^+(y)+Y_N^-(y)|\mathrm{d}y \leq \int_{I^+/N}^{|x|+2\theta}N^{1/8}\mathrm{d}y \leq N^{-3/8}N^{1/8}=N^{-1/4}$.  We deduce 
 \[
 \left|\int_{\chi(N)/N}^{(I^+\vee (|x|+2\theta)N)/N}(Y_N^+(y)+Y_N^-(y))\mathrm{d}y-2\int_{\chi(N)/N}^{|x|+2\theta}Y_N(y)\mathrm{d}y\right| 
 \]
 \[
 \leq 2\left(|x|+2\theta-\frac{\chi(N)}{N}\right)N^{-1/12}+N^{-1/4}+\int_{I^+/N}^{|x|+2\theta}|2Y_N(y)|\mathrm{d}y.
 \]
 In addition, for any $y\in[\frac{I^+}{N},|x|+2\theta]$, we have $|Y_N(y)| \leq |Y_N(y)-Y_N(\frac{I^+}{N})|+|Y_N(\frac{I^+}{N})-Y_N^-(\frac{I^+}{N})|+|Y_N^-(\frac{I^+}{N})|$. Lemma \ref{lem_Y_close} yields that $|Y_N(\frac{I^+}{N})-Y_N^-(\frac{I^+}{N})| \leq N^{-1/12}$, and since $|I^+ - (|x|+2\theta)N| < N^{5/8}$ we have $|Y_N^-(\frac{I^+}{N})|=|\frac{1}{\sqrt{N}}(\ell^\pm(T_N,I^+)-N(\frac{|x|-|I^+/N|}{2}+\theta)_+)|
\leq \frac{1}{2}N^{1/8}$, hence 
\[
|Y_N(y)| \leq \left|Y_N(y)-Y_N\left(\frac{I^+}{N}\right)\right|+N^{-1/12}+\frac{1}{2}N^{1/8}=\frac{1}{\sqrt{N}}\left|\sum_{i=I^+}^{\lfloor Ny\rfloor-1}\zeta_i\right|+N^{-1/12}+\frac{1}{2}N^{1/8} 
\]
\[
\leq \frac{1}{\sqrt{N}}\sum_{i=I^+}^{\lfloor (|x|+2\theta)N\rfloor-1}|\zeta_i|+N^{-1/12}+\frac{1}{2}N^{1/8} \leq \frac{1}{\sqrt{N}}N^{5/8}N^{1/16}+N^{-1/12}+\frac{1}{2}N^{1/8} \leq 2N^{3/16}
\]
since $(\mathcal{B}_2)^c$ occurs. This implies $\int_{I^+/N}^{|x|+2\theta}|2Y_N(y)|\mathrm{d}y \leq \int_{I^+/N}^{|x|+2\theta}4N^{3/16}\mathrm{d}y=N^{-3/8}4N^{3/16}=4N^{-3/16}$. We deduce $|\int_{\chi(N)/N}^{(I^+\vee (|x|+2\theta)N)/N}(Y_N^+(y)+Y_N^-(y))\mathrm{d}y-2\int_{\chi(N)/N}^{|x|+2\theta}Y_N(y)\mathrm{d}y| \leq 2(|x|+2\theta-\frac{\chi(N)}{N})N^{-1/12}+5N^{-3/16}$. Consequently, in all cases we have $|\int_{\chi(N)/N}^{(I^+\vee (|x|+2\theta)N)/N}(Y_N^+(y)+Y_N^-(y))\mathrm{d}y-2\int_{\chi(N)/N}^{|x|+2\theta}Y_N(y)\mathrm{d}y| \leq 2(|x|+2\theta-\frac{\chi(N)}{N})N^{-1/12}+8N^{-3/16}$. One can prove similarly that $|\int_{(I^-\wedge (-(|x|+2\theta)N))/N}^{\chi(N)/N}(Y_N^+(y)+Y_N^-(y))\mathrm{d}y-2\int_{-(|x|+2\theta)}^{\chi(N)/N}Y_N(y)\mathrm{d}y| \leq 2(|x|+2\theta+\frac{\chi(N)}{N})N^{-1/12}+8N^{-3/16}$. We conclude that $|\int_{(I^-\wedge (-(|x|+2\theta)N))/N}^{(I^+\vee (|x|+2\theta)N)/N}(Y_N^+(y)+Y_N^-(y))\mathrm{d}y-2\int_{-(|x|+2\theta)}^{|x|+2\theta}Y_N(y)\mathrm{d}y| \leq 4(|x|+2\theta)N^{-1/12}+16N^{-3/16}$. 
 
We are now in position to conclude. Indeed, the previous result and \eqref{eq_T_close_int} imply that when $N$ is large enough, $|\frac{1}{N^{3/2}}(T_N-N^2(|x|+2\theta)^2)- 2\int_{-(|x|+2\theta)}^{|x|+2\theta}Y_N(y)\mathrm{d}y| \leq 16N^{-13/16}+4(|x|+2\theta)N^{-1/12}+16N^{-3/16}$. Moreover, $(\mathcal{B}_2)^c$ occurs, hence Lemma \ref{lem_Y_close_Y'} yields $\sup\{|Y_N(y)-Y_N'(y)|:y\in[-|x|-2\theta,|x|+2\theta)\} \leq N^{-7/16}$, therefore $|\int_{-(|x|+2\theta)}^{|x|+2\theta}Y_N(y)\mathrm{d}y-\int_{-(|x|+2\theta)}^{|x|+2\theta}Y_N'(y)\mathrm{d}y| \leq 2(|x|+2\theta)N^{-7/16}$. We deduce that when $N$ is large enough, $|\frac{1}{N^{3/2}}(T_N-N^2(|x|+2\theta)^2)- 2\int_{-(|x|+2\theta)}^{|x|+2\theta}Y_N'(y)\mathrm{d}y| \leq 16N^{-13/16}+4(|x|+2\theta)N^{-1/12}+16N^{-3/16}+4(|x|+2\theta)N^{-7/16} \leq 5(|x|+2\theta)N^{-1/12}$, which ends the proof.
 \end{proof}
 
 Now that we know $\frac{1}{N^{3/2}}(T_N-N^2(|x|+2\theta)^2)$ is close to $2\int_{-|x|-2\theta}^{|x|+2\theta}Y_N'(y)\mathrm{d}y$, we need to prove $2\int_{-|x|-2\theta}^{|x|+2\theta}Y_N'(y)\mathrm{d}y$ converges to the desired distribution. In order to do that, we will use the convergence of $Y_N'$ to a Brownian motion stated in Lemma \ref{lem_conv_Y'_C}, so $2\int_{-|x|-2\theta}^{|x|+2\theta}Y_N'(y)\mathrm{d}y$ will converge to the integral of a Brownian motion, the law of the latter being characterized by the following lemma, where we denote by $(B_t)_{t\in\mathds{R}^+}$ a standard Brownian motion with $B_0=0$. This lemma is quite standard (the interested reader can find a proof in the first arXiv version of this paper \cite{Mareche2022v1}).

 \begin{lemma}\label{lem_int_BM}
  For any $y>0$, the integral $\int_0^y B_{z} \mathrm{d}z$ has distribution $\mathcal{N}(0,\frac{y^3}{3})$.
 \end{lemma}
 
 We are now able to prove Proposition \ref{prop_fluctuations_T}. 
 
 \begin{proof}[Proof of Proposition \ref{prop_fluctuations_T}.]
 Proposition \ref{prop_T_N_conv_int} implies $\frac{1}{N^{3/2}}(T_N-N^2(|x|+2\theta)^2)-2\int_{-|x|-2\theta}^{|x|+2\theta}Y_N'(y)\mathrm{d}y$ converges in probability to 0 when $N$ tends to $+\infty$. Hence by Slutsky's Theorem, it is enough to prove $2\int_{-|x|-2\theta}^{|x|+2\theta}Y_N'(y)\mathrm{d}y$ converges in distribution to $\mathcal{N}(0,\mathrm{Var}(\rho_-)\frac{32}{3}((|x|+\theta)^3+\theta^3))$ when $N$ tends to $+\infty$ to prove Proposition \ref{prop_fluctuations_T}. In addition, by Lemma \ref{lem_conv_Y'_C}, $Y_N'|_{[-|x|-2\theta,|x|+2\theta]}$ converges in distribution to $B^x|_{[-|x|-2\theta,|x|+2\theta]}$ when $N$ tends to $+\infty$ for the topology defined on $C[-|x|-2\theta,|x|+2\theta]$ by the uniform norm. Moreover, the integral between $-|x|-2\theta$ and $|x|+2\theta$ is continuous for this topology, hence $\int_{-|x|-2\theta}^{|x|+2\theta}Y_N'(y)\mathrm{d}y$ converges in distribution to $\int_{-|x|-2\theta}^{|x|+2\theta}B_y^x\mathrm{d}y$ when $N$ tends to $+\infty$. Furthermore, $B^x$ is a two-sided Brownian motion with $B_x^x=0$ and variance $\mathrm{Var}(\rho_-)$, hence we can write $\int_{-|x|-2\theta}^{|x|+2\theta}B_y^x\mathrm{d}y=\int_{-|x|-2\theta}^{x}B_y^x\mathrm{d}y+\int_{x}^{|x|+2\theta}B_y^x\mathrm{d}y$ where $\int_{-|x|-2\theta}^{x}B_y^x\mathrm{d}y$ and $\int_{x}^{|x|+2\theta}B_y^x\mathrm{d}y$ are independent. In addition, $\int_{x}^{|x|+2\theta}B_y^x\mathrm{d}y$ has the distribution of $\sqrt{\mathrm{Var}(\rho_-)}\int_{0}^{2\theta}B_y\mathrm{d}y$, which is $\mathcal{N}(0,\mathrm{Var}(\rho_-)\frac{(2\theta)^3}{3})$ by Lemma \ref{lem_int_BM}, and $\int_{-|x|-2\theta}^x B_y^x\mathrm{d}y$ has the distribution of $\sqrt{\mathrm{Var}(\rho_-)}\int_{0}^{2|x|+2\theta}B_y\mathrm{d}y$, which is $\mathcal{N}(0,\mathrm{Var}(\rho_-)\frac{(2|x|+2\theta)^3}{3})$ by Lemma \ref{lem_int_BM}. We obtain that $\int_{-|x|-2\theta}^{|x|+2\theta}B_y^x\mathrm{d}y$ has the distribution $\mathcal{N}(0,\mathrm{Var}(\rho_-)\frac{(2|x|+2\theta)^3}{3}+\mathrm{Var}(\rho_-)\frac{(2\theta)^3}{3})=\mathcal{N}(0,\mathrm{Var}(\rho_-)\frac{8}{3}((|x|+\theta)^3+\theta^3))$. Consequently, $\int_{-|x|-2\theta}^{|x|+2\theta}Y_N'(y)\mathrm{d}y$ converges in distribution to $\mathcal{N}(0,\mathrm{Var}(\rho_-)\frac{8}{3}((|x|+\theta)^3+\theta^3))$ when $N$ tends to $+\infty$, which ends the proof of Proposition \ref{prop_fluctuations_T}.
 \end{proof}
 
\section*{Acknowledgements}

The author was supported by the University of Strasbourg Initiative of Excellence. She wishes to thank Thomas Mountford for introducing her to this random walk and pointing her to some references.

\end{document}